\documentclass[11pt,a4paper]{amsart}
\usepackage[utf8]{inputenc}

\usepackage[mono=false,osf]{libertine}
\usepackage{libertinust1math}
\usepackage[T1]{fontenc}
\usepackage[scaled=0.88]{beramono}
\renewcommand{\mathsf}[1]{\text{\normalfont\sffamily#1}}
\usepackage[cal=euler,scr=boondoxo,bb=pazo]{mathalfa}

\usepackage[OkabeIto]{colorblind}
\usepackage{xcolor-material}
\usepackage[unicode, psdextra, colorlinks=true, linkcolor=GoogleBlue, citecolor=GoogleGreen, urlcolor=GoogleBlue, linktocpage]{hyperref}

\usepackage{amsmath}
\usepackage{amsfonts,amssymb,amsthm,amscd}

\usepackage{trimclip}

\usepackage{setspace}

\usepackage{tikz}
\usetikzlibrary{cd}
\usepackage{mathrsfs}
\usepackage{multirow}
\usepackage{stmaryrd}
\SetSymbolFont{stmry}{bold}{U}{stmry}{m}{n} 
\usepackage{bm,bbm}
\usepackage{mathtools}
\usepackage{etoolbox}
\usepackage{enumitem}

\usepackage[margin=3cm]{geometry}

\tikzcdset{scale cd/.style={every label/.append style={scale=#1},
    cells={nodes={scale=#1}}}}

\makeatletter
\def\l@subsection{\@tocline{2}{0pt}{2.5pc}{5pc}{}}
\makeatother

\makeatletter
\newcommand\RedeclareMathOperator{%
  \@ifstar{\def\rmo@s{m}\rmo@redeclare}{\def\rmo@s{o}\rmo@redeclare}%
}
\newcommand\rmo@redeclare[2]{%
  \begingroup \escapechar\m@ne\xdef\@gtempa{{\string#1}}\endgroup
  \expandafter\@ifundefined\@gtempa
     {\@latex@error{\noexpand#1undefined}\@ehc}%
     \relax
  \expandafter\rmo@declmathop\rmo@s{#1}{#2}}
\newcommand\rmo@declmathop[3]{%
  \DeclareRobustCommand{#2}{\qopname\newmcodes@#1{#3}}%
}
\@onlypreamble\RedeclareMathOperator
\makeatother

\newcommand{\dmerge}[4]{
	\draw (#1,#2) .. controls (#1,#4*0.5+#2*0.5) and (#3*0.5+#1*0.5,#4*0.5+#2*0.5) .. (#3*0.5+#1*0.5,#4);
	\draw (#3,#2) .. controls (#3,#4*0.5+#2*0.5) and (#3*0.5+#1*0.5,#4*0.5+#2*0.5) .. (#3*0.5+#1*0.5,#4);
}
\newcommand{\dmergedot}[5][OI5]{
	\draw [dotted,line cap=round,rounded corners=2pt,color=#1] (#2,#3) .. controls (#2,#5*0.5+#3*0.5) and (#4*0.5+#2*0.5,#5*0.5+#3*0.5) .. (#4*0.5+#2*0.5,#5);
	\draw [dotted,line cap=round,rounded corners=2pt,color=#1] (#4,#3) .. controls (#4,#5*0.5+#3*0.5) and (#4*0.5+#2*0.5,#5*0.5+#3*0.5) .. (#4*0.5+#2*0.5,#5);
}

\newcommand{\opbox}[5]{
	\draw (#1,#2) rectangle (#3,#4);
	\node[] at (#3*0.5+#1*0.5,#4*0.5+#2*0.5) {#5};
}
\newcommand{\ntxt}[3]{
	\node[text height=1.2ex,text depth=.25ex] at (#1,#2) {#3};
}
\newcommand{\crosin}[4]{
	\draw (#1,#2) .. controls (#1,#4*0.6+#2*0.4) and (#3,#4*0.4+#2*0.6) .. (#3,#4);
	\draw (#3,#2) .. controls (#3,#4*0.6+#2*0.4) and (#1,#4*0.4+#2*0.6) .. (#1,#4);
}
\newcommand{\overcros}[4]{
	\draw (#3,#2) .. controls (#3,#4*0.6+#2*0.4) and (#1,#4*0.4+#2*0.6) .. (#1,#4);
	\draw[line width=3pt, white] (#1,#2) .. controls (#1,#4*0.6+#2*0.4) and (#3,#4*0.4+#2*0.6) .. (#3,#4);
	\draw (#1,#2) .. controls (#1,#4*0.6+#2*0.4) and (#3,#4*0.4+#2*0.6) .. (#3,#4);
}
\newcommand{\overcrosdash}[4]{
	\draw [dotted,line cap=round,rounded corners=2pt](#3,#2) .. controls (#3,#4*0.6+#2*0.4) and (#1,#4*0.4+#2*0.6) .. (#1,#4);
	\draw[line width=5pt, white] (#1,#2) .. controls (#1,#4*0.6+#2*0.4) and (#3,#4*0.4+#2*0.6) .. (#3,#4);
	\draw [dotted,line cap=round,rounded corners=2pt](#1,#2) .. controls (#1,#4*0.6+#2*0.4) and (#3,#4*0.4+#2*0.6) .. (#3,#4);
}
\newcommand{\undercros}[4]{
	\draw (#1,#2) .. controls (#1,#4*0.6+#2*0.4) and (#3,#4*0.4+#2*0.6) .. (#3,#4);
	\draw[line width=3pt, white] (#3,#2) .. controls (#3,#4*0.6+#2*0.4) and (#1,#4*0.4+#2*0.6) .. (#1,#4);
	\draw (#3,#2) .. controls (#3,#4*0.6+#2*0.4) and (#1,#4*0.4+#2*0.6) .. (#1,#4);
}

\newcommand{\identity}[3]{
	\draw[white] (#1,#2) -- (#1,#3);
	\draw (#1,#2*0.7+#3*0.3) -- (#1,#3);
	\draw (#1+0.2,#2*0.7+#3*0.3+0.2) arc[start angle=0, end angle=-180, radius=0.2];
}

\newcommand{\strdot}[2]{
	\fill (#1,#2) circle (4pt);
}
\newcommand{\strdotwhite}[2]{
	\draw[fill=white] (#1,#2) circle (4pt);
}
\newcommand{\curline}[4]{
	\draw (#1,#2) .. controls (#1,#4*0.6+#2*0.4) and (#3,#4*0.4+#2*0.6) .. (#3,#4);
}
\newcommand{\curlinedot}[4]{
	\draw[dotted,line cap=round,rounded corners]  (#1,#2) .. controls (#1,#4*0.6+#2*0.4) and (#3,#4*0.4+#2*0.6) .. (#3,#4);
}

\newcommand{\sigact}[5][black]{
	\draw[dotted,line cap=round,rounded corners,color=#1] (#2,#3) .. controls (#2,#5*0.6+#3*0.4) and (#4,#5*0.4+#3*0.6) .. (#4,#5);
	\draw (#4,#3) .. controls (#4,#5*0.6+#3*0.4) and (#2,#5*0.4+#3*0.6) .. (#2,#5);
	\draw[fill=white] (#2*0.5+#4*0.5,#5*0.5+#3*0.5) circle (4pt);
}

\newcommand{\rhoact}[5][black]{
	\draw[dotted,line cap=round,rounded corners=2pt,color=#1] (#2,#3) .. controls (#2,#5*0.4+#3*0.6) and (#4*0.5+#2*0.5,#5*0.4+#3*0.6) .. (#4*0.5+#2*0.5,#5*0.8+#3*0.2);
	\draw (#4,#3) .. controls (#4,#5*0.4+#3*0.6) and (#4*0.5+#2*0.5,#5*0.4+#3*0.6) .. (#4*0.5+#2*0.5,#5*0.8+#3*0.2);
	\draw (#4*0.5+#2*0.5,#5*0.7+#3*0.3) -- (#4*0.5+#2*0.5,#5);
	\fill (#4*0.5+#2*0.5,#5*0.7+#3*0.3) circle (4pt);
}

\newcommand{\morTi}[4]{
	\draw[dotted,line cap=round,rounded corners=2pt,color=OI5] (#1,#2) -- (#1,#2*0.75+#4*0.25) -- (#1*0.5+#3*0.5,#2*0.5+#4*0.5);
	\draw[dotted,line cap=round,rounded corners=2pt,color=OI5] (#3,#2) -- (#3,#2*0.75+#4*0.25) -- (#1*0.5+#3*0.5,#2*0.5+#4*0.5);
	\draw[dotted,line cap=round,color=OI5] (#1*0.5+#3*0.5,#4) -- (#1*0.5+#3*0.5,#2*0.5+#4*0.5);
	\draw[dotted,line cap=round,color=OI1] (#1*0.5+#3*0.5,#2) -- (#1*0.5+#3*0.5,#2*0.5+#4*0.5);
	\draw[dotted,line cap=round,rounded corners=2pt,color=OI1] (#1,#4) -- (#1,#2*0.25+#4*0.75) -- (#1*0.5+#3*0.5,#2*0.5+#4*0.5);
	\draw[dotted,line cap=round,rounded corners=2pt,color=OI1] (#3,#4) -- (#3,#2*0.25+#4*0.75) -- (#1*0.5+#3*0.5,#2*0.5+#4*0.5);
}

\newcommand{\morSi}[5][OI5]{
	\draw[dotted,line cap=round,rounded corners=2pt,color=#1] (#2,#3) -- (#2,#3*0.75+#5*0.25) -- (#2*0.5+#4*0.5,#3*0.5+#5*0.5);
	\draw[dotted,line cap=round,rounded corners=2pt,color=#1] (#4,#3) -- (#4,#3*0.75+#5*0.25) -- (#2*0.5+#4*0.5,#3*0.5+#5*0.5);
	\draw[rounded corners=2pt] (#2,#5) -- (#2,#3*0.25+#5*0.75) -- (#2*0.5+#4*0.5,#3*0.5+#5*0.5);
	\draw[dotted,line cap=round,rounded corners=2pt,color=#1] (#4,#5) -- (#4,#3*0.25+#5*0.75) -- (#2*0.5+#4*0.5,#3*0.5+#5*0.5);
	\draw[fill=white] (#2*0.5+#4*0.5-0.15,#5*0.5+#3*0.5-0.15) rectangle (#2*0.5+#4*0.5+0.15,#5*0.5+#3*0.5+0.15);
	\draw (#2*0.5+#4*0.5-0.15,#5*0.5+#3*0.5-0.15) -- (#2*0.5+#4*0.5+0.15,#5*0.5+#3*0.5+0.15);
	\draw (#2*0.5+#4*0.5+0.15,#5*0.5+#3*0.5-0.15) -- (#2*0.5+#4*0.5-0.15,#5*0.5+#3*0.5+0.15);
}

\newcommand{\morRi}[5][OI5]{
	\draw[dotted,line cap=round,rounded corners=2pt,color=#1] (#2,#3) -- (#2,#3*0.75+#5*0.25) -- (#2*0.5+#4*0.5,#3*0.5+#5*0.5);
	\draw[dotted,line cap=round,rounded corners=2pt,color=#1] (#4,#3) -- (#4,#3*0.75+#5*0.25) -- (#2*0.5+#4*0.5,#3*0.5+#5*0.5);
	\draw[rounded corners=2pt] (#2*0.5+#4*0.5,#5) -- (#2*0.5+#4*0.5,#3*0.5+#5*0.5);
	\draw[fill=white] (#2*0.5+#4*0.5-0.15,#5*0.5+#3*0.5-0.15) rectangle (#2*0.5+#4*0.5+0.15,#5*0.5+#3*0.5+0.15);
	\draw (#2*0.5+#4*0.5,#5*0.5+#3*0.5-0.15) -- (#2*0.5+#4*0.5,#5*0.5+#3*0.5+0.15);
}

\usepackage[capitalise]{cleveref}

\theoremstyle{plain}
\newtheorem{thm}{Theorem}[section]
\newtheorem*{thm*}{Theorem}

\newtheorem{prop}[thm]{Proposition}
\newtheorem{lem}[thm]{Lemma}
\newtheorem{cor}[thm]{Corollary}

\theoremstyle{definition}
\newtheorem{defn}[thm]{Definition}
\newtheorem*{defn*}{Definition}

\theoremstyle{remark}
\newtheorem{rmk}[thm]{Remark}
\newtheorem{expl}[thm]{Example}

\AddToHook{env/prop/begin}{\crefalias{thm}{prop}}
\AddToHook{env/lem/begin}{\crefalias{thm}{lem}}
\AddToHook{env/cor/begin}{\crefalias{thm}{cor}}
\AddToHook{env/conj/begin}{\crefalias{thm}{conj}}
\AddToHook{env/defn/begin}{\crefalias{thm}{defn}}
\AddToHook{env/notation/begin}{\crefalias{thm}{notation}}
\AddToHook{env/rmk/begin}{\crefalias{thm}{rmk}}
\AddToHook{env/expl/begin}{\crefalias{thm}{expl}}

\Crefname{thm}{Theorem}{Theorems}
\Crefname{thmintro}{Theorem}{Theorems}
\Crefname{lem}{Lemma}{Lemmata}
\Crefname{prop}{Proposition}{Propositions}
\Crefname{cor}{Corollary}{Corollaries}
\Crefname{conj}{Conjecture}{Conjectures}
\Crefname{defn}{Definition}{Definitions}
\Crefname{notation}{Notation}{Notations}
\Crefname{rmk}{Remark}{Remarks}

\numberwithin{equation}{section}

\usepackage{mysymbols}
\newcommand{\Vn}{\mathbb{H}_n}
\allowdisplaybreaks

\title{Quantum wreath products and Yang--Baxter equations}

\author{Alexandre Minets}
\address{Mathematisches Institut, University of Bonn, Endenicher Allee 60, 53115 Bonn, Germany}
\email{aminets@math.uni-bonn.de}

\begin{document}

\begin{abstract}
	We extend the setup of Lai--Nakano--Xiang's quantum wreath product algebras to symmetric monoidal categories and deformed braid relations.
	Under some mild assumptions, we drastically simplify the criterion for such algebras to have PBW basis; curiously, the new criterion is phrased in terms of associative and quantum Yang--Baxter equations.
	As an application, we clean up some basis results in the existing literature, and propose a certain generalization of quiver Hecke algebras.
\end{abstract}
\date{\today}
\maketitle

\section{Introduction}
\subsection{Quantum wreath products}
The Iwahori-Hecke algebra $\cal{H}_q = \cal{H}_q(\fkS_n)$ of type $A$ is a deformation of the group algebra $\bbk[\fkS_n]$.
Moving from finite to affine type $A$, we get to the affine Hecke algebra $\widehat{\cal{H}}_q$.
It is a classic result of Bernstein that $\widehat{\cal{H}}_q$ can be realized as a certain deformation of the wreath product algebra $\bbk[x^{\pm 1}]\wr \fkS_n$.
In the last 15 years, there arised a number of similar, but distinct deformations of wreath products $F\wr \fkS_n$, where $F$ is an algebra different from (Laurent) polynomials.
To list a few examples:
\begin{itemize}[leftmargin=15pt]
\item $F = Z[x]$, $Z$ the zigzag algebra of a Dynkin quiver: affine zigzag algebras~\cite{KM_SKAA2017};
\item $F = \bbk[\bbZ/(p-1)]\otimes \bbk[x^{\pm 1}]$: pro-$p$ Iwahori-Hecke algebras~\cite{Vigneras_2016,chlouveraki2016affine}; 
\item $F = \mathsf{Cl}\ltimes \bbk[x^{\pm 1}]$, $\mathsf{Cl}$ the Clifford algebra: affine Sergeev algebras~\cite{jones1999affine};
\item $F = \bbk[x]$, where $x$ is an \textit{odd} variable: odd nilHecke algebras~\cite{ellis2014odd};
\item $F = \bigoplus_{i\in I} \bbk[x]$, $I$ finite set: quiver Hecke algebras~\cite{KL_DACQ2009}.
\end{itemize}
There has been a number of attempts to put all these examples under one roof.
For example, Savage~\cite{Sav_AWPA2018} and Rosso--Savage~\cite{rosso2020quantum} introduced affine, resp. quantum affine wreath product algebras, which roughly speaking resemble degenerate affine, resp. affine Hecke algebras.
These two constructions can be treated uniformly~\cite{lai2025schurification,MS}.
A much more general setup was studied in~\cite{lai2024quantum}.
There, the starting point is an arbitrary deformation of wreath and quadratic relations:
\[
	H f = \sigma(f)H + \rho(f), \quad f\in F^{\otimes 2};  \qquad H^2 = SH + R, 
\]
where $\sigma,\rho\in \End_\bbk(F^{\otimes 2})$, and $S,R\in F^{\otimes 2}$.
Of course, there is no guarantee that the resulting algebra, called \textit{quantum wreath product} or QWP and denoted $F\wr \cal{H}_n$, has the same size as $F^{\otimes n}\otimes \bbk[\fkS_n]$, or more precisely a \textit{PBW basis} of the form $\{\mathbf{f}H_w : \mathbf{f}\in F^{\otimes n}, w\in \fkS_n\}$.
The main theorem of \textit{loc.~cit.}, which can be regarded as a classification result, gives an explicit list of necessary and sufficient conditions on the quadruple $(\sigma,\rho,S,R)$, such that $F\wr \cal{H}_n$ has PBW basis.
However, there are some sticking points:
\begin{itemize}[leftmargin=15pt]
	\item The braid relations are not deformed, so~\cite{lai2024quantum} does not cover e.g. quiver Hecke algebras;
	\item The list of relations is rather long (see e.g. \cref{ssec:alg-form}), and difficult to check in practice.
\end{itemize}

\subsection{Main result}
One of the first conditions that~\cite{lai2024quantum} tell us to check is that $\sigma$ is an algebra homomorphism, and $\rho$ is a left $\sigma$-derivation.
The present paper began as an observation that in all cases of interest, the skew derivation $\rho$ is \textit{inner}, i.e. of the form $\rho_r(x) = rx-\sigma(x)r$, where $r$ is an element of a localization of $F^{\otimes 2}$.
Building on this, we restrict our attention to a rich class of QWP-like algebras with deformed braid relation:
\begin{defn*}[\cref{defn:TIC}]
	Let $F$ be a unital algebra, and $\phi,\psi\in \Aut(F)$ algebra automorphisms.
	Assume that there exists a central element $c\in F$, such that $\phi(c) = \psi(c) = c$, and $c\otimes 1-1\otimes c$ is not a zero divisor.
	Fix $S,R,r\in F^{\otimes 2}$, $D,\overline{D},E\in F^{\otimes 3}$, and assume $(\phi\otimes \phi)(r) = (\psi\otimes \psi)(r) = r$.
	A \textit{TIC algebra}\footnote{The abbreviation TIC stands for ``\textbf{T}ressée, \textbf{I}ntérieure, en \textbf{C}ouronne''.
	It was chosen as a nod to RAFH algebras~\cite{MS}, as coming up with good axiomatics for QWP algebras is rough enough to develop a tic.} $F\wr \cal{H}_n$ is the algebra generated by $H_1,\ldots,H_{n-1}$ and a subalgebra isomorphic to $F^{\otimes n}$, modulo the relations
	\begin{gather*}
		H_i^2 = S_iH_i + R_i;\qquad H_i\mathbf{f} = \sigma_i(\mathbf{f})H_i + \rho_i(\mathbf{f}),\quad\mathbf{f}\in F^{\otimes n};\\
		H_{i+1} H_i H_{i+1} = H_i H_{i+1} H_i + D_i H_i + \overline{D}_iH_{i+1} + E_i;  \qquad H_j H_i = H_i H_j, \quad |i-j|>1,
	\end{gather*}
	where $\sigma(f_1\otimes f_2) \coloneqq \phi(f_2)\otimes \psi(f_1)$, and $\rho = \rho_r$ the inner $\sigma$-derivation.
\end{defn*}

While the notion of TIC algebra might seem a bit synthetic, it is general enough to cover all natural examples the author managed to come up with.
Our main result is that for TIC algebras, the equivalent conditions for the existence of PBW basis are dramatically simplified:
\begin{thm*}[\cref{thm:deformed-TIC-PBW},~\cref{cor:PBW-in-loc}]
	Let $F\wr \cal{H}_n$ be a TIC algebra, with deformation parameters $(S,R,D,\overline{D},E)$.
	Denote $\gamma\coloneqq R - \sigma(r)r$.
	If $F\wr \cal{H}_n$ has PBW basis, then $\gamma$ satisfies
	\begin{equation}\label{eq:gamma-intro}
		\sigma(\gamma) = \gamma, \qquad \gamma f = \sigma^2(f)\gamma \quad\text{for all }f\in F^{\otimes 2}.
	\end{equation}
	Conversely, for any $R\in F^{\otimes 2}$ such that $\gamma$ satisfies the conditions above, there exists the unique choice of $(S,D,\overline{D},E)$ such that $F\wr \cal{H}_n$ has PBW basis.
	Explicitly, these elements are given by
	\begin{gather*}
	S = r_{12} + r_{21}, \quad D = r_{23}r_{13} - r_{12}r_{23} - r_{13}r_{21}, \quad \overline{D} = r_{23}r_{12} + r_{13}r_{32} - r_{12}r_{13},\\
	E = r_{12}r_{13}r_{23} - r_{23}r_{13}r_{12} + r_{13}(R_{23}-R_{12}).
\end{gather*}
More generally, if $r$ belongs to a localization of $F^{\otimes 2}$, then $F\wr \cal{H}_n$ has PBW basis if and only if $\gamma$ satisfies~\eqref{eq:gamma-intro}, $(S,D,\overline{D},E)$ are given by the formulas above, and moreover $S \in F^{\otimes 2}$, $D,\overline{D},E \in F^{\otimes 3}$.
\end{thm*}

Why is this a useful result?
For a general Hecke-like algebra with deformed braid relation, the hardest condition to check in order to ensure the existence of a PBW basis is the so-called \textit{four-strand ambiguity}; see e.g. the final diagram in~\cite[(5.1)]{Elias-diamond}\footnote{Note that this paper predates~\cite{lai2024quantum}!}.
The set of conditions it imposes on the deformation parameters is prohibitively complicated to unwind.
However, our definition of TIC algebras allows to sidestep this thorny issue via an elementary trick:
\begin{itemize}[leftmargin=15pt]
	\item When $\rho = 0$, TIC algebra axioms imply that the cubic relation is not deformed (\cref{prop:rho-zero-nondef}), hence the main theorem of~\cite{lai2024quantum} applies;
	\item Renormalizing $H \mapsto H - r$, we reduce to the case $\rho = 0$;
	\item When the cubic relation is not deformed, the formulas for $S,D, \overline{D}, E$ are trivial to obtain.
\end{itemize}
Our simplified conditions furthermore allow to reverse the logic of constructing deformed quantum wreath products.
Instead of picking deformation parameters and painstakingly checking all required conditions, we can now treat the existence of PBW basis as an \textit{ansatz}.
Then having chosen $r, R\in F^{\otimes 2}$, the theorem immediately tells us what the rest of the parameters must be.

\subsection{Applications}
Our main theorem systematically recovers various basis theorems for affine Hecke-like algebras, listed in \cref{sec:expls}.
Let us highlight some examples:
\begin{enumerate}[leftmargin=20pt]
	\item We simplify the proof of basis theorem for PQWP~\cite{lai2025schurification} and skew PQWP~\cite{buciumas2026quantum} algebras (\cref{ssec:PQWP,ssec:skew-PQWP});
	\item We extend the notion of quantum affine wreath product algebras to non-symmetric Frobenius algebras (\cref{subs:mult-case}). This generalization turns out to be more subtle than suggested in~\cite[Rem.~2.8]{rosso2020quantum}. Namely, it requires $\psi$ in the definition of TIC algebra to be non-trivial;
	\item We propose a generalization of quiver Hecke algebras, where one is allowed to insert a Frobenius algebra at each vertex of the quiver (\cref{ssec:Frob-KLR}).
\end{enumerate}
We plan to extend the Schur--Weyl duality results in~\cite{lai2024quantum} to TIC algebras in future work~\cite{LM26}.
As a first step in this direction, we study whether $F\wr \cal{H}_n$ admits an action on the tensor space $F^{\otimes n}$ extending the left regular action (\cref{prop:action}).

Finally, note that the formulas for the coefficients $D, \overline{D}, E$ in the main theorem are nothing else than the \textit{associative and quantum Yang--Baxter equations}.
Moreover, quantum Yang--Baxter equation also arises when we try to construct a ``polynomial'' representation of a given TIC algebra (\cref{subs:polyrep},~\cref{cor:TIC-and-poly}).
This hints that TIC algebras are closely related to Lie bialgebras and their deformations (\cref{sec:Lie-bialg}), and gives the present paper its name.

\subsection*{Organization} 
In order to cover certain super-versions of Hecke algebras, we begin with the lengthy \cref{sec:QWPs}, which extends and reformulates the setup of QWP algebras from vector spaces to (symmetric) tensor categories.
While we hope this can be of use for the emerging study of representation theory in tensor categories, the reader not interested in such generalizations is invited to skip directly to \cref{sec:inv}.
There, we move away from diagrammatics to purely algebraic relations, deform the cubic relation of QWP algebras, and introduce a localized version of the setup, designed to treat various flavours of Demazure operators.
The definition of TIC algebras is given in \cref{sec:computations}, along with the basis theorem, which follows from~\cite{lai2024quantum} via a bag of simple tricks as discussed above.
\cref{sec:expls} is one big list of examples of TIC algebras.
Finally, we describe some situations when the tensor space representation of $F\wr \cal{H}_n$ exists in \cref{sec:poly-exist}.

\subsection*{Acknowledgments} It is the author's pleasure to thank Aaron Hofer, Chun-Ju Lai, and Leo Schelstraete for illuminating discussions related to this work.
I am grateful to Max Planck Institute for Mathematics in Bonn and University of Bonn for their hospitality and financial support.

\medskip
\section{Quantum wreath products}\label{sec:QWPs}
In this section, we recall and extend the basis theorem for quantum wreath product algebras from~\cite{lai2024quantum} to algebras in tensor categories.
The proofs in \textit{loc.~cit.} carry over almost verbatim; our contribution is in appropriately generalizing the setup.
Our main motivation for this is to be able to treat some superalgebras of interest in \cref{ssec:odd-Hecke,ssec:PQWP-super,ssec:Sergeev}.
We urge the reader who is only interested in usual $\bbk$-linear algebras to skip directly to \cref{sec:inv}.

\subsection{Categorical incantations}
Let $\bbk$ be a commutative ring, and fix a $\bbk$-linear strict rigid braided monoidal abelian category $\cal{C}$ with the unit object $\bb{1}$, satisfying $\End_{\cal{C}}(\bb{1}) = \bbk$.
We liberally use the string calculus for morphisms in monoidal categories, see e.g.~\cite{BakKir}.
Our convention is that the string diagrams are read from bottom to top.
For instance, the braiding $\tau = \tau_{M,N}: M\otimes N\to N\otimes M$ will be denoted by an overcrossing:
\[
\tikz[thick,xscale=.4,yscale=.4,font=\footnotesize]{
	\overcros{0}{0}{2}{2}
	\ntxt{0}{-0.5}{$M$}
	\ntxt{2}{-0.5}{$N$}
	\ntxt{0}{2.5}{$N$}
	\ntxt{2}{2.5}{$M$}
	\ntxt{-1.5}{1}{$\tau_{M,N}=$}

	\ntxt{2.7}{0.9}{,}

	\undercros{7}{0}{9}{2}
	\ntxt{7}{-0.5}{$M$}
	\ntxt{9}{-0.5}{$N$}
	\ntxt{7}{2.5}{$N$}
	\ntxt{9}{2.5}{$M$}
	\ntxt{5.5}{1}{$\tau^{-1}_{M,N}=$}

	\ntxt{9.7}{0.9}{.}
}
\]
It is sometimes convenient for us to work in the \textit{ind-completion} of $\cal{C}$, which we denote by $\vec{\cal{C}}$.
This category is naturally $\bbk$-linear strict braided monoidal abelian, but not rigid.

\begin{expl}
	One can take $\cal{C}$ to be any symmetric fusion category~\cite{EGNO}.
	We are mostly interested in the categories $\Vec$, $\sVec$ of finite dimensional vector, resp. super-vector spaces over a field $\bbk$.
	Note that passing to ind-completions $\vec{\Vec}$, $\vec{\sVec}$ is equivalent to forgetting about finite-dimensionality.
	For a more exotic example, one can set $\cal{C}$ to be the Verlinde category $\mathsf{Ver}_p$, $p>2$, see~\cite{Ostrik}.
\end{expl}

\subsection{Algebra objects}\label{sec:alg-obj}
Recall that a \textit{(unital) algebra} in $\cal{C}$ is the data of $(A,\mu,\eta)$, where $A\in \cal{C}$, $\mu: A\otimes A\to A$, $\eta:\bb{1}\to A$, such that the following diagrams commute:
\[
\begin{tikzcd}
	A\otimes A\otimes A \ar[r,"\mu\otimes \id"]\ar[d,"\id\otimes \mu"] & A\otimes A\ar[d,"\mu"]\\
	A\otimes A\ar[r,"\mu"] & A
\end{tikzcd}\qquad
\begin{tikzcd}
	\bb{1}\otimes A\ar[dr,"\simeq"']\ar[r,"\eta\otimes \id"] & A\otimes A\ar[d,"\mu"] & A\otimes \bb{1}\ar[dl,"\simeq"]\ar[l,"\id\otimes \eta"'] \\
	& A &
\end{tikzcd}
\]
We often drop either $\eta$ or both $\mu$ and $\eta$ from the notation, when they are clear from context.
In terms of string calculus, this definition amounts to
\[
\tikz[thick,xscale=.4,yscale=.4,font=\footnotesize,baseline=(current  bounding  box.center)]{
	\dmerge{0}{0}{2}{1.5}
	\draw (1,1.5) -- (1,2);
	\ntxt{-1}{1}{$\mu=$}
	\ntxt{2.4}{0.9}{,}
}
\tikz[thick,xscale=.4,yscale=.4,font=\footnotesize,baseline=(current  bounding  box.center)]{
	\identity{0}{0}{2}
	\ntxt{-1}{1}{$\eta=$}
	\ntxt{0.7}{0.9}{;}
}\qquad
\tikz[thick,xscale=.4,yscale=.4,font=\footnotesize,baseline=(current  bounding  box.center)]{
	\dmerge{0}{0}{1}{0.8}
	\draw (2,0) -- (2,0.8);
	\dmerge{0.5}{0.8}{2}{1.8}
	\draw (1.25,1.8) -- (1.25,2);

	\ntxt{3}{0.9}{$=$}

	\draw (4,0) -- (4,0.8);
	\dmerge{5}{0}{6}{0.8}
	\dmerge{4}{0.8}{5.5}{1.8}
	\draw (4.75,1.8) -- (4.75,2);

	\ntxt{6.4}{0.9}{,}
}\quad
\tikz[thick,xscale=.4,yscale=.4,font=\footnotesize,baseline=(current  bounding  box.center)]{
	\dmerge{0}{0.8}{1}{1.8}
	\draw (0.5,1.8) -- (0.5,2);
	\identity{0}{0}{0.8}
	\draw (1,0) -- (1,0.8);

	\ntxt{2}{0.9}{$=$}

	\draw (3,0) -- (3,2);

	\ntxt{4}{0.9}{$=$}

	\dmerge{5}{0.8}{6}{1.8}
	\draw (5.5,1.8) -- (5.5,2);
	\draw (5,0) -- (5,0.8);
	\identity{6}{0}{0.8}

	\ntxt{6.4}{0.9}{.}
}
\]

\begin{expl}\label{ex:construct-algs}
We have several simple ways to construct new algebras out of the existing ones:
\begin{enumerate}
	\item Let $A,B\in \cal{C}$ be two algebras. The product $A\otimes B$ has two natural algebra structures. We set $\eta_{A\otimes B} = \eta_A\otimes \eta_B$ in both cases, and the products are given by 
	\[
	\tikz[thick,xscale=.4,yscale=.4,font=\footnotesize,baseline=(current  bounding  box.center)]{
		\ntxt{-1}{1.5}{$\mu = $}
		\draw (0,0) -- (0,1.8);
		\overcros{1}{0}{3}{1.8}
		\draw (4,0) -- (4,1.8);
		\dmerge{0}{1.8}{1}{2.7}
		\dmerge{3}{1.8}{4}{2.7}
		\draw (0.5,2.7) -- (0.5,3);
		\draw (3.5,2.7) -- (3.5,3);
		\ntxt{4.4}{1.3}{,}
		\ntxt{0}{-0.5}{$A$}
		\ntxt{1}{-0.5}{$B$}
		\ntxt{3}{-0.5}{$A$}
		\ntxt{4}{-0.5}{$B$}
	}\qquad 
	\tikz[thick,xscale=.4,yscale=.4,font=\footnotesize,baseline=(current  bounding  box.center)]{
		\ntxt{-1.1}{1.5}{$\mu' = $}
		\draw (0,0) -- (0,1.8);
		\undercros{1}{0}{3}{1.8}
		\draw (4,0) -- (4,1.8);
		\dmerge{0}{1.8}{1}{2.7}
		\dmerge{3}{1.8}{4}{2.7}
		\draw (0.5,2.7) -- (0.5,3);
		\draw (3.5,2.7) -- (3.5,3);
		\ntxt{4.4}{1.3}{.}
		\ntxt{0}{-0.5}{$A$}
		\ntxt{1}{-0.5}{$B$}
		\ntxt{3}{-0.5}{$A$}
		\ntxt{4}{-0.5}{$B$}
	}
	\]
	Unless otherwise specified, we always use the first product.
	In particular, $A^{\otimes n}$ is an algebra for all $n\geq 1$. 
	\item Given an object $H\in \cal{C}$, consider the ind-object $T(H)\coloneqq \bigoplus_{k\geq 0}^\infty H^{\otimes k}$.
	This is naturally an algebra in $\vec{\cal{C}}$, called the \textit{tensor algebra} of $H$, for $\mu: H^{\otimes i}\otimes H^{\otimes j}\simto H^{\otimes (i+j)}$.
	\item Let $I\in \cal{C}$, and $i\in \Hom_{\cal{C}}(I,A)$. Define 
	\[
		A/\langle I\rangle\coloneqq \coker(A\otimes I\otimes A\xra{\id\otimes i\otimes \id}A^{\otimes 3}\xra{\mu\circ (\mu\otimes \id)} A).
	\]
	The product on $A$ induces a map $\mu: (A/\langle I\rangle )^{\otimes 2}\to A/\langle I\rangle$, and setting $\eta$ to be the composition $\bb{1}\xra{\eta}A\twoheadrightarrow A/\langle I\rangle$, we get an algebra $(A/\langle I\rangle,\mu,\eta)$.
\end{enumerate}
\end{expl}

Let $A\in\vec{\cal{C}}$ be an (ind-)algebra, $H\in \cal{C}$ an object, and consider $A\otimes T(H)\in\vec{\cal{C}}$.
By (1) and (2) above, this is an algebra in $\vec{\cal{C}}$; however, it admits many more algebra structures.
Let us draw $H$ as a dotted line, and choose two morphisms:
\[
	\tikz[thick,xscale=.4,yscale=.4,font=\footnotesize,baseline=(current  bounding  box.center)]{
		\ntxt{-1.4}{1}{$\sigma=$}
		\ntxt{0}{-0.5}{$H$}
		\ntxt{2}{-0.5}{$A$}
		\sigact{0}{0}{2}{2}
		\ntxt{6.5}{1}{$\in \Hom_{\cal{C}}(H\otimes A, A\otimes H)$,}
	}\qquad 
	\tikz[thick,xscale=.4,yscale=.4,font=\footnotesize,baseline=(current  bounding  box.center)]{
		\ntxt{-1.2}{1}{$\rho=$}
		\ntxt{0}{-0.5}{$H$}
		\ntxt{2}{-0.5}{$A$}
		\rhoact{0}{0}{2}{2}
		\ntxt{6}{1}{$\in \Hom_{\cal{C}}(H\otimes A, A)$.}
	}
\]
Consider the map $\mu: (A\otimes T(H))^{\otimes 2}\to A\otimes T(H)$, where $\mu = \sum_{k,j\geq 0} \mu_{k,j}$, and
\begin{equation}\label{eq:mu-kj}
	\begin{gathered}
		\mu_{k,j}: (A\otimes H^{\otimes k})\otimes (A\otimes H^{\otimes j})\to \bigoplus_{i=0}^k A\otimes H^{\otimes i+j},\\
		\mu_{k,j} = (\mu_A\otimes \id)\circ(\id_{A}\otimes (\id\otimes\sigma\otimes\id + \id\otimes\rho\otimes\id)^{k} \otimes \id_{H^{\otimes j}}).
	\end{gathered}
\end{equation}
Let us elucidate the meaning of the symbol $(\id\otimes\sigma\otimes\id + \id\otimes\rho\otimes\id)^{k}$ above.
This is a map $H^{\otimes k}\otimes A\to \bigoplus_{i\leq k} A\otimes H^{\otimes i}$.
Each time we apply $\sigma$, it moves $A$ past one copy of $H$, and $\rho$ removes a copy of $H$ to the left of $A$.
So we apply either $\sigma$ or $\rho$ at an appropriate place, and $(\id\otimes - \otimes \id)$ serves to remind that we have some extra copies of $H$ on both sides.
Observe that $\mu$ extends the products on $T(H)\simeq \bb{1}\otimes T(H)$ and $A\simeq A\otimes \bb{1}$ inside $A\otimes T(H)$.

\begin{expl}
	The map $\mu_{2,0}$ is given by the following diagram:
	\[
	\tikz[thick,xscale=.4,yscale=.4,font=\footnotesize,baseline=(current  bounding  box.center)]{
		\ntxt{-1.4}{1.4}{$\mu_{2,0}=$}
		\draw (0,0) -- (0,2);
		\draw[dotted,line cap=round,rounded corners] (1,0) -- (1,1);
		\sigact{2}{0}{3}{1} 
		\sigact{1}{1}{2}{2} 
		\dmerge{0}{2}{1}{2.7}
		\draw (0.5,2.7) -- (0.5,3);
		\draw[dotted,line cap=round,rounded corners] (2,2) -- (2,3);
		\draw[dotted,line cap=round,rounded corners] (3,1) -- (3,3);
		\ntxt{3.5}{1.4}{$+$}

		\begin{scope}[xshift=4cm]
			\draw (0,0) -- (0,2);
			\draw[dotted,line cap=round,rounded corners] (1,0) -- (1,1);
			\sigact{2}{0}{3}{1} 
			\rhoact{1}{1}{2}{2} 
			\dmerge{0}{2}{1.5}{2.7}
			\draw (0.75,2.7) -- (0.75,3);
			\draw[dotted,line cap=round,rounded corners] (3,1) -- (3,3);
			\ntxt{3.5}{1.4}{$+$}
		\end{scope}

		\begin{scope}[xshift=8cm]
			\draw (0,0) -- (0,2);
			\draw[dotted,line cap=round,rounded corners] (1,0) -- (1,1);
			\rhoact{2}{0}{3}{1} 
			\sigact{1}{1}{2.5}{2} 
			\dmerge{0}{2}{1}{2.7}
			\draw (0.5,2.7) -- (0.5,3);
			\draw[dotted,line cap=round,rounded corners] (2.5,2) -- (2.5,3);
			\ntxt{3.5}{1.4}{$+$}
		\end{scope}

		\begin{scope}[xshift=12cm]
			\draw (0,0) -- (0,2);
			\draw[dotted,line cap=round,rounded corners] (1,0) -- (1,1);
			\rhoact{2}{0}{3}{1} 
			\rhoact{1}{1}{2.5}{2} 
			\dmerge{0}{2}{1.75}{2.7}
			\draw (0.875,2.7) -- (0.875,3);
			\ntxt{3.3}{1.3}{.}
		\end{scope}
	}
	\]
\end{expl}

\begin{prop}\label{prop:tens-alg}
	Let $\eta = \eta_A\otimes \eta_{T(H)}$, and $\mu$ as above.
	Write $\sigma_{12}\coloneqq \sigma\otimes \id_A$, $\sigma_{23}\coloneqq \id_A\otimes\sigma$, and same for $\rho$.
	Then $(A\otimes T(H),\mu,\eta)$ is an algebra in $\vec{\cal{C}}$ if and only if the following conditions hold:
	\begin{gather}
		\label{eq:P-unit}\sigma\circ (\id_H\otimes \eta_A) = \eta_A\otimes \id_H, \qquad \rho\circ (\id_H\otimes\eta_A) = 0,\\
		\label{eq:P-der}\sigma\circ (\id_H\otimes\mu_A) = (\mu_A\otimes \id_H)\circ\sigma_{23}\circ\sigma_{12},\qquad \rho\circ(\id_H\otimes\mu_A) = \mu_A\circ (\rho_{23}\circ \sigma_{12} + \rho_{12}).
	\end{gather}
\end{prop}
In terms of string calculus, the conditions (\ref{eq:P-unit},~\ref{eq:P-der}) are written as follows:
\begin{gather}\label{diag:wr}\tag{D1}
	\tikz[thick,xscale=.4,yscale=.4,font=\footnotesize,baseline=(current  bounding  box.center)]{
		\draw [dotted,line cap=round,rounded corners] (0,0) -- (0,1);
		\identity{1}{0}{1}
		\sigact{0}{1}{1}{2}
		\ntxt{2}{0.8}{$=$}
		\draw [dotted,line cap=round,rounded corners] (4,0) -- (4,2);
		\identity{3}{0}{2}
		\ntxt{4.5}{0.8}{,}
	}\qquad
	\tikz[thick,xscale=.4,yscale=.4,font=\footnotesize,baseline=(current  bounding  box.center)]{
		\draw [dotted,line cap=round,rounded corners] (0,0) -- (0,1);
		\identity{1}{0}{1}
		\rhoact{0}{1}{1}{2}
		\ntxt{2}{0.8}{$=0,$}
	}\\ \label{diag:wr2}\tag{D2}
	\tikz[thick,xscale=.4,yscale=.4,font=\footnotesize,baseline=(current  bounding  box.center)]{
		\draw [dotted,line cap=round,rounded corners] (0,0) -- (0,1);
		\dmerge{1}{0}{2}{1}
		\sigact{0}{1}{1.5}{2}
		\ntxt{3}{0.8}{$=$}
		\curline{4}{1}{5}{0}
		\curline{5}{1}{6}{0}
		\curlinedot{4}{0}{6}{1}
		\strdotwhite{4.64}{0.42}
		\strdotwhite{5.34}{0.57}
		\draw [dotted,line cap=round,rounded corners] (6,1) -- (6,2);
		\dmerge{4}{1}{5}{2}
		\ntxt{6.5}{0.8}{,}
	}\qquad
	\tikz[thick,xscale=.4,yscale=.4,font=\footnotesize,baseline=(current  bounding  box.center)]{
		\draw [dotted,line cap=round,rounded corners] (0,0) -- (0,1);
		\dmerge{1}{0}{2}{1}
		\rhoact{0}{1}{1.5}{2}
		\ntxt{3}{0.8}{$=$}
		\curline{4}{1}{5}{0}
		\curline{5}{1}{6}{0}
		\curlinedot{4}{0}{6}{1}
		\strdotwhite{4.64}{0.42}
		\fill[white] (5.34,0.57) rectangle (6.1,1);
		\strdot{5.34}{0.57}
		\dmerge{4}{1}{5}{2}
		\ntxt{6.5}{0.8}{$+$}
		\rhoact{7}{0}{8}{1}
		\draw (9,0) -- (9,1);
		\dmerge{7.5}{1}{9}{2}
		\ntxt{9.5}{0.8}{.}
	}
\end{gather}
Note that the second condition of~\eqref{eq:P-unit} follows from the other three conditions. Indeed,
\[
	\tikz[thick,xscale=.4,yscale=.4,font=\footnotesize,baseline=(current  bounding  box.center)]{
		\ntxt{-3}{1.3}{$\rho\circ (\id_H\otimes\eta_A)=$}

		\draw [dotted,line cap=round,rounded corners] (0,0) -- (0,2);
		\identity{1}{0}{2}
		\rhoact{0}{2}{1}{3}
		\ntxt{2}{1.3}{$=$}
		
		\draw [dotted,line cap=round,rounded corners] (3,0) -- (3,2);
		\identity{4}{0}{1}
		\identity{5}{0}{1}
		\dmerge{4}{1}{5}{2}
		\rhoact{3}{2}{4.5}{3}
		\ntxt{6}{1.3}{$=$}

		\curline{7}{2}{8}{1}
		\curline{8}{2}{9}{1}
		\curlinedot{7}{1}{9}{2}
		\strdotwhite{7.64}{1.42}
		\fill[white] (8.34,1.57) rectangle (9.1,2);
		\strdot{8.34}{1.57}
		\dmerge{7}{2}{8}{3}
		\draw [dotted,line cap=round,rounded corners] (7,0) -- (7,1);
		\identity{8}{0}{1}
		\identity{9}{0}{1}
		\ntxt{9.5}{1.3}{$+$}
		\rhoact{10}{1}{11}{2}
		\draw (12,1) -- (12,2);
		\dmerge{10.5}{2}{12}{3}
		\draw [dotted,line cap=round,rounded corners] (10,0) -- (10,1);
		\identity{11}{0}{1}
		\identity{12}{0}{1}
		\ntxt{13}{1.3}{$=$}

		\draw [dotted,line cap=round,rounded corners] (14,0) -- (14,2);
		\identity{15}{0}{2}
		\rhoact{14}{2}{15}{3}
		\ntxt{15.5}{1.3}{$+$}
		\draw [dotted,line cap=round,rounded corners] (16,0) -- (16,2);
		\identity{17}{0}{2}
		\rhoact{16}{2}{17}{3}
		\ntxt{17.5}{1.3}{,}
	}
\]
and so $\rho\circ (\id_H\otimes\eta_A)$ must vanish.

\begin{proof}[Proof of \cref{prop:tens-alg}]
	Let us denote $\pi = \sigma + \rho: H\otimes A\to (A\otimes H) \oplus A$ for brevity.
	First, let us prove that the conditions (\ref{eq:P-unit},~\ref{eq:P-der}) are necessary.
	For this, we restrict our attention to $A \oplus (A\otimes H) \subset A\otimes T(H)$.
	Using the unitality condition on $\bb{1}\otimes H$, we obtain
	\begin{align*}
		\eta_A\otimes \id_H 
		= (\mu\otimes\id_{H})\circ (\id_A\otimes \pi)\circ (\eta_A\otimes \id_H\otimes \eta_A)
		= \sigma\circ(\id_H\otimes \eta_A) + \rho\circ(\id_H\otimes \eta_A),
	\end{align*}
	which yields the equations~\eqref{eq:P-unit}.
	Using the associativity condition on $(\bb{1}\otimes H)\otimes A\otimes A$, we obtain
	\begin{align*}
		\mu\circ(\id\otimes \mu)
		&= (\mu_A\otimes \id_H, \mu_A)\circ(\id_A\otimes \pi)\circ(\eta_A \otimes\id_{H}\otimes \mu_A)
		= \pi\circ(\id_H\otimes \mu_A),\\
		\mu\circ(\mu\otimes \id)
		&= (\mu_A\otimes \id_H, \mu_A)\circ (\id_A\otimes\pi,\id_{A\otimes A})\circ (\mu_A\otimes \id_{H\otimes A},\mu_A\otimes \id_{A})\circ (\eta_A\otimes \pi\otimes \id_A)\\
		& = \mu_A\circ (\pi, \id_{A\otimes A})\circ (\pi\otimes \id_A) = \mu_A\circ ((\id_A\otimes \rho)\circ (\sigma\otimes\id_A) + \rho\otimes\id_{A}),
	\end{align*}
	which yields equations~\eqref{eq:P-der}. 

	Now let us assume that both~\eqref{eq:P-unit} and~\eqref{eq:P-der} hold.
	The left triangle of the unitality condition is obviously commutative.
	Evaluating the right triangle on $A\otimes H^{\otimes k}$, we have
	\begin{align*}
		\mu\circ(\id_{A\otimes H^{\otimes k}}\otimes \eta_A)
		&= \mu_A \circ (\id_A\otimes\pi^k)\circ (\id_{A\otimes H^{\otimes k}}\otimes \eta_A)\\
		&\stackrel{\eqref{eq:P-unit}}{=} \mu_A\circ (\id_A\otimes\eta_A\otimes\id_{H^{\otimes k}})
		= \id_A\otimes \id_{H^{\otimes k}},
	\end{align*}
	and so it also commutes.
	For the associativity, let $W_k$ be the set of words in letters $\sigma, \rho$ of length $k$, and for $w\in W_k$ denote by $\sigma(w)$ the amount of $\sigma$'s occurring in $w$.
	A simple induction argument using~\eqref{eq:P-der} as the inductive step shows that 
	\begin{equation}\label{eq:higher-der}
		(\sigma+\rho)^k\circ (\id_{H^{\otimes k}}\otimes\mu_A) = (\mu_A\otimes \id)\circ \left(\sum_{w\in W_k}(\id_A\otimes (\sigma+\rho)^{\sigma(w)})\circ(w\otimes\id_A)\right).
	\end{equation}
	Shedding irrelevant factors from the associativity diagram, we need to prove that, for any $k,l\geq 0$ we have  
	\[
		\pi^k\circ (\id_{H^{\otimes k}}\otimes \mu_A\otimes \id)\circ (\id_{H^{\otimes k}\otimes A}\otimes \pi^l) 
		= (\mu_A\otimes \id)\circ (\id_A\otimes (\pi^{k+l},\ldots,\pi^l)) \circ (\pi^k\otimes \id_{H^{\otimes l}\otimes A}).
	\]
	Note that we can further remove the $H^{\otimes l}$-factor, and so the precomposition by $\pi^l$.
	It remains to prove the following equality of maps with source $H^{\otimes k}\otimes A\otimes A$:
	\[
		(\pi_k\otimes \id)\circ (\id_{H^{\otimes k}}\otimes \mu_A)
		= (\mu_A\otimes \id)\circ (\id_A\otimes (\pi^{k},\ldots,\id)) \circ (\pi^k\otimes \id_{A}).
	\]
	This is exactly the equation~\eqref{eq:higher-der}, written in a different form.
\end{proof}

The following two claims are proved by direct computation.

\begin{lem}\label{lem:extra-lines}
	Let $H\in \cal{C}$, and $\sigma:H\otimes A\to A\otimes H$, $\rho: H\otimes A\to A$ satisfy the conditions~(\ref{eq:P-unit},~\ref{eq:P-der}).
	Pick two other algebras $B_1,B_2\in\cal{C}$, and consider the product $B_1\otimes A\otimes B_2$.
	Then the maps 
	\begin{align*}
	\sigma' & \coloneqq (\id_{B_1\otimes A}\otimes \tau^{-1}_{H,B_2})\circ(\id_{B_1}\otimes \sigma\otimes \id_{B_2})\circ(\tau_{H,B_1}\otimes \id_{A\otimes B_2}),\\
	\rho' & \coloneqq (\id_{B_1}\otimes \rho\otimes \id_{B_2})\circ(\tau_{H,B_1}\otimes \id_{A\otimes B_2})
	\end{align*}
	also satisfy~(\ref{eq:P-unit},~\ref{eq:P-der}), and so turn $(B_1\otimes A\otimes B_2)\otimes T(H)$ into an algebra.\qed
\end{lem}

\begin{lem}\label{lem:sum-of-tensors}
	For each $1\leq i\leq n$, let $H_i\in\cal{C}$, and $\sigma_i:H_i\otimes A\to A\otimes H_i$, $\rho_i: H_i\otimes A\to A$ satisfy the conditions~(\ref{eq:P-unit},~\ref{eq:P-der}).
	Then $\sigma\coloneqq \sum_i \sigma_i$, $\rho\coloneqq \sum_i \rho_i$ also satisfy~(\ref{eq:P-unit},~\ref{eq:P-der}), and so turn $A\otimes T(\bigoplus_i H_i)$ into an algebra.\qed
\end{lem}

\begin{expl}\label{ex:symm-braiding}
	Let $F\in \cal{C}$ be an algebra, and set $A = F\otimes F$.
	If the braiding $\tau_{F,F}$ is symmetric, the maps $\sigma = \tau_{H,F^{\otimes 2}}\circ(\id_H\otimes\tau_{F,F})$, $\rho = 0$ satisfy the conditions (\ref{eq:P-unit},~\ref{eq:P-der}).
	Indeed, the only non-trivial condition is proved as follows, where we write $\tau = \tau_{F\otimes F}$, $\mu = \mu_F$:
	\[
	\tikz[thick,xscale=.45,yscale=.45,font=\footnotesize,baseline=(current  bounding  box.center)]{
		\ntxt{-2.2}{1.8}{$\sigma\circ \mu_A=$}
		\draw (0,0) -- (0,1.5);
		\overcros{1}{0}{3}{1.5}
		\draw (4,0) -- (4,1.5);
		\dmerge{0}{1.5}{1}{2.3}
		\dmerge{3}{1.5}{4}{2.3}
		\overcros{0.5}{2.3}{3.5}{3}
		\ntxt{4.6}{1.8}{$=$}
		\draw (0.5,3) -- (0.5,4);
		\draw (3.5,3) -- (3.5,4);
		\draw[line width=3pt, white] (-0.5,2.8) .. controls (-0.5,4*0.6+2.8*0.4) and (4.5,4*0.4+2.8*0.6) .. (4.5,4);
		\draw [dotted,line cap=round,rounded corners] (-0.5,2.8) .. controls (-0.5,4*0.6+2.8*0.4) and (4.5,4*0.4+2.8*0.6) .. (4.5,4);
		\draw [dotted,line cap=round,rounded corners] (-0.5,0) -- (-0.5,2.8);

		\begin{scope}[xshift=6cm]
		\draw (0,0) -- (0,1.5);
		\undercros{1}{0}{3}{1.5}
		\draw (4,0) -- (4,1.5);
		\dmerge{0}{1.5}{1}{2.3}
		\dmerge{3}{1.5}{4}{2.3}
		\overcros{0.5}{2.3}{3.5}{3}
		\ntxt{4.6}{1.8}{$=$}
		\draw (0.5,3) -- (0.5,4);
		\draw (3.5,3) -- (3.5,4);
		\draw[line width=3pt, white] (-0.5,2.8) .. controls (-0.5,4*0.6+2.8*0.4) and (4.5,4*0.4+2.8*0.6) .. (4.5,4);
		\draw [dotted,line cap=round,rounded corners] (-0.5,2.8) .. controls (-0.5,4*0.6+2.8*0.4) and (4.5,4*0.4+2.8*0.6) .. (4.5,4);
		\draw [dotted,line cap=round,rounded corners] (-0.5,0) -- (-0.5,2.8);
		\end{scope}

		\begin{scope}[xshift=12cm]
		\draw (1,0) .. controls (1,1.8*0.6) and (0,1.8*0.4) .. (0,1.8);
		\draw (4,0) .. controls (4,1.8*0.6) and (1,1.8*0.4) .. (1,1.8);
		\draw[line width=3pt, white] (0,0) .. controls (0,1.8*0.6) and (3,1.8*0.4) .. (3,1.8);
		\draw (0,0) .. controls (0,1.8*0.6) and (3,1.8*0.4) .. (3,1.8);
		\draw[line width=3pt, white] (3,0) .. controls (3,1.8*0.6) and (4,1.8*0.4) .. (4,1.8);
		\draw (3,0) .. controls (3,1.8*0.6) and (4,1.8*0.4) .. (4,1.8);
		\dmerge{0}{1.8}{1}{2.7}
		\dmerge{3}{1.8}{4}{2.7}
		\draw (0.5,2.7) -- (0.5,3.5);
		\draw (3.5,2.7) -- (3.5,3.5);
		\ntxt{4.6}{1.8}{$=$}
		\draw (0.5,3) -- (0.5,4);
		\draw (3.5,3) -- (3.5,4);
		\draw[line width=3pt, white] (-0.5,2.8) .. controls (-0.5,4*0.6+2.8*0.4) and (4.5,4*0.4+2.8*0.6) .. (4.5,4);
		\draw [dotted,line cap=round,rounded corners] (-0.5,2.8) .. controls (-0.5,4*0.6+2.8*0.4) and (4.5,4*0.4+2.8*0.6) .. (4.5,4);
		\draw [dotted,line cap=round,rounded corners] (-0.5,0) -- (-0.5,2.8);
		\end{scope}

		\begin{scope}[xshift=18cm]
		\overcros{0}{0}{1}{1}
		\overcros{3}{0}{4}{1}
		\overcros{1}{1.7}{3}{2.7}
		\draw (0,1) -- (0,2.7);
		\draw (1,1) -- (1,1.7);
		\draw (3,1) -- (3,1.7);
		\draw (4,1) -- (4,2.7);
		\dmerge{0}{2.7}{1}{3.7}
		\dmerge{3}{2.7}{4}{3.7}
		\draw (0.5,3.7) -- (0.5,4);
		\draw (3.5,3.7) -- (3.5,4);
		\draw[line width=3pt, white] (-0.5,0.7) .. controls (-0.5,2*0.6+0.7*0.4) and (4.5,2*0.4+0.7*0.6) .. (4.5,2);
		\draw [dotted,line cap=round,rounded corners] (-0.5,0.7) .. controls (-0.5,2*0.6+0.7*0.4) and (4.5,2*0.4+0.7*0.6) .. (4.5,2);
		\draw [dotted,line cap=round,rounded corners] (-0.5,0) -- (-0.5,0.7);
		\draw [dotted,line cap=round,rounded corners] (4.5,2) -- (4.5,4);
		\ntxt{7.5}{1.8}{$=\mu_A\otimes (\sigma\otimes \sigma).$}
		\end{scope}
	}
	\]
	In general, $\tau_{F,F}$ is instead a map of algebras $(F\otimes F,\mu')\to (F\otimes F,\mu)$, where $\mu$, $\mu'$ are as in \cref{ex:construct-algs}(1).
\end{expl}

\subsection{Quantum wreath products}\label{subs:qwp}
Let $H\in \cal{C}$, $n\geq 2$, and write $\Vn \coloneqq \bigoplus_{i=1}^{n-1} H_i$, where each $H_i$ is a copy of $H$.
Pick an algebra $F\in \vec{\cal{C}}$, together with maps $\sigma\in \Hom_{\cal{C}}(H\otimes F^{\otimes 2}, F^{\otimes 2}\otimes H)$, $\rho \in \Hom_{\cal{C}}(H\otimes F^{\otimes 2}, F^{\otimes 2})$ satisfying the conditions (\ref{eq:P-unit},~\ref{eq:P-der}).
Set 
\[
A \coloneqq F^{\otimes n}.
\]
For each $1\leq i\leq n-1$, we define 
\begin{align*}
	\sigma_i & \coloneqq (\id_{F^{\otimes (i+1)}}\otimes \tau^{-1}_{H,F^{\otimes (n-i-1)}})\circ(\id_{F^{\otimes (i-1)}}\otimes \sigma\otimes \id_{F^{\otimes (n-i-1)}})\circ(\tau_{H,F^{\otimes (i-1)}}\otimes \id_{F^{\otimes (n-i+1)}}):H\otimes A\to A\otimes H,\\
	\rho_i & \coloneqq (\id_{F^{\otimes (i-1)}}\otimes \rho\otimes \id_{F^{\otimes (n-i-1)}})\circ(\tau_{H,F^{\otimes (i-1)}}\otimes \id_{F^{\otimes (n-i+1)}}):H\otimes A\to A.
\end{align*}
The following relations follow immediately from the definitions for any $j\geq i+2$:
\begin{equation}\label{eq:sig-rho}
\tikz[thick,xscale=.5,yscale=.5,font=\footnotesize,baseline=(current  bounding  box.center)]{
	\draw [dotted,line cap=round,rounded corners] (0,0) -- (0,1);
	\sigact{1}{0}{2}{1}
	\sigact{0}{1}{1}{2}
	\draw [dotted,line cap=round,rounded corners] (2,1) -- (2,2);
	\overcrosdash{1}{2}{2}{3}
	\draw (0,2) -- (0,3);
	\ntxt{2.5}{1.5}{$=$}
	\ntxt{0}{-0.5}{$j$}
	\ntxt{1}{-0.5}{$i$}
	\ntxt{2}{-0.5}{$A$}

	\begin{scope}[xshift=3cm]
	\draw [dotted,line cap=round,rounded corners] (0,1) -- (0,2);
	\sigact{1}{1}{2}{2}
	\sigact{0}{2}{1}{3}
	\draw [dotted,line cap=round,rounded corners] (2,2) -- (2,3);
	\overcrosdash{0}{0}{1}{1}
	\draw (2,0) -- (2,1);
	\ntxt{2.3}{1.4}{,}
	\ntxt{0}{-0.5}{$j$}
	\ntxt{1}{-0.5}{$i$}
	\ntxt{2}{-0.5}{$A$}
	\end{scope}
}\quad
\tikz[thick,xscale=.5,yscale=.5,font=\footnotesize,baseline=(current  bounding  box.center)]{
	\draw [dotted,line cap=round,rounded corners] (0,0) -- (0,1);
	\sigact{1}{0}{2}{1}
	\rhoact{0}{1}{1}{2}
	\draw [dotted,line cap=round,rounded corners] (2,1) -- (2,3);
	\draw (0.5,2) -- (0.5,3);
	\ntxt{2.5}{1.5}{$=$}
	\ntxt{0}{-0.5}{$j$}
	\ntxt{1}{-0.5}{$i$}
	\ntxt{2}{-0.5}{$A$}

	\begin{scope}[xshift=3cm]
	\draw [dotted,line cap=round,rounded corners] (0,1) -- (0,2);
	\rhoact{1}{1}{2}{2}
	\sigact{0}{2}{1.5}{3}
	\overcrosdash{0}{0}{1}{1}
	\draw (2,0) -- (2,1);
	\ntxt{2.3}{1.4}{,}
	\ntxt{0}{-0.5}{$j$}
	\ntxt{1}{-0.5}{$i$}
	\ntxt{2}{-0.5}{$A$}
	\end{scope}
}\quad
\tikz[thick,xscale=.5,yscale=.5,font=\footnotesize,baseline=(current  bounding  box.center)]{
	\draw [dotted,line cap=round,rounded corners] (0,0) -- (0,1);
	\rhoact{1}{0}{2}{1}
	\sigact{0}{1}{1.5}{2}
	\draw [dotted,line cap=round,rounded corners] (1.5,2) -- (1.5,3);
	\draw (0,2) -- (0,3);
	\ntxt{2.5}{1.5}{$=$}
	\ntxt{0}{-0.5}{$j$}
	\ntxt{1}{-0.5}{$i$}
	\ntxt{2}{-0.5}{$A$}

	\begin{scope}[xshift=3cm]
	\draw [dotted,line cap=round,rounded corners] (0,1) -- (0,2);
	\sigact{1}{1}{2}{2}
	\rhoact{0}{2}{1}{3}
	\draw [dotted,line cap=round,rounded corners] (2,2) -- (2,3);
	\overcrosdash{0}{0}{1}{1}
	\draw (2,0) -- (2,1);
	\ntxt{2.3}{1.4}{,}
	\ntxt{0}{-0.5}{$j$}
	\ntxt{1}{-0.5}{$i$}
	\ntxt{2}{-0.5}{$A$}
	\end{scope}
}\quad
\tikz[thick,xscale=.5,yscale=.5,font=\footnotesize,baseline=(current  bounding  box.center)]{
	\draw [dotted,line cap=round,rounded corners] (0,0) -- (0,1);
	\rhoact{1}{0}{2}{1}
	\rhoact{0}{1}{1.5}{2}
	\draw (0.75,2) -- (0.75,3);
	\ntxt{2.5}{1.5}{$=$}
	\ntxt{0}{-0.5}{$j$}
	\ntxt{1}{-0.5}{$i$}
	\ntxt{2}{-0.5}{$A$}

	\begin{scope}[xshift=3cm]
	\draw [dotted,line cap=round,rounded corners] (0,1) -- (0,2);
	\rhoact{1}{1}{2}{2}
	\rhoact{0}{2}{1.5}{3}
	\overcrosdash{0}{0}{1}{1}
	\draw (2,0) -- (2,1);
	\ntxt{2.3}{1.4}{.}
	\ntxt{0}{-0.5}{$j$}
	\ntxt{1}{-0.5}{$i$}
	\ntxt{2}{-0.5}{$A$}
	\end{scope}
}
\end{equation}
By \cref{lem:extra-lines,lem:sum-of-tensors}, the sums $\sum_i \sigma_i$, $\sum_i \rho_i$ give rise to an algebra structure on $F^{\otimes n}\otimes T(\Vn)$.

We further fix the following datum:
\begin{align*}
	S \in \Hom_{\cal{C}}(H^{\otimes 2},F^{\otimes 2}\otimes H),\qquad
	R\in \Hom_{\cal{C}}(H^{\otimes 2},F^{\otimes 2}).
\end{align*}
For any collection of indices $I\subset \{1,\ldots, n\}$ we write $\iota_I:F^{\otimes |I|}\to F^n$ for the inclusion, given by tensoring identity maps for $i\in I$ and unit maps for $i\not\in I$.
For each $1\leq i\leq n-1$, denote the identity map $H_i\simto H$ by $h_i$, and consider the following ``coloured'' versions of our extra data:
\begin{align*}
	S_i = (\iota_{\{i,i+1\}}\otimes h_i^{-1}) \circ S \circ (h_i\otimes h_i),\qquad
	R_i = \iota_{\{i,i+1\}} \circ R \circ (h_i\otimes h_i).
\end{align*}
We also define $Z_i = (h_i\otimes h_{i+1}\otimes h_i)^{-1} \circ (h_{i+1}\otimes h_i\otimes h_{i+1})$.
Note that this is essentially the identity map $\id_{H^{\otimes 3}}$, except for being coloured in a non-trivial fashion.
Diagrammatically, we draw $H_i$'s as dotted lines, and $A = F^{\otimes n}$ as a solid line. 
We draw the maps above as follows, where $H_i$ is yellow and $H_{i+1}$ blue:
\[
	\tikz[thick,xscale=.42,yscale=.42,font=\footnotesize,baseline=(current  bounding  box.center)]{
		\ntxt{-1}{0.9}{$S_i=$}
		\morSi{0}{0}{2}{2}
		\ntxt{2.5}{0.9}{,}
	}\qquad 
	\tikz[thick,xscale=.42,yscale=.42,font=\footnotesize,baseline=(current  bounding  box.center)]{
		\ntxt{-1}{0.9}{$R_i=$}
		\morRi[OI1]{0}{0}{2}{2}
		\ntxt{2.5}{0.9}{,}
	}\qquad 
	\tikz[thick,xscale=.42,yscale=.42,font=\footnotesize,baseline=(current  bounding  box.center)]{
		\ntxt{-1}{0.9}{$Z_i=$}
		\morTi{0}{0}{2}{2}
		\ntxt{2.5}{0.9}{.}
	}
\]
Let us write $H_{ij}\coloneqq H_i\otimes H_j$ and so on for brevity.
Consider the following collection of maps into $A\otimes T(\Vn)$:
\begin{enumerate}
	\item For all $1\leq i\leq n-1$, $\mathbf{Q}^i = (\id_{H_{ii}},S_i,R_i): H_{ii}\to A\otimes T(\Vn)$;
	\item For all $1\leq i<j-1\leq n-2$, $\mathbf{B}_2^{ij}=(\id_{H_{ji}}, \tau_{H_j,H_i}): H_{ji}\to A\otimes T(\Vn)$; 
	\item For all $1\leq i\leq n-1$, $\mathbf{B}_3^i=(\id_{H_{i+1,i,i+1}},Z_i): H_{i+1,i,i+1}\to H_{i,i+1,i}$.
\end{enumerate}
This collection gives rise to a map $\mathbf{I}_n\to A\otimes T(\Vn)$, where $\mathbf{I}_n$ is the big direct sum  
\[
	\mathbf{I}_n \coloneqq \bigoplus_{i=1}^{n-1} H_{ii} \oplus \bigoplus_{i=1}^{n-3}\bigoplus_{j=i+2}^{n-1} H_{ji} \oplus \bigoplus_{i=1}^{n-1} H_{i+1,i,i+1}.
\]

\begin{defn}\label{def:QWP}
	Let $F\in \cal{C}$ be an algebra.
	Denote $\qbar\coloneqq (\sigma,\rho,S,R)$, and call this quadruple the \textit{deformation datum}.
	The \textit{quantum wreath product} (QWP) associated to $(F,\qbar)$ is the quotient algebra $F\wr\cal{H}_n = F\wr \cal{H}_n(\qbar)\coloneqq (F^{\otimes n}\otimes T(\Vn))/\langle \mathbf{I}_n\rangle$ as above.
\end{defn}

\subsection{Necessary conditions}
For each $w\in \fkS_n$, pick a reduced expression $r(w) = s_{i_1}\ldots s_{i_l}$, such that $r(1)=1$, and for every $w\in \fkS_n\setminus\{1\}$ there exists a unique $s_i$, satisfying the equality of reduced expressions $r(w) = r(ws_i)s_i$.
With this choice in mind, denote
\begin{align*}
	H_w \coloneqq H_{i_1\ldots i_l} = H_{i_1}\otimes\ldots \otimes H_{i_l}.
\end{align*}
\begin{lem}\label{lem:phi-surj}
	The composition $\varphi_n:\bigoplus_{w\in \fkS_n} A\otimes H_w \hookrightarrow A\otimes T(\Vn) \twoheadrightarrow F\wr\cal{H}_n$ is surjective.
\end{lem}
\begin{proof}
	Using the relations $\mathbf{B}_2^{ij}$, $\mathbf{B}_3^i$, the image of each summand of $A\otimes T(\Vn)$ in $F\wr \mathcal{H}_n$ is isomorphic to the image of a summand of the form $A\otimes H_{i_1\ldots i_l}$, where $s_{i_1}\ldots s_{i_l}$ is a reduced expression in the braid group $\mathfrak{Br}_n$.
	Further using the relation $\mathbf{Q}^{i}$, we can express the images of summands with repeated indices in terms of summands $A\otimes H_w$, $w\in\mathfrak{S}_n$.
\end{proof}

Our goal is to obtain the conditions on the deformation datum $\qbar$, under which the map $\phi_n$ is an isomorphism.
Using the associativity of multiplication, we can rewrite various products in $F\wr\cal{H}_n$ as linear combinations in the image of $\varphi_n$ in several ways, by applying the rules $\mathbf{Q}^i,\mathbf{B}_2^{ij},\mathbf{B}_3^{i}$ in different order.
If the map $\phi_n$ is injective, such rewritings must produce the same result.
We list the conditions obtained in this way below, in diagrammatic terms, and refer the reader to~\cite[App.~B.1]{lai2024quantum} for their derivation, which carries over essentially verbatim.
\begin{equation}\label{diag:TTT}\tag{D3}
	\tikz[thick,xscale=.5,yscale=.5,font=\footnotesize,baseline=(current  bounding  box.center)]{
		\morSi{0}{0}{1}{1}
		\draw [dotted,line cap=round,rounded corners,color=OI5] (2,0) -- (2,1);
		\morRi{1}{1}{2}{2}
		\draw (0,1) -- (0,2);
		\dmerge{0}{2}{1.5}{2.75}
		\draw (0.75,2.75) -- (0.75,3);
		\ntxt{2.55}{1.5}{$=$}

		\begin{scope}[xshift=3.1cm]
		\morRi{1}{0}{2}{1}
		\draw [dotted,line cap=round,rounded corners,color=OI5] (0,0) -- (0,1);
		\rhoact[OI5]{0}{1}{1.5}{2}
		\draw (0.75,2) -- (0.75,3);
		\ntxt{2.3}{1.5}{$+$}
		\end{scope}

		\begin{scope}[xshift=6.2cm]
		\morSi{1}{0}{2}{0.75}
		\draw [dotted,line cap=round,rounded corners,color=OI5] (0,0) -- (0,0.75);
		\sigact[OI5]{0}{0.75}{1}{1.5}
		\draw [dotted,line cap=round,rounded corners,color=OI5] (2,0.75) -- (2,1.5);
		\morRi{1}{1.5}{2}{3*0.75}
		\draw (0,1.5) -- (0,3*0.75);
		\dmerge{0}{3*0.75}{1.5}{3.75*0.75}
		\draw (0.75,3.75*0.75) -- (0.75,3);
		\ntxt{2.5}{1.4}{,}
		\end{scope}
	}\qquad
	\tikz[thick,xscale=.5,yscale=.5,font=\footnotesize,baseline=(current  bounding  box.center)]{
		\morSi{0}{0}{1}{1}
		\draw [dotted,line cap=round,rounded corners,color=OI5] (2,0) -- (2,1);
		\morSi{1}{1}{2}{2}
		\draw (0,1) -- (0,2);
		\dmerge{0}{2}{1}{2.75}
		\draw (0.5,2.75) -- (0.5,3);
		\draw [dotted,line cap=round,rounded corners,color=OI5] (2,2) -- (2,3);
		\ntxt{2.55}{1.5}{$+$}

		\begin{scope}[xshift=3.1cm]
		\morRi{0}{0}{1}{2.5}
		\draw [dotted,line cap=round,rounded corners,color=OI5] (2,0) -- (2,3);
		\draw (0.5,2.5) -- (0.5,3);
		\ntxt{2.55}{1.5}{$=$}
		\end{scope}

		\begin{scope}[xshift=6.2cm]
		\morSi{1}{0}{2}{1}
		\draw [dotted,line cap=round,rounded corners,color=OI5] (0,0) -- (0,1);
		\rhoact[OI5]{0}{1}{1}{2.5}
		\draw (0.5,2.5) -- (0.5,3);
		\draw [dotted,line cap=round,rounded corners,color=OI5] (2,1) -- (2,3);
		\ntxt{2.55}{1.5}{$+$}
		\end{scope}

		\begin{scope}[xshift=9.3cm]
		\morSi{1}{0}{2}{0.75}
		\draw [dotted,line cap=round,rounded corners,color=OI5] (0,0) -- (0,0.75);
		\sigact[OI5]{0}{0.75}{1}{1.5}
		\draw [dotted,line cap=round,rounded corners,color=OI5] (2,0.75) -- (2,1.5);
		\morSi{1}{1.5}{2}{3*0.75}
		\draw (0,1.5) -- (0,3*0.75);
		\dmerge{0}{3*0.75}{1}{3.75*0.75}
		\draw (0.5,3.75*0.75) -- (0.5,3);
		\draw [dotted,line cap=round,rounded corners,color=OI5] (2,2.25) -- (2,3);
		\ntxt{2.55}{1.5}{$+$}
		\end{scope}

		\begin{scope}[xshift=12.4cm]
		\morRi{1}{0}{2}{1.5}
		\draw [dotted,line cap=round,rounded corners,color=OI5] (0,0) -- (0,1.5);
		\sigact[OI5]{0}{1.5}{1.5}{3}
		\ntxt{2.5}{1.4}{,}
		\end{scope}
	}
\end{equation}

\begin{equation}\label{diag:qu1}\tag{D4}
	\tikz[thick,xscale=.5,yscale=.5,font=\footnotesize,baseline=(current  bounding  box.center)]{
		\morSi{0}{0}{1}{1}
		\draw (2,0) -- (2,1);
		\sigact[OI5]{1}{1}{2}{2}
		\draw (0,1) -- (0,2);
		\dmerge{0}{2}{1}{2.75}
		\draw (0.5,2.75) -- (0.5,3);
		\draw [dotted,line cap=round,rounded corners,color=OI5] (2,2) -- (2,3);
		\ntxt{2.55}{1.5}{$=$}

		\begin{scope}[xshift=3.1cm]
		\draw [dotted,line cap=round,rounded corners,color=OI5] (0,0) -- (0,1.5);
		\rhoact[OI5]{1}{0}{2}{1.5}
		\sigact[OI5]{0}{1.5}{1.5}{3}
		
		\ntxt{2.55}{1.5}{$+$}
		\end{scope}

		\begin{scope}[xshift=6.2cm]
		\draw [dotted,line cap=round,rounded corners,color=OI5] (0,0) -- (0,1.5);
		\sigact[OI5]{1}{0}{2}{1.5}
		\rhoact[OI5]{0}{1.5}{1}{3}
		\draw [dotted,line cap=round,rounded corners,color=OI5] (2,1.5) -- (2,3);
		\ntxt{2.55}{1.5}{$+$}
		\end{scope}

		\begin{scope}[xshift=9.3cm]
		\sigact[OI5]{1}{0}{2}{0.75}
		\draw [dotted,line cap=round,rounded corners,color=OI5] (0,0) -- (0,0.75);
		\sigact[OI5]{0}{0.75}{1}{1.5}
		\draw [dotted,line cap=round,rounded corners,color=OI5] (2,0.75) -- (2,1.5);
		\morSi{1}{1.5}{2}{3*0.75}
		\draw (0,1.5) -- (0,3*0.75);
		\dmerge{0}{3*0.75}{1}{3.75*0.75}
		\draw (0.5,3.75*0.75) -- (0.5,3);
		\draw [dotted,line cap=round,rounded corners,color=OI5] (2,2.25) -- (2,3);
		\ntxt{2.3}{1.4}{,}
		\end{scope}
	}\quad 
	\tikz[thick,xscale=.5,yscale=.5,font=\footnotesize,baseline=(current  bounding  box.center)]{
		\morSi{0}{0}{1}{1}
		\draw (2,0) -- (2,1);
		\rhoact[OI5]{1}{1}{2}{2}
		\draw (0,1) -- (0,2);
		\dmerge{0}{2}{1.5}{2.75}
		\draw (0.75,2.75) -- (0.75,3);
		\ntxt{2.75}{1.5}{$+$}

		\begin{scope}[xshift=3.5cm]
		\draw (2,0) -- (2,2);
		\morRi{0}{0}{1}{2}
		\dmerge{0.5}{2}{2}{2.75}
		\draw (1.25,2.75) -- (1.25,3);
		\ntxt{2.75}{1.5}{$=$}
		\end{scope}

		\begin{scope}[xshift=7cm]
		\draw [dotted,line cap=round,rounded corners,color=OI5] (0,0) -- (0,1.5);
		\rhoact[OI5]{1}{0}{2}{1.5}
		\rhoact[OI5]{0}{1.5}{1.5}{3}
		\ntxt{2.75}{1.5}{$+$}
		\end{scope}

		\begin{scope}[xshift=10.5cm]
		\sigact[OI5]{1}{0}{2}{0.75}
		\draw [dotted,line cap=round,rounded corners,color=OI5] (0,0) -- (0,0.75);
		\sigact[OI5]{0}{0.75}{1}{1.5}
		\draw [dotted,line cap=round,rounded corners,color=OI5] (2,0.75) -- (2,1.5);
		\morRi{1}{1.5}{2}{3*0.75}
		\draw (0,1.5) -- (0,3*0.75);
		\dmerge{0}{3*0.75}{1.5}{3.75*0.75}
		\draw (0.75,3.75*0.75) -- (0.75,3);
		\ntxt{2.3}{1.4}{,}
		\end{scope}
	}
\end{equation}

\begin{equation}\label{diag:br1}\tag{D5}
	\tikz[thick,xscale=.42,yscale=.42,font=\footnotesize,baseline=(current  bounding  box.center)]{
		\draw [dotted,line cap=round,rounded corners,color=OI5] (0,0) -- (0,1.5);
		\draw [dotted,line cap=round,rounded corners,color=OI1] (1,0) -- (1,0.75);
		\sigact[OI5]{2}{0}{3}{0.75}
		\sigact[OI1]{1}{0.75}{2}{1.5}
		\draw [dotted,line cap=round,rounded corners,color=OI1] (2,1.5) -- (2,2.25);
		\draw [dotted,line cap=round,rounded corners,color=OI5] (3,0.75) -- (3,2.25);
		\sigact[OI5]{0}{1.5}{1}{2.25}
		\draw (0,2.25) -- (0,4);
		\morTi{1}{2.25}{3}{4}
		\ntxt{3.75}{2}{$=$}

		\begin{scope}[xshift=4.5cm]
		\draw (3,0) -- (3,1.75);
		\morTi{0}{0}{2}{1.75}
		\draw [dotted,line cap=round,rounded corners,color=OI1] (0,1.75) -- (0,3.25);
		\draw [dotted,line cap=round,rounded corners,color=OI5] (1,1.75) -- (1,2.5);
		\sigact[OI1]{2}{1.75}{3}{2.5}
		\sigact[OI5]{1}{2.5}{2}{3.25}
		\draw [dotted,line cap=round,rounded corners,color=OI5] (2,3.25) -- (2,4);
		\draw [dotted,line cap=round,rounded corners,color=OI1] (3,2.5) -- (3,4);
		\sigact[OI1]{0}{3.25}{1}{4}
		\ntxt{3.5}{1.8}{,}
		\end{scope}
	}\qquad
	\tikz[thick,xscale=.42,yscale=.42,font=\footnotesize,baseline=(current  bounding  box.center)]{
		\draw [dotted,line cap=round,rounded corners,color=OI5] (0,0) -- (0,2);
		\draw [dotted,line cap=round,rounded corners,color=OI1] (1,0) -- (1,1);
		\rhoact[OI5]{2}{0}{3}{1}
		\sigact[OI1]{1}{1}{2.5}{2}
		\draw [dotted,line cap=round,rounded corners,color=OI1] (2.5,2) -- (2.5,4);
		\sigact[OI5]{0}{2}{1}{3}
		\draw [dotted,line cap=round,rounded corners,color=OI5] (1,3) -- (1,4);
		\draw (0,3) -- (0,4);
		\ntxt{3.75}{2}{$=$}

		\begin{scope}[xshift=4.5cm]
		\draw (3,0) -- (3,1.75);
		\morTi{0}{0}{2}{1.75}
		\draw [dotted,line cap=round,rounded corners,color=OI1] (0,1.75) -- (0,3.25);
		\draw [dotted,line cap=round,rounded corners,color=OI5] (1,1.75) -- (1,2.5);
		\sigact[OI1]{2}{1.75}{3}{2.5}
		\sigact[OI5]{1}{2.5}{2}{3.25}
		\draw [dotted,line cap=round,rounded corners,color=OI5] (2,3.25) -- (2,4);
		\draw [dotted,line cap=round,rounded corners,color=OI1] (3,2.5) -- (3,4);
		\rhoact[OI1]{0}{3.25}{1}{4}
		\ntxt{3.5}{1.8}{,}
		\end{scope}
	}\qquad
	\tikz[thick,xscale=.42,yscale=.42,font=\footnotesize,baseline=(current  bounding  box.center)]{
		\draw [dotted,line cap=round,rounded corners,color=OI5] (0,0) -- (0,2);
		\draw [dotted,line cap=round,rounded corners,color=OI1] (1,0) -- (1,1);
		\sigact[OI5]{2}{0}{3}{1}
		\sigact[OI1]{1}{1}{2}{2}
		\draw [dotted,line cap=round,rounded corners,color=OI1] (2,2) -- (2,4);
		\draw [dotted,line cap=round,rounded corners,color=OI5] (3,1) -- (3,4);
		\rhoact[OI5]{0}{2}{1}{3}
		\draw (0.5,3) -- (0.5,4);
		\ntxt{3.75}{2}{$=$}

		\begin{scope}[xshift=4.5cm]
		\draw (3,0) -- (3,1.75);
		\morTi{0}{0}{2}{1.75}
		\draw [dotted,line cap=round,rounded corners,color=OI1] (0,1.75) -- (0,3.25);
		\draw [dotted,line cap=round,rounded corners,color=OI5] (1,1.75) -- (1,2.5);
		\rhoact[OI1]{2}{1.75}{3}{2.5}
		\sigact[OI5]{1}{2.5}{2.5}{3.25}
		\draw [dotted,line cap=round,rounded corners,color=OI5] (2.5,3.25) -- (2.5,4);
		\sigact[OI1]{0}{3.25}{1}{4}
		\ntxt{3.5}{1.8}{,}
		\end{scope}
	}
\end{equation}

\begin{equation}\label{diag:br2}\tag{D6}
	\tikz[thick,xscale=.42,yscale=.42,font=\footnotesize,baseline=(current  bounding  box.center)]{
		\draw [dotted,line cap=round,rounded corners,color=OI5] (0,0) -- (0,2);
		\draw [dotted,line cap=round,rounded corners,color=OI1] (1,0) -- (1,1);
		\rhoact[OI5]{2}{0}{3}{1}
		\sigact[OI1]{1}{1}{2.5}{2}
		\draw [dotted,line cap=round,rounded corners,color=OI1] (2.5,2) -- (2.5,4);
		\rhoact[OI5]{0}{2}{1}{3}
		\draw (0.5,3) -- (0.5,4);
		\ntxt{3.55}{2}{$=$}

		\begin{scope}[xshift=4.1cm]
		\draw (3,0) -- (3,1.75);
		\morTi{0}{0}{2}{1.75}
		\draw [dotted,line cap=round,rounded corners,color=OI1] (0,1.75) -- (0,3.25);
		\draw [dotted,line cap=round,rounded corners,color=OI5] (1,1.75) -- (1,2.5);
		\rhoact[OI1]{2}{1.75}{3}{2.5}
		\rhoact[OI5]{1}{2.5}{2.5}{3.25}
		\sigact[OI1]{0}{3.25}{1.75}{4}
		\ntxt{3.55}{2}{$+$}
		\end{scope}

		\begin{scope}[xshift=8.2cm]
		\draw (3,0) -- (3,1.75);
		\morTi{0}{0}{2}{1.75}
		\draw [dotted,line cap=round,rounded corners,color=OI1] (0,1.75) -- (0,3.25);
		\draw [dotted,line cap=round,rounded corners,color=OI5] (1,1.75) -- (1,2.5);
		\sigact[OI1]{2}{1.75}{3}{2.5}
		\rhoact[OI5]{1}{2.5}{2}{3.25}
		\draw [dotted,line cap=round,rounded corners,color=OI1] (3,2.5) -- (3,4);
		\rhoact[OI1]{0}{3.25}{1.5}{4}
		\ntxt{3.55}{2}{$+$}
		\end{scope}

		\begin{scope}[xshift=12.3cm]
		\draw (3,0) -- (3,1.25);
		\morTi{0}{0}{2}{1.25}
		\draw [dotted,line cap=round,rounded corners,color=OI1] (0,1.25) -- (0,2.25);
		\draw [dotted,line cap=round,rounded corners,color=OI5] (1,1.25) -- (1,1.75);
		\sigact[OI1]{2}{1.25}{3}{1.75}
		\rhoact[OI5]{1}{1.75}{2}{2.25}
		\draw [dotted,line cap=round,rounded corners,color=OI1] (3,1.75) -- (3,2.75);
		\sigact[OI1]{0}{2.25}{1.5}{2.75}
		\draw (0,2.75) -- (0,3.5);
		\morSi[OI1]{1.5}{2.75}{3}{3.5}
		\dmerge{0}{3.5}{1.5}{4}
		\draw [dotted,line cap=round,rounded corners,color=OI1] (3,3.5) -- (3,4);
		\ntxt{3.3}{1.8}{,}
		\end{scope}
	}\quad
	\tikz[thick,xscale=.42,yscale=.42,font=\footnotesize,baseline=(current  bounding  box.center)]{
		\draw (3,0) -- (3,1.75);
		\morTi{0}{0}{2}{1.75}
		\draw [dotted,line cap=round,rounded corners,color=OI1] (0,1.75) -- (0,3.25);
		\draw [dotted,line cap=round,rounded corners,color=OI5] (1,1.75) -- (1,2.5);
		\rhoact[OI1]{2}{1.75}{3}{2.5}
		\sigact[OI5]{1}{2.5}{2.5}{3.25}
		\draw [dotted,line cap=round,rounded corners,color=OI5] (2.5,3.25) -- (2.5,4);
		\rhoact[OI1]{0}{3.25}{1}{4}
		\ntxt{3.55}{2}{$=$}

		\begin{scope}[xshift=4.1cm]
		\draw [dotted,line cap=round,rounded corners,color=OI5] (0,0) -- (0,2);
		\draw [dotted,line cap=round,rounded corners,color=OI1] (1,0) -- (1,1);
		\rhoact[OI5]{2}{0}{3}{1}
		\rhoact[OI1]{1}{1}{2.5}{2}
		\sigact[OI5]{0}{2}{1.75}{3}
		\draw (0,3) -- (0,4);
		\draw [dotted,line cap=round,rounded corners,color=OI5] (1.75,3) -- (1.75,4);
		\ntxt{3.55}{2}{$+$}
		\end{scope}

		\begin{scope}[xshift=8.2cm]
		\draw [dotted,line cap=round,rounded corners,color=OI5] (0,0) -- (0,2);
		\draw [dotted,line cap=round,rounded corners,color=OI1] (1,0) -- (1,1);
		\sigact[OI5]{2}{0}{3}{1}
		\rhoact[OI1]{1}{1}{2}{2}
		\draw [dotted,line cap=round,rounded corners,color=OI5] (3,1) -- (3,4);
		\rhoact[OI5]{0}{2}{1.5}{3}
		\draw (0.75,3) -- (0.75,4);
		\ntxt{3.55}{2}{$+$}
		\end{scope}

		\begin{scope}[xshift=12.3cm]
		\draw [dotted,line cap=round,rounded corners,color=OI5] (0,0) -- (0,1.4);
		\draw [dotted,line cap=round,rounded corners,color=OI1] (1,0) -- (1,0.7);
		\sigact[OI5]{2}{0}{3}{0.7}
		\rhoact[OI1]{1}{0.7}{2}{1.4}
		\draw [dotted,line cap=round,rounded corners,color=OI5] (3,0.7) -- (3,2.1);
		\sigact[OI5]{0}{1.4}{1.5}{2.1}
		\draw (0,2.1) -- (0,3.5);
		\morSi{1.5}{2.1}{3}{3.5}
		\draw [dotted,line cap=round,rounded corners,color=OI5] (3,3.5) -- (3,4);
		\dmerge{0}{3.5}{1.5}{4}
		\ntxt{3.3}{1.8}{,}
		\end{scope}
	}
\end{equation}

\begin{equation}\label{diag:br3}\tag{D7}
	\tikz[thick,xscale=.42,yscale=.42,font=\footnotesize,baseline=(current  bounding  box.center)]{
		\draw (3,0) -- (3,1.75);
		\morTi{0}{0}{2}{1.75}
		\draw [dotted,line cap=round,rounded corners,color=OI1] (0,1.75) -- (0,3.25);
		\draw [dotted,line cap=round,rounded corners,color=OI5] (1,1.75) -- (1,2.5);
		\rhoact[OI1]{2}{1.75}{3}{2.5}
		\rhoact[OI5]{1}{2.5}{2.5}{3.25}
		\rhoact[OI1]{0}{3.25}{1.75}{4}
		\ntxt{3.75}{2}{$+$}

		\begin{scope}[xshift=4.5cm]
		\draw (3,0) -- (3,1.25);
		\morTi{0}{0}{2}{1.25}
		\draw [dotted,line cap=round,rounded corners,color=OI1] (0,1.25) -- (0,2.25);
		\draw [dotted,line cap=round,rounded corners,color=OI5] (1,1.25) -- (1,1.75);
		\sigact[OI1]{2}{1.25}{3}{1.75}
		\rhoact[OI5]{1}{1.75}{2}{2.25}
		\draw [dotted,line cap=round,rounded corners,color=OI1] (3,1.75) -- (3,2.75);
		\sigact[OI1]{0}{2.25}{1.5}{2.75}
		\draw (0,2.75) -- (0,3.5);
		\morRi[OI1]{1.5}{2.75}{3}{3.5}
		\dmerge{0}{3.5}{2.25}{4}
		\ntxt{3.75}{2}{$=$}
		\end{scope}

		\begin{scope}[xshift=9cm]
		\draw [dotted,line cap=round,rounded corners,color=OI5] (0,0) -- (0,2);
		\draw [dotted,line cap=round,rounded corners,color=OI1] (1,0) -- (1,1);
		\rhoact[OI5]{2}{0}{3}{1}
		\rhoact[OI1]{1}{1}{2.5}{2}
		\rhoact[OI5]{0}{2}{1.75}{3}
		\draw (0.875,3) -- (0.875,4);
		\ntxt{3.75}{2}{$+$}
		\end{scope}

		\begin{scope}[xshift=13.5cm]
		\draw [dotted,line cap=round,rounded corners,color=OI5] (0,0) -- (0,1.4);
		\draw [dotted,line cap=round,rounded corners,color=OI1] (1,0) -- (1,0.7);
		\sigact[OI5]{2}{0}{3}{0.7}
		\rhoact[OI1]{1}{0.7}{2}{1.4}
		\draw [dotted,line cap=round,rounded corners,color=OI5] (3,0.7) -- (3,2.1);
		\sigact[OI5]{0}{1.4}{1.5}{2.1}
		\draw (0,2.1) -- (0,3.5);
		\morRi{1.5}{2.1}{3}{3.5}
		\dmerge{0}{3.5}{2.25}{4}
		\ntxt{3.5}{1.8}{,}
		\end{scope}
	}
\end{equation}

\begin{equation}
\label{diag:br4}\tag{D8}
\begin{gathered}
	\tikz[thick,xscale=.42,yscale=.42,font=\footnotesize,baseline=(current  bounding  box.center)]{
		\morSi[OI5]{0}{0}{1}{2}
		\draw [dotted,line cap=round,rounded corners,color=OI1] (2,0) -- (2,2);
		\draw [dotted,line cap=round,rounded corners,color=OI5] (3,0) -- (3,2);
		\morTi{1}{2}{3}{4}
		\draw (0,2) -- (0,4);
		\ntxt{3.55}{2}{$=$}

		\begin{scope}[xshift=4.1cm]
		\draw [dotted,line cap=round,rounded corners,color=OI5] (0,0) -- (0,1.25);
		\morTi{1}{0}{3}{1.25}
		\morTi{0}{1.25}{2}{2.5}
		\draw [dotted,line cap=round,rounded corners,color=OI1] (3,1.25) -- (3,2.5);
		\draw [dotted,line cap=round,rounded corners,color=OI1] (0,2.5) -- (0,3.5);
		\draw [dotted,line cap=round,rounded corners,color=OI5] (1,2.5) -- (1,3);
		\draw [dotted,line cap=round,rounded corners,color=OI5] (2,3.5) -- (2,4);
		\morSi[OI1]{2}{2.5}{3}{3}
		\sigact[OI5]{1}{3}{2}{3.5}
		\sigact[OI1]{0}{3.5}{1}{4}
		\draw [dotted,line cap=round,rounded corners,color=OI1] (3,3) -- (3,4);
		\ntxt{3.3}{1.8}{,}
		\end{scope}
	}\enspace
	\tikz[thick,xscale=.42,yscale=.42,font=\footnotesize,baseline=(current  bounding  box.center)]{
		\draw [dotted,line cap=round,rounded corners,color=OI5] (3,0) -- (3,1.5);
		\morTi{0}{0}{2}{1.5}
		\morTi{1}{1.5}{3}{3}
		\draw [dotted,line cap=round,rounded corners,color=OI1] (0,1.5) -- (0,3);
		\morSi[OI1]{0}{3}{1}{4}
		\draw [dotted,line cap=round,rounded corners,color=OI5] (2,3) -- (2,4);
		\draw [dotted,line cap=round,rounded corners,color=OI1] (3,3) -- (3,4);
		\ntxt{3.55}{2}{$=$}

		\begin{scope}[xshift=4.1cm]
		\draw [dotted,line cap=round,rounded corners,color=OI5] (0,0) -- (0,1.75);
		\draw [dotted,line cap=round,rounded corners,color=OI1] (1,0) -- (1,1);
		\morSi[OI5]{2}{0}{3}{1}
		\sigact[OI1]{1}{1}{2}{1.75}
		\draw [dotted,line cap=round,rounded corners,color=OI1] (2,1.75) -- (2,2.5);
		\sigact[OI5]{0}{1.75}{1}{2.5}
		\draw [dotted,line cap=round,rounded corners,color=OI5] (3,1) -- (3,2.5);
		\draw (0,2.5) -- (0,4);
		\morTi{1}{2.5}{3}{4}
		\ntxt{3.3}{1.8}{,}
		\end{scope}
	}\enspace
	\tikz[thick,xscale=.42,yscale=.42,font=\footnotesize,baseline=(current  bounding  box.center)]{
		\draw [dotted,line cap=round,rounded corners,color=OI5] (3,0) -- (3,1.5);
		\morTi{0}{0}{2}{1.5}
		\morTi{1}{1.5}{3}{3}
		\draw [dotted,line cap=round,rounded corners,color=OI1] (0,1.5) -- (0,3);
		\morRi[OI1]{0}{3}{1}{4}
		\draw [dotted,line cap=round,rounded corners,color=OI5] (2,3) -- (2,4);
		\draw [dotted,line cap=round,rounded corners,color=OI1] (3,3) -- (3,4);
		\ntxt{3.55}{2}{$=$}

		\begin{scope}[xshift=4.1cm]
		\draw [dotted,line cap=round,rounded corners,color=OI5] (0,0) -- (0,2);
		\draw [dotted,line cap=round,rounded corners,color=OI1] (1,0) -- (1,1);
		\morRi[OI5]{2}{0}{3}{1}
		\sigact[OI1]{1}{1}{2.5}{2}
		\draw [dotted,line cap=round,rounded corners,color=OI1] (2.5,2) -- (2.5,4);
		\sigact[OI5]{0}{2}{1}{3}
		\draw (0,3) -- (0,4);
		\draw [dotted,line cap=round,rounded corners,color=OI5] (1,3) -- (1,4);
		\ntxt{3.3}{1.8}{,}
		\end{scope}
	}\enspace
	\tikz[thick,xscale=.42,yscale=.42,font=\footnotesize,baseline=(current  bounding  box.center)]{
		\morRi[OI5]{0}{0}{1}{3}
		\draw [dotted,line cap=round,rounded corners,color=OI1] (2,0) -- (2,4);
		\draw [dotted,line cap=round,rounded corners,color=OI5] (3,0) -- (3,4);
		\draw (0.5,3) -- (0.5,4);
		\ntxt{3.55}{2}{$=$}

		\begin{scope}[xshift=4.1cm]
		\draw [dotted,line cap=round,rounded corners,color=OI5] (0,0) -- (0,1.25);
		\morTi{1}{0}{3}{1.25}
		\morTi{0}{1.25}{2}{2.5}
		\draw [dotted,line cap=round,rounded corners,color=OI1] (3,1.25) -- (3,2.5);
		\draw [dotted,line cap=round,rounded corners,color=OI1] (0,2.5) -- (0,3.5);
		\draw [dotted,line cap=round,rounded corners,color=OI5] (1,2.5) -- (1,3);
		\draw [dotted,line cap=round,rounded corners,color=OI5] (2.5,3.5) -- (2.5,4);
		\morRi[OI1]{2}{2.5}{3}{3}
		\sigact[OI5]{1}{3}{2.5}{3.5}
		\sigact[OI1]{0}{3.5}{1}{4}
		\ntxt{3.3}{1.8}{,}
		\end{scope}
	}
\\
	\tikz[thick,xscale=.42,yscale=.42,font=\footnotesize,baseline=(current  bounding  box.center)]{
		\draw [dotted,line cap=round,rounded corners,color=OI5] (0,0) -- (0,1.25);
		\morTi{1}{0}{3}{1.25}
		\morTi{0}{1.25}{2}{2.5}
		\draw [dotted,line cap=round,rounded corners,color=OI1] (3,1.25) -- (3,2.5);
		\draw [dotted,line cap=round,rounded corners,color=OI1] (0,2.5) -- (0,3.5);
		\draw [dotted,line cap=round,rounded corners,color=OI5] (1,2.5) -- (1,3);
		\draw [dotted,line cap=round,rounded corners,color=OI5] (2,3.5) -- (2,4);
		\morSi[OI1]{2}{2.5}{3}{3}
		\sigact[OI5]{1}{3}{2}{3.5}
		\rhoact[OI1]{0}{3.5}{1}{4}
		\draw [dotted,line cap=round,rounded corners,color=OI1] (3,3) -- (3,4);
		\ntxt{3.75}{2}{$=$}

		\begin{scope}[xshift=4.5cm]
		\morSi[OI5]{2}{0}{3}{1}
		\draw [dotted,line cap=round,rounded corners,color=OI5] (0,0) -- (0,2);
		\draw [dotted,line cap=round,rounded corners,color=OI1] (1,0) -- (1,1);
		\sigact[OI1]{1}{1}{2}{2}
		\rhoact[OI5]{0}{2}{1}{3}
		\draw (0.5,3) -- (0.5,4);
		\draw [dotted,line cap=round,rounded corners,color=OI1] (2,2) -- (2,4);
		\draw [dotted,line cap=round,rounded corners,color=OI5] (3,1) -- (3,4);
		\ntxt{3.75}{2}{$=$}
		\end{scope}

		\begin{scope}[xshift=9cm]
		\draw [dotted,line cap=round,rounded corners,color=OI5] (0,0) -- (0,3);
		\draw [dotted,line cap=round,rounded corners,color=OI1] (1,0) -- (1,2);
		\morRi[OI5]{2}{0}{3}{2}
		\sigact[OI1]{1}{2}{2.5}{3}
		\draw [dotted,line cap=round,rounded corners,color=OI1] (2.5,3) -- (2.5,4);
		\rhoact[OI5]{0}{3}{1}{4}
		\ntxt{3.75}{2}{$=$}
		\end{scope}

		\begin{scope}[xshift=13.5cm]
		\draw [dotted,line cap=round,rounded corners,color=OI5] (0,0) -- (0,1.2);
		\morTi{1}{0}{3}{1.2}
		\draw [dotted,line cap=round,rounded corners,color=OI1] (3,1.2) -- (3,2.4);
		\morTi{0}{1.2}{2}{2.4}
		\morRi[OI1]{2}{2.4}{3}{3}
		\draw [dotted,line cap=round,rounded corners,color=OI1] (0,2.4) -- (0,3.5);
		\draw [dotted,line cap=round,rounded corners,color=OI1] (1,2.4) -- (1,3);
		\sigact[OI5]{1}{3}{2.5}{3.5}
		\draw [dotted,line cap=round,rounded corners,color=OI5] (2.5,3.5) -- (2.5,4);
		\rhoact[OI1]{0}{3.5}{1}{4}
		\ntxt{3.75}{2}{$=0,$}
		\end{scope}
	}
\end{gathered}
\end{equation}

\begin{equation}
\label{diag:br5}\tag{D9}
\begin{gathered}
	\tikz[thick,xscale=.42,yscale=.42,font=\footnotesize,baseline=(current  bounding  box.center)]{
		\draw [dotted,line cap=round,rounded corners,color=OI5] (0,0) -- (0,1.2);
		\morTi{1}{0}{3}{1.2}
		\morTi{0}{1.2}{2}{2.4}
		\draw [dotted,line cap=round,rounded corners,color=OI1] (3,1.2) -- (3,2.4);
		\draw [dotted,line cap=round,rounded corners,color=OI1] (0,2.4) -- (0,3.5);
		\draw [dotted,line cap=round,rounded corners,color=OI5] (1,2.4) -- (1,3);
		\morSi[OI1]{2}{2.4}{3}{3}
		\rhoact[OI5]{1}{3}{2}{3.5}
		\rhoact[OI1]{0}{3.5}{1.5}{4}
		\draw [dotted,line cap=round,rounded corners,color=OI1] (3,3) -- (3,4);
		\ntxt{3.75}{2}{$+$}

		\begin{scope}[xshift=4.5cm]
		\draw [dotted,line cap=round,rounded corners,color=OI5] (0,0) -- (0,1);
		\morTi{1}{0}{3}{1}
		\morTi{0}{1}{2}{2}
		\draw [dotted,line cap=round,rounded corners,color=OI1] (3,1) -- (3,2);
		\draw [dotted,line cap=round,rounded corners,color=OI1] (0,2) -- (0,2.8);
		\draw [dotted,line cap=round,rounded corners,color=OI5] (1,2) -- (1,2.4);
		\morSi[OI1]{2}{2}{3}{2.4}
		\rhoact[OI5]{1}{2.4}{2}{2.8}
		\sigact[OI1]{0}{2.8}{1.5}{3.2}
		\draw [dotted,line cap=round,rounded corners,color=OI1] (3,2.4) -- (3,3.2);
		\draw (0,3.2) -- (0,3.6);
		\morSi[OI1]{1.5}{3.2}{3}{3.6}
		\dmerge{0}{3.6}{1.5}{4}
		\draw [dotted,line cap=round,rounded corners,color=OI1] (3,3.6) -- (3,4);
		\ntxt{3.75}{2}{$+$}
		\end{scope}

		\begin{scope}[xshift=9cm]
		\draw [dotted,line cap=round,rounded corners,color=OI5] (0,0) -- (0,1.2);
		\morTi{1}{0}{3}{1.2}
		\morTi{0}{1.2}{2}{2.4}
		\draw [dotted,line cap=round,rounded corners,color=OI1] (3,1.2) -- (3,2.4);
		\draw [dotted,line cap=round,rounded corners,color=OI1] (0,2.4) -- (0,3.5);
		\draw [dotted,line cap=round,rounded corners,color=OI5] (1,2.4) -- (1,3);
		\morRi[OI1]{2}{2.4}{3}{3}
		\rhoact[OI5]{1}{3}{2.5}{3.5}
		\sigact[OI1]{0}{3.5}{1.75}{4}
		\ntxt{3.75}{2}{$=$}
		\end{scope}

		\begin{scope}[xshift=13.5cm]
		\draw [dotted,line cap=round,rounded corners,color=OI5] (0,0) -- (0,3);
		\draw [dotted,line cap=round,rounded corners,color=OI1] (1,0) -- (1,2);
		\morSi[OI5]{2}{0}{3}{2}
		\draw [dotted,line cap=round,rounded corners,color=OI5] (3,2) -- (3,4);
		\rhoact[OI1]{1}{2}{2}{3}
		\rhoact[OI5]{0}{3}{1.5}{4}
		\ntxt{3.75}{2}{$+$}
		\end{scope}

		\begin{scope}[xshift=18cm]
		\draw [dotted,line cap=round,rounded corners,color=OI5] (0,0) -- (0,1.6);
		\draw [dotted,line cap=round,rounded corners,color=OI1] (1,0) -- (1,0.8);
		\morSi[OI5]{2}{0}{3}{0.8}
		\draw [dotted,line cap=round,rounded corners,color=OI5] (3,0.8) -- (3,2.4);
		\rhoact[OI1]{1}{0.8}{2}{1.6}
		\sigact[OI5]{0}{1.6}{1.5}{2.4}
		\draw (0,2.4) -- (0,3.2);
		\morSi[OI5]{1.5}{2.4}{3}{3.2}
		\dmerge{0}{3.2}{1.5}{3.7}
		\draw (0.75,3.7) -- (0.75,4);
		\draw [dotted,line cap=round,rounded corners,color=OI5] (3,3.2) -- (3,4);
		\ntxt{3.75}{2}{$+$}
		\end{scope}

		\begin{scope}[xshift=22.5cm]
		\draw [dotted,line cap=round,rounded corners,color=OI5] (0,0) -- (0,3);
		\draw [dotted,line cap=round,rounded corners,color=OI1] (1,0) -- (1,2);
		\morRi[OI5]{2}{0}{3}{2}
		\rhoact[OI1]{1}{2}{2.5}{3}
		\sigact[OI5]{0}{3}{1.75}{4}
		\ntxt{3.75}{2}{$=0,$}
		\end{scope}
	}
\\
	\tikz[thick,xscale=.42,yscale=.42,font=\footnotesize,baseline=(current  bounding  box.center)]{
		\draw [dotted,line cap=round,rounded corners,color=OI5] (0,0) -- (0,1.2);
		\morTi{1}{0}{3}{1.2}
		\morTi{0}{1.2}{2}{2.4}
		\draw [dotted,line cap=round,rounded corners,color=OI1] (3,1.2) -- (3,2.4);
		\draw [dotted,line cap=round,rounded corners,color=OI1] (0,2.4) -- (0,3.5);
		\draw [dotted,line cap=round,rounded corners,color=OI5] (1,2.4) -- (1,3);
		\morRi[OI1]{2}{2.4}{3}{3}
		\rhoact[OI5]{1}{3}{2.5}{3.5}
		\rhoact[OI1]{0}{3.5}{1.75}{4}
		\ntxt{3.75}{2}{$+$}

		\begin{scope}[xshift=4.5cm]
		\draw [dotted,line cap=round,rounded corners,color=OI5] (0,0) -- (0,1);
		\morTi{1}{0}{3}{1}
		\morTi{0}{1}{2}{2}
		\draw [dotted,line cap=round,rounded corners,color=OI1] (3,1) -- (3,2);
		\draw [dotted,line cap=round,rounded corners,color=OI1] (0,2) -- (0,2.8);
		\draw [dotted,line cap=round,rounded corners,color=OI5] (1,2) -- (1,2.4);
		\morSi[OI1]{2}{2}{3}{2.4}
		\rhoact[OI5]{1}{2.4}{2}{2.8}
		\sigact[OI1]{0}{2.8}{1.5}{3.2}
		\draw [dotted,line cap=round,rounded corners,color=OI1] (3,2.4) -- (3,3.2);
		\draw (0,3.2) -- (0,3.6);
		\morRi[OI1]{1.5}{3.2}{3}{3.6}
		\dmerge{0}{3.6}{2.25}{4}
		\ntxt{3.75}{2}{$=$}
		\end{scope}

		\begin{scope}[xshift=9cm]
		\draw [dotted,line cap=round,rounded corners,color=OI5] (0,0) -- (0,2);
		\draw [dotted,line cap=round,rounded corners,color=OI1] (1,0) -- (1,1);
		\morRi[OI5]{2}{0}{3}{1}
		\rhoact[OI1]{1}{1}{2.5}{2}
		\rhoact[OI5]{0}{2}{1.75}{4}
		\ntxt{3.75}{2}{$+$}
		\end{scope}

		\begin{scope}[xshift=13.5cm]
		\draw [dotted,line cap=round,rounded corners,color=OI5] (0,0) -- (0,1.5);
		\draw [dotted,line cap=round,rounded corners,color=OI1] (1,0) -- (1,1);
		\morSi[OI5]{2}{0}{3}{1}
		\rhoact[OI1]{1}{1}{2}{1.5}
		\sigact[OI5]{0}{1.5}{1.5}{2}
		\draw [dotted,line cap=round,rounded corners,color=OI5] (3,1) -- (3,2);
		\draw (0,2) -- (0,3);
		\morRi[OI5]{1.5}{2}{3}{3}
		\dmerge{0}{3}{2.25}{3.7}
		\draw (1.125,3.7) -- (1.125,4);
		\ntxt{3.75}{2}{$=0.$}
		\end{scope}
	}
\end{gathered}
\end{equation}
For completeness, here is the list of rewritings which give rise to these relations:
\begin{align*}
	\eqref{diag:TTT}&: \quad H_{ii}\otimes A = H_i\otimes (H_i\otimes A),\\ 
	\eqref{diag:qu1}&: \quad H_{ii}H_i = H_{iii} = H_iH_{ii},\\
	\eqref{diag:br1},\eqref{diag:br2},\eqref{diag:br3}&: \quad H_2\otimes (H_1\otimes (H_2\otimes A)) = H_{212}\otimes A \simeq H_{121} \otimes A = H_1\otimes (H_2\otimes (H_1\otimes A)),\\
	\eqref{diag:br4},\eqref{diag:br5}&: \quad H_{21}H_{22} = H_{2122} \simeq H_{1212} \simeq H_{1121} = H_{11}H_{21}, \quad H_{22}H_{12} \simeq H_{12}H_{11}.
\end{align*}

\subsection{Basis theorem}
Let us denote $\mathbb{V} = \bigoplus_{w\in \fkS_n} A\otimes H_w$.
We say that $F\wr \cal{H}_n$ \textit{has a PBW basis} if the map $\phi_n : \mathbb{V}\to F\wr \cal{H}_n$ is an isomorphism.
The following is the main result of~\cite{lai2024quantum}.

\begin{thm}\label{thm:QWP-basis}
	The algebra $F\wr \cal{H}_n$ has a PBW basis if and only if the conditions (\ref{diag:wr}--\ref{diag:br5}) hold.
\end{thm}
\begin{proof}[Sketch of proof]
	This theorem is proved in~\cite{lai2024quantum} for $\cal{C} = \Vec$, $H = \bbk$.
	The necessity of conditions (\ref{diag:wr}--\ref{diag:br5}) is explained above.
	For the sufficiency, the strategy in \textit{loc.~cit.} is as follows: 
	\begin{enumerate}[leftmargin=20pt,label={\arabic*.}]
		\item Let $\mathbb{V}^\ell = \bigoplus_{\ell(w)\leq \ell} A\otimes H_w\subset \mathbb{V}$. Define maps of left $A$-modules $f: \mathbb{V}\otimes A\to \mathbb{V}$, $T_i: \mathbb{V}\otimes H_i\to \mathbb{V}$ by induction on the length $\ell$;
		\item Check that these maps satisfy the defining relations of $F\wr \cal{H}_n$, and hence define a right $F\wr \cal{H}_n$-module structure on $\mathbb{V}$. By definition, this module is cyclic;
		\item Verify that the map $\phi_n$ is compatible with the right $F\wr \cal{H}_n$-action;
		\item As $\phi_n$ surjects from a cyclic module onto the regular one, it must be an isomorphism.
	\end{enumerate}
	This proof carries over almost verbatim to our generalized setup; only the second step requires some additional explanation.
	More precisely, in~\cite{lai2024quantum} the defining relations are broken down into six statements ($W[\ell]$), ($M[\ell]$), ($Q[\ell]$), ($B_2[\ell]$), ($B_3[\ell]$), ($R[\ell]$).
	Their verification is achieved via a ``grand loop'' argument, once again inductive on the length $\ell$, which is the computational heart of that paper.
	The argument for the first three statement remains identical in our generalized setting.
	The other three statements get slightly modified:
	\begin{align}
		\tag{$B_2[\ell]$} &T_i\circ (T_j\otimes \id_{H_i}) = T_j\circ (T_i\otimes \id_{H_j})\circ (\id_{\mathbb{V}}\otimes \tau_{H_j,H_i})\quad \text{ on $\mathbb{V}^\ell$};\\
		\tag{$B_3[\ell]$} &T_{i+1}\circ(T_{i}\otimes \id_{H_{i+1}})\circ(T_{i+1}\otimes \id_{H_{i,i+1}}) = T_{i}\circ(T_{i+1}\otimes \id_{H_{i}})\circ(T_{i}\otimes \id_{H_{i+1,i}}) \circ E_i \quad \text{ on $\mathbb{V}^\ell$};\\
		\tag{$R[\ell]$} &\text{$f|_{\mathbb{V}^\ell}$ does not depend on the choice of reduced expressions of $w\in \mathfrak{S}_n$.}
	\end{align}
	With these reformulations, the proofs of $(B_2[\ell])$ and $(B_3[\ell])$ once again remain identical.
	Finally, in order to check the statement $(R[\ell])$, one has to show the following equalities:
	\begin{align*}
		f&\circ(\tau_{H_j,H_i}\otimes \id_A) = (\tau_{H_j,H_i} \oplus \id_{\mathbb{V}^1})\circ f&\in \Hom_{\cal{C}}(H_{ji}\otimes A, \mathbb{V}^2),\\
		f&\circ(E_i\otimes \id_A) = (E \oplus \id_{\mathbb{V}^2})\circ f &\in \Hom_{\cal{C}}(H_{i+1,i,i+1}\otimes A, \mathbb{V}^3).
	\end{align*}
	Recall that the map $f^{(\ell)}:\mathbb{V}^\ell\otimes A\to \mathbb{V}^\ell$ is defined as follows:
	\[
		f^{(\ell)}|_{A\otimes H_{ws_i}} = 
		\begin{cases*}
			\mu_A & \text{ if $l=0$,}\\
			(\id\otimes \sigma)\circ (f^{(\ell-1)}\otimes \id)\circ T_i^{(\ell)} + (\id\otimes \rho)\circ f^{(\ell-1)} & \text{ if $r(ws_i) = r(w)s_i$.}\\
		\end{cases*}
	\]
	One uses this definition to expand both sides of the equalities above and check them term by term.
	By the same inductive argument as in~\cite[App.~B]{lai2024quantum}, for the length $2$ relation this boils down to~\eqref{eq:sig-rho}, and for the length $3$ relation to (\ref{diag:br1}--\ref{diag:br3}).
	This concludes the proof.
\end{proof}

\begin{rmk}\label{rmk:Zamolod}
	For a general Hecke-type algebra, one needs to check more relations in order to ensure the existence of a nice basis, see e.g. the main theorem of~\cite{Elias-diamond}.
	In our special case these checks are superfluous by a minor miracle.
	If one deforms the cubic relation $\mathbf{B}_3$ this ceases to be the case; however, see \cref{prop:rho-zero-nondef}.
\end{rmk}

\subsection{On the choice of diagrammatics}\label{ex:Frob-special}
Let us conclude the section by explaining why we opted to draw the identity maps $Z_i$ as six-valent vertices.
Consider the very degenerate setup, where $\bbk$ is a field, $\cal{C} = \Vec$, $F = \bbk$ the trivial algebra, and $H$ a vector space.
The relations \eqref{diag:wr} imply that that $\sigma = \id_H$, $\rho = 0$, and the relations \eqref{diag:wr2} are trivially satisfied.
Set $R=0$, and treat $S:H\otimes H\to H$ as an algebra product.
Then most of the relations (\ref{diag:TTT}--\ref{diag:br5}) become trivial, except for the following consequences of \eqref{diag:TTT} and \eqref{diag:br4}:
\[
\tikz[thick,xscale=.45,yscale=.45,font=\footnotesize,baseline=(current  bounding  box.center)]{
	\dmergedot{0}{0}{1}{1}
	\draw [dotted,line cap=round,rounded corners, color=OI5] (2,0) -- (2,1);
	\dmergedot{0.5}{1}{2}{2}
	\draw [dotted,line cap=round,rounded corners, color=OI5] (1.25,2) -- (1.25,2.5);
	\ntxt{2.5}{1}{$=$}
	\begin{scope}[xshift=3cm]
		\dmergedot{1}{0}{2}{1}
		\draw [dotted,line cap=round,rounded corners, color=OI5] (0,0) -- (0,1);
		\dmergedot{0}{1}{1.5}{2}
		\draw [dotted,line cap=round,rounded corners, color=OI5] (0.75,2) -- (0.75,2.5);
		\ntxt{2.3}{0.9}{,}
	\end{scope}
}\qquad
\tikz[thick,xscale=.45,yscale=.45,font=\footnotesize,baseline=(current  bounding  box.center)]{
	\dmergedot{0}{0}{1}{1}
	\morTi{0.5}{1}{3}{3}
	\draw [dotted,line cap=round,rounded corners, color=OI1] (1.75,0) -- (1.75,1);
	\draw [dotted,line cap=round,rounded corners, color=OI5] (3,0) -- (3,1);
	\ntxt{3.5}{1.5}{$=$}
	\begin{scope}[xshift=4cm]
		\morTi{1}{0}{3}{1.2}
		\draw [dotted,line cap=round,rounded corners, color=OI5] (0,0) -- (0,1.1);
		\morTi{0}{1.1}{2}{2.2}
		\draw [dotted,line cap=round,rounded corners, color=OI1] (3,1.1) -- (3,2.2);
		\draw [dotted,line cap=round,rounded corners, color=OI1] (0,2.2) -- (0,3);
		\draw [dotted,line cap=round,rounded corners, color=OI5] (1,2.2) -- (1,3);
		\dmergedot[OI1]{2}{2.2}{3}{2.8}
		\draw [dotted,line cap=round,rounded corners, color=OI1] (2.5,2.8) -- (2.5,3);
		\ntxt{3.5}{1.4}{,}
	\end{scope}
}\qquad
\tikz[thick,xscale=.45,yscale=.45,font=\footnotesize,baseline=(current  bounding  box.center)]{
	\dmergedot{2}{0}{3}{1}
	\morTi{0}{1}{2.5}{3}
	\draw [dotted,line cap=round,rounded corners, color=OI1] (1.25,0) -- (1.25,1);
	\draw [dotted,line cap=round,rounded corners, color=OI5] (0,0) -- (0,1);
	\ntxt{3.5}{1.5}{$=$}
	\begin{scope}[xshift=4cm]
		\morTi{0}{0}{2}{1.2}
		\draw [dotted,line cap=round,rounded corners, color=OI5] (3,0) -- (3,1.1);
		\morTi{1}{1.1}{3}{2.2}
		\draw [dotted,line cap=round,rounded corners, color=OI1] (0,1.1) -- (0,2.2);
		\draw [dotted,line cap=round,rounded corners, color=OI5] (2,2.2) -- (2,3);
		\draw [dotted,line cap=round,rounded corners, color=OI1] (3,2.2) -- (3,3);
		\dmergedot[OI1]{0}{2.2}{1}{2.8}
		\draw [dotted,line cap=round,rounded corners, color=OI1] (0.5,2.8) -- (0.5,3);
		\ntxt{3.5}{1.4}{.}
	\end{scope}
}
\]
For our choice $Z_i = \id_{H^{\otimes 3}}$, the last two relations are equivalent to $S\otimes \id_{H^{\otimes 2}} = \id_{H^{\otimes 2}} \otimes S$, so that either $H$ is one-dimensional or $S = 0$.
However, there is a more general setup where one can exhibit non-trivial examples.
Namely, let us allow the objects $H_i$ to be non-isomorphic, and let $Z_i$ be some map from $H_{i+1,i,i+1}$ to $H_{i,i+1,i}$.
Then the relations above are satisfied in the category of Soergel bimodules $\mathbb{S}\mathrm{Bim}$, see e.g.~\cite[\S 1.4.1-2]{EW16}.
Of course, \cref{thm:QWP-basis} has no reason to extend to this setting; the appropriate generalization would at the very least require to impose the $A_3$-Zamolodchikov relation on the elements $Z_i$, see also \cref{rmk:Zamolod}.
Still, the author believes that there should be an interesting theory of quantum wreath products internal to $\mathbb{S}\mathrm{Bim}$.

\medskip
\section{The invertible case}\label{sec:inv}

\subsection{Yoneda notations}\label{ssec:invertible}
It will be convenient for us to write down some relations from Yoneda point of view.
Given $X, M\in\cal{C}$, we refer to maps $X\to M$ as \textit{$X$-elements} of $M$.
Note that the set of $X$-elements $M_X\coloneqq \Hom_{\cal{C}}(X,M)$ is a $\bbk$-module. By a slight abuse of notation, we write $f(x) \coloneqq f\circ x$ for any $f\in \Hom_{\cal{C}}(M,N)$ and $x\in M_X$.
It follows from Yoneda's lemma that $f=g$ in $\Hom_{\cal{C}}(M,N)$ if and only if $f(x)=g(x)$ for all $x\in M_X$, $X\in \cal{C}$.
In view of this, when $X$ runs over all objects of $\cal{C}$ we will further abuse the notation and shorten ``$x\in M_X$, $X\in \cal{C}$'' to ``$x\in M$''.
Given a (unital) algebra $A\in\cal{C}$, we also write $xy\coloneqq \mu_A(x\otimes y)\in A_{X\otimes Y}$ for $x\in A_X$, $y\in A_{Y}$.
Finally, given $x\in F_X$, $y\in (F\otimes F)_Y$ and $1\leq i<j\leq n$, we write
\begin{equation}\label{eq:coords}
	x_i\coloneqq \iota_{\{i\}}(x)\in (F^{\otimes n})_X, \qquad y_{ij}\coloneqq \iota_{\{i,j\}}(y)\in (F^{\otimes n})_Y,
\end{equation}
see \cref{subs:qwp} for the notation.
For example, when $\cal{C} = \Vec$ and $X = \bbk$, the $\bbk$-elements of $M\in \cal{C}$ are precisely what one would ordinarily call elements or vectors.
When $\cal{C} = \sVec$, then $M_{\bbk^{1|0}}$ is the subspace of even elements of $M$, and $M_{\bbk^{0|1}}$ is the subspace of odd elements.
The casual reader is strongly encouraged to only ever keep these two examples in mind for the rest of the paper.

\subsection{$H$ is invertible}
Recall that an object $L\in \cal{C}$ is called \textit{invertible} if both the unit $\bb{1}\to L\otimes L^\vee$ and the counit $L^\vee\otimes L\to \bb{1}$ are isomorphisms, where $L^\vee$ is the dual of $L$.
By duality adjunction, we have $\End_{\cal{C}}(L) = \End_{\cal{C}}(\bb{1}) = \bbk$.
For an invertible $L$, we sometimes drop the symbol of tensor product from powers, and write $L^k = L^{\otimes k}$, $L^{-k} = (L^\vee)^{\otimes k}$ for $k>0$.

\begin{expl}
	The only invertible element in $\cal{C} = \Vec$ is the unit $\bb{1} = \bbk^1$.
	In super-vector spaces $\cal{C} = \sVec$, the invertible elements are the even one-dimensional space $\bbk^{1|0} = \bb{1}$ and the odd one-dimensional space $\bbk^{0|1}$. 
\end{expl}

From now on, unless otherwise stated, we assume that \textit{$H$ is invertible}.
We impose this condition for several reasons.
First, we simply do not have any examples with non-invertible $H$ in mind (besides \cref{rmk:super-KLR}).
Further, when $H$ is not invertible it might be more natural to consider a more general setup, as discussed in \cref{ex:Frob-special}. 
Finally, it is much easier to work with the algebraic relations (\ref{def:wr1}--\ref{def:br5}) below than with the diagrammatic relations (\ref{diag:wr}--\ref{diag:br5}).

\subsection{QWP algebras, redux}\label{ssec:inv-QWP}
In order for the algebra $F\wr \cal{H}_n$ to be well-defined, we need to impose the conditions (\ref{eq:P-unit},~\ref{eq:P-der}).
When $H$ is invertible, they admit a simpler form.
Namely, any $\sigma\in \Hom_{\cal{C}}(H\otimes F^{\otimes 2}, F^{\otimes 2}\otimes H)$ can be written as $\sigma = (\sigma'\otimes\id_H)\circ \tau_{H,F^{\otimes 2}}$, where $\sigma'\in \End_{\cal{C}}(F^{\otimes 2})$.
Let us write $\sigma$ instead of $\sigma'$ by abuse of notation.
Suppressing the braiding $\tau_{H,F^{\otimes 2}}$, as well as $H$, from the notation, the conditions rewrite as follows: 
\begin{align}
&\label{def:wr1}\tag{P1}
\sigma(1\otimes 1) = 1\otimes 1, 
&&  \rho(1 \otimes1)=0,
\\
&\label{def:wr2}\tag{P2}
\sigma(ab) = \sigma(a)\sigma(b),
&&
\rho(ab) =  \sigma(a)\rho(b) + \rho(a)b.
\end{align}
In other words, $\sigma$ is an algebra endomorphism of $F^{\otimes 2}$, and $\rho$ is an ($H$-twisted) left $\sigma$-derivation.

Recall that $A = F^{\otimes n}$.
Fixing $\sigma$ an endomorphism of $F^{\otimes 2}$, we have
\begin{align*}
	\sigma_i: A\to A, \quad \rho_i: H\otimes A\to A,\\
	S_i\in A_H, \quad
	R_i\in A_{H^2},\quad 
	Z_i = 1\in \bbk.
\end{align*}
Let us write $\sigma_{ij} = \sigma_{i}\sigma_{j}$, $\rho_{ij} = \rho_{i}\rho_{j}$ and so on for brevity.
By invertibility, the braiding $\tau_{H,H}$ is a scalar, and the braiding $\tau_{H,F}$ an endomorphism of $F$:
\[
	\tau_{H,H}\rightsquigarrow \epsilon\in \bbk^*, \qquad \tau_{H,F}\rightsquigarrow t\in \End_{\cal{C}}(F).
\]
It follows from \eqref{eq:sig-rho} that the maps $\sigma_i, \rho_i$ satisfy
\begin{equation}\label{eq:sig-rho-inv}
	\sigma_{ji} = \sigma_{ij}, \quad \rho_{ji} = \epsilon\rho_{ij},\quad \rho_j\sigma_i = \epsilon\sigma_i\rho_j, \quad\sigma_j\rho_i = \epsilon\rho_i\sigma_j; \qquad j\geq i+2.
\end{equation}
Furthermore, the defining relations of $F\wr\cal{H}_n$ can be rewritten as follows:
\begin{align}
	&H_ia = \sigma_i(a)H_i + \rho_i(a),&&a\in A,\tag{wreath}\\
	&H_i^2 = S_iH_i + R_i, &&1\leq i\leq n-1,\tag{quadratic}\\
	&H_i H_{i+1} H_i = H_{i+1} H_i H_{i+1},  \quad H_j H_i = \epsilon H_i H_j, &&1\leq i\leq n-2, \quad j\geq i+2.\tag{braid}
\end{align}
Here, the wreath relation comes from the definition of map $\mu_{1,0}$ in~\eqref{eq:mu-kj}, and the other relations from the defining ideal $\mathbf{I}_n \to A\otimes T(\Vn)$.

\subsection{Algebraic form of necessary relations}\label{ssec:alg-form}
Let us write $l_X$, resp. $r_X$ stand for left, resp. right multiplication by $X$.
Ignoring the dashed lines in the diagrams (\ref{diag:TTT}--\ref{diag:br5}), the conditions of \cref{thm:QWP-basis} translate into the relations below, where $\{i,j\}=\{1,2\}$: 
\begin{align}
	&\label{def:TTT}\tag{P3} SR = \rho(R) + \sigma(S)R,\qquad
	S^2+R = \rho(S) + \sigma(S)S + \sigma(R),
	\\
	&\label{def:qu1}\tag{P4} {l}_{S}\sigma = \sigma\rho+\rho\sigma+r_S\sigma^2,
	\qquad
	{l}_S\rho + {l}_R = \rho^2 + r_R\sigma^2,
	\\
	&\label{def:br1}\tag{P5}\sigma_{1} \sigma_2 \sigma_{1} = \sigma_2\sigma_{1}\sigma_2,
	\qquad
	\sigma_i \sigma_j \rho_i = \rho_j\sigma_i\sigma_j,
	\\
	&\label{def:br2}\tag{P6} \rho_i\sigma_j\rho_i  = \sigma_j\rho_i\rho_j + \rho_j\rho_i\sigma_j + r_{S_j}\sigma_j\rho_i\sigma_j,
	\\
	&\label{def:br3}\tag{P7} \rho_2\rho_1\rho_2 + r_{R_2} \sigma_2 \rho_1 \sigma_2 = \rho_1\rho_2\rho_1 + r_{R_1} \sigma_1 \rho_2 \sigma_1,
	\\
	&\label{def:br4}\tag{P8}S_i = \sigma_j\sigma_i(S_j),
	\qquad 
	R_i = \sigma_j\sigma_i(R_j),
	\qquad
	\rho_i\sigma_j(S_i) = 0 = \rho_{i}\sigma_j(R_i),
	\\
	&\label{def:br5}\tag{P9} \rho_i\rho_j(S_i) + \sigma_i\rho_j(S_i)S_i + \sigma_i\rho_j(R_i) = 0 = \sigma_i\rho_j(S_i)R_i + \rho_i\rho_j(R_i).
\end{align}

Note that these relations do not depend on braidings $\tau_{H,A}\in \End_{\cal{C}}(A)$, as these braidings do not occur in the diagrams (\ref{diag:TTT}--\ref{diag:br5}) either.

\begin{cor}\label{prop:LNX331}
	Assume that $H$ is invertible.
	The QWP algebra $F\wr \cal{H}_n$ has a PBW basis if and only if the relations (\ref{def:wr1}--\ref{def:br5}) hold.\qed
\end{cor}

\subsection{Localization}\label{ssec:loc-alg}
Most of the examples of QWPs we want to study require us to slightly extend the algebra $F^{\otimes n}$.
Consider the tensor algebra $T(L) = \bigoplus_{k\geq 0} L^{k}$, as well as its bi-infinite version $T_{\bbZ}(L)\coloneqq \bigoplus_{k\in \bbZ} L^{k}$.
We adapt the following definition from~\cite[Sec.~4.1]{coulem23}, which coincides with the usual definition of localization for $\cal{C} = \Vec,\sVec$.
\begin{defn}\label{def:loc}
	Let $L\in \cal{C}$ be invertible, $A\in \vec{\cal{C}}$ an (ind-)algebra object, and $x\in \Hom_{\cal{C}}(L,A)$.
	Assume that $\tau_{L,A}^2 = \id_{L\otimes A}$ and $x$ is central, that is $\mu\circ(x\otimes \id_A) = \mu\circ\tau_{L,A}\circ(x\otimes \id_A)$.
	The \textit{localization} of $A$ with respect to $x$ is defined by
	\[
		A[x^{-1}] \coloneqq A\otimes_{T(L)} T_{\bbZ}(L),
	\]
	where $T(L)\to A$ is the algebra map induced by $x$, and $T(L)\to T_{\bbZ}(L)$ the natural inclusion.

	Similarly, for any (countable) collection of central elements $x_i:L_i\to A$, $i\in I$ with $L_i$ invertible, satisfying $\tau_{L_i,L_j}^2 = \id$.
	Write $\mathscr{S} = \{x_i: i\in I\}$. We define $A[\mathscr{S}^{-1}]$ by iterating the definition above.
\end{defn}

\begin{rmk}
	In~\cite{coulem23}, localization is only defined for commutative algebras.
	Since we impose centrality on $x$, this is not an issue, as $A\otimes_{T(L)} T_{\bbZ}(L) = A\otimes_{Z(A)}(Z(A)_{T(L)} T_{\bbZ}(L))$.
\end{rmk}

Now let $A \coloneqq F^{\otimes n}$, $\sigma,\tau$ as before, and $x:L\to F^{\otimes 2}$ as in \cref{def:loc}.
Consider the set $\mathscr{S}_x$ of all maps one obtains from $x_i$, $1\leq i\leq n-1$ by applying sequences of $\sigma_j,\rho_j$, $1\leq j\leq n-1$.
Define 
\[
	A_\loc \coloneqq A[\mathscr{S}_x^{-1}].
\]
The maps $\sigma_i$, $\tau_i$ naturally extend to $A_\loc$.
We can therefore pick $S,R\in (F^{\otimes 2})_\loc$, and repeat \cref{def:QWP} verbatim.
In order to make the distinction between the two setups, we denote the resulting algebras by $(F\wr \cal{H}_n)_\loc$.
Note that, since the proof of \cref{thm:QWP-basis} treats $A$ as a black box, it also applies to these localized algebras.

\subsection{Polynomial representation and diagrammatics}\label{subs:polyrep}
In order to realize quantum wreath product $F\wr\cal{H}_n$ as a diagrammatic algebra, we need to construct its action on some object of $\vec{\cal{C}}$.
The most natural candidate for this is $F^{\otimes n}$, together with the requirement that $F^{\otimes n}\subset F\wr\cal{H}_n$ acts by left multiplication.
Let us fix an $H$-element $\alpha\in (F^{\otimes 2})_H = \Hom_{\cal{C}}(H,F^{\otimes 2})$, and define the following maps for $1\leq i\leq n-1$: 
\begin{equation}\label{eq:repn}
	h_i : H_i\otimes A \to A, \qquad h_i = \mu_A\circ (\id_A \otimes \alpha_i)\circ \sigma_i + \rho_i = 
	\tikz[thick,xscale=.52,yscale=.52,font=\small,baseline=(current  bounding  box.center)]{
		\sigact{0}{0}{1}{0.7}
		\draw (0,0.7) -- (0,1.3);
		\opbox{0.7}{0.7}{1.3}{1.3}{$\alpha_i$}
		\dmerge{0}{1.3}{1}{1.8}
		\draw (0.5,1.8) -- (0.5,2);
		\ntxt{1.7}{1.1}{$+$}
		\rhoact{2}{0}{3}{2}
		\ntxt{3.3}{1}{.}
	}
\end{equation}
In Yoneda notation, this becomes $h_i(x) = \sigma_i(x)\alpha_i + \rho_i(x)$.

Note that we have $\alpha_j\alpha_i = \epsilon\alpha_i\alpha_j$ for $j\geq i+2$.
Indeed, let us set $i=1$, $j=3$ without loss of generality.
Drawing $F\otimes F$ as one solid string, we have
	\[
		\tikz[thick,xscale=.55,yscale=.55,font=\small,baseline=(current  bounding  box.center)]{
		\ntxt{-1.4}{1.4}{$\alpha_3\alpha_1 = $}
		\draw [dotted,line cap=round,rounded corners] (1,0) -- (1,0.5);
		\opbox{0.8}{0.5}{1.2}{0.9}{$\alpha$}
		\draw (1,0.9) -- (1,1.3);
		\identity{0}{0}{2}
		\identity{3}{0}{2}
		\draw [dotted,line cap=round,rounded corners] (2,0) -- (2,0.5);
		\opbox{1.8}{0.5}{2.2}{0.9}{$\alpha$}
		\draw (2,0.9) -- (2,1.3);
		\overcros{1}{1.3}{2}{2}
		\dmerge{0}{2}{1}{2.7}
		\dmerge{2}{2}{3}{2.7}
		\draw (0.5,2.7) -- (0.5,3);
		\draw (2.5,2.7) -- (2.5,3);
		\ntxt{3.5}{1.5}{$=$}

		\begin{scope}[xshift=4cm]
			\draw [dotted,line cap=round,rounded corners] (1,0) -- (1,0.5);
			\opbox{0.8}{0.5}{1.2}{0.9}{$\alpha$}
			\draw (1,0.9) -- (1,1.5);
			\draw [dotted,line cap=round,rounded corners] (0,0) -- (0,0.5);
			\opbox{-0.2}{0.5}{0.2}{0.9}{$\alpha$}
			\draw (0,0.9) -- (0,1.5);
			\overcros{0}{1.5}{1}{3}
			\ntxt{1.5}{1.5}{$=$}
		\end{scope}

		\begin{scope}[xshift=6cm]
			\draw [dotted,line cap=round,rounded corners] (1,1.5) -- (1,2);
			\opbox{0.8}{2}{1.2}{2.4}{$\alpha$}
			\draw (1,2.4) -- (1,3);
			\draw [dotted,line cap=round,rounded corners] (0,1.5) -- (0,2);
			\opbox{-0.2}{2}{0.2}{2.4}{$\alpha$}
			\draw (0,2.4) -- (0,3);
			\overcrosdash{0}{0}{1}{1.5}
			\ntxt{1.5}{1.5}{$=$}
		\end{scope}

		\begin{scope}[xshift=8.6cm]
			\ntxt{-0.5}{1.5}{$\epsilon$}
			\draw [dotted,line cap=round,rounded corners] (1,0) -- (1,1.3);
			\opbox{0.8}{1.3}{1.2}{1.7}{$\alpha$}
			\draw (1,1.7) -- (1,3);
			\draw [dotted,line cap=round,rounded corners] (0,0) -- (0,1.3);
			\opbox{-0.2}{1.3}{0.2}{1.7}{$\alpha$}
			\draw (0,1.7) -- (0,3);
			\ntxt{1.6}{1.5}{$=$}
		\end{scope}

		\begin{scope}[xshift=11.3cm]
		\ntxt{-0.5}{1.5}{$\epsilon$}
		\draw [dotted,line cap=round,rounded corners] (0,0) -- (0,0.5);
		\opbox{-0.2}{0.5}{0.2}{0.9}{$\alpha$}
		\draw (0,0.9) -- (0,2);
		\identity{1}{0}{1.3}
		\identity{2}{0}{1.3}
		\draw [dotted,line cap=round,rounded corners] (3,0) -- (3,0.5);
		\opbox{2.8}{0.5}{3.2}{0.9}{$\alpha$}
		\draw (3,0.9) -- (3,2);
		\overcros{1}{1.3}{2}{2}
		\dmerge{0}{2}{1}{2.7}
		\dmerge{2}{2}{3}{2.7}
		\draw (0.5,2.7) -- (0.5,3);
		\draw (2.5,2.7) -- (2.5,3);
		\ntxt{4.5}{1.5}{$= \epsilon \alpha_1\alpha_3.$}
		\end{scope}
	}
	\]

\begin{prop}\label{prop:action}
	Assume that the QWP algebra $F\wr\cal{H}_n$ has PBW basis.
	The collection of maps $h_i$ together with the left action of $F^{\otimes n}$ defines an action of $F\wr\cal{H}_n$ on $F^{\otimes n}$ if and only if the following conditions hold:
	\begin{align}
		&\label{def:quad-al}\tag{R1} R - \rho(\alpha) = (\sigma(\alpha) - S)\alpha,
		\\
		&\label{def:cube-al}\tag{R2}\begin{aligned}
			(\sigma_{12}(\alpha_1)&\sigma_2(\alpha_1) +\sigma_1\rho_2(\alpha_1))\alpha_1 + \rho_1(\sigma_2(\alpha_1)\alpha_2) + \rho_{12}(\alpha_1)\\
			&= (\sigma_{21}(\alpha_2)\sigma_1(\alpha_2)+\sigma_2\rho_1(\alpha_2))\alpha_2 + \rho_2(\sigma_1(\alpha_2)\alpha_1) + \rho_{21}(\alpha_2).
		\end{aligned}
	\end{align}
\end{prop}
\begin{proof}
	Let $1\in F^{\otimes n}$ be the unit.
	One directly checks that the condition~\eqref{def:quad-al} is equivalent to $H^2(1) = SH(1) + R(1)$, and~\eqref{def:quad-al} is equivalent to $H_1H_2H_1(1) = H_2H_1H_2(1)$.
	Hence these two conditions are necessary.
	For sufficiency, we need to check the compatibility of formulas~\eqref{eq:repn} with the defining relations of $F\wr\cal{H}_n$, as described in \cref{ssec:inv-QWP}.
	For the wreath relation, we have
	\begin{align*}
		(Ha)b &= H(ab) 
		= \sigma(ab)\alpha + \rho(ab)
		\stackrel{\eqref{def:wr2}}{=} \sigma(a)\sigma(b)\alpha + \sigma(a)\rho(b) + \rho(a)b\\
		& = (\sigma(a)H + \rho(a))b,
	\end{align*}
	so that it holds automatically.
	For the quadratic relation, we have
	\begin{align*}
		(H^2)a &= H(Ha)
		= H(\sigma(a)\alpha + \rho(a))
		= \sigma^2(a)\sigma(\alpha)\alpha + \sigma\rho(a)\alpha + \rho(\sigma(a)\alpha) + \rho^2(a)\\
		&\stackrel{\eqref{def:wr2},~\eqref{def:qu1}}{=\joinrel=\joinrel=} \sigma^2(a)(\sigma(\alpha)\alpha + \rho(\alpha)- S\alpha - R) + S\sigma(a)\alpha + S\rho(a) + Ra \\
		&= (SH+ R)a + \sigma^2(a)(\sigma(\alpha)\alpha + \rho(\alpha)- S\alpha - R) \stackrel{\eqref{def:quad-al}}{=\joinrel=} (SH+ R)a.
	\end{align*}
	The braid relation $H_jH_i = \epsilon H_iH_j$ holds automatically:
	\begin{align*}
		(H_jH_i)a &= H_j(\sigma_i(a)\alpha_i + \rho_i(a))
		\stackrel{\eqref{def:wr2}}{=} \sigma_{ji}(a)\sigma_j(\alpha_i)\alpha_j + \rho_j\sigma_i(a)\alpha_i + \sigma_j\rho_i(a)\alpha_j+ \rho_{ji}(a)\\
		& \stackrel{\eqref{eq:sig-rho-inv}}{=} \sigma_{ij}(a)\alpha_i\alpha_j + \epsilon(\sigma_i\rho_j(a)\alpha_i + \sigma_j\rho_i\sigma_j(a)\alpha_j+\rho_{ij}(a))\\ 
		& = (\epsilon H_iH_j)a + \sigma_{ij}(a)(\alpha_i\alpha_j - \epsilon\alpha_j\alpha_i) = (\epsilon H_iH_j)a,
	\end{align*}
	Finally, for the remaining braid relation, using the relation \eqref{def:wr2} we obtain
	\begin{align*}
		(H_1&H_2H_1)a 
		= \sigma_{121}(a)(\sigma_{12}(\alpha_1)\sigma_1(\alpha_2)\alpha_1 + \sigma_1\rho_2(\alpha_1)\alpha_1 + \rho_1\sigma_2(\alpha_1)\alpha_2 + \sigma_{12}(\alpha_1)\rho_1(\alpha_2) + \rho_{12}(\alpha_1))\\
		&+ \rho_1\sigma_{21}(a)(\sigma_2(\alpha_1)\alpha_2+\rho_2(\alpha_1)) + \rho_2\sigma_{12}(a)(\sigma_1(\alpha_2)\alpha_1+\rho_1(\alpha_2))\\
		&+ (\sigma_1\rho_{21}(a) + \rho_{12}\sigma_1(a)+ \sigma_1\rho_2\sigma_1(a)\sigma_1(\alpha_1))\alpha_1 + \sigma_1\rho_2\sigma_1(a)\rho_1(\alpha_1) + \rho_1\sigma_2\rho_1(a)\alpha_2 + \rho_{121}(a),
	\end{align*}
	and an analogous expression for $H_2H_1H_2$ with $1$ and $2$ swapped.
	The first line in the two expressions coincides by the first equation of~\eqref{def:br1} and~\eqref{def:cube-al}.
	The second equation of~\eqref{def:br1} immediately implies that the second line coincides as well.
	Let us rewrite the third line:
	\begin{align*}
		(\sigma_1&\rho_{21}(a) + \rho_{12}\alpha_1(a)+ \sigma_1\rho_2\sigma_1(a)\sigma_1(\alpha_1))\alpha_1 + \sigma_1\rho_2\sigma_1(a)\rho_1(\alpha_1) + \rho_1\sigma_2\rho_1(a)\alpha_2 + \rho_{121}(a)\\
		& \stackrel{\eqref{def:br2}}{=} \sigma_1\rho_2\sigma_1(a)((\sigma_1(\alpha_1)-S_1)\alpha_1 + \rho_1(\alpha_1)) + (\rho_2\sigma_1\rho_2(a)\alpha_1+\rho_1\sigma_2\rho_1(a)\alpha_2) + \rho_{121}(a)\\
		& \stackrel{\eqref{def:quad-al}}{=} \rho_{121}(a) + \sigma_1\rho_2\sigma_1(a)R_1 + (\rho_2\sigma_1\rho_2(a)\alpha_1+\rho_1\sigma_2\rho_1(a)\alpha_2).
	\end{align*}
	Using the relation~\eqref{def:br3}, we see that this expression is invariant under swapping $1$ with $2$, and so we may conclude.
\end{proof}

\begin{rmk}
	If $\sigma = \tau$ and $\rho = 0$, conditions (\ref{def:quad-al},~\ref{def:cube-al}) assume a much simpler form:
	\[
		R = (\alpha_{21}-S)\alpha_{12}; \qquad \alpha_{j}\alpha_{i} = \epsilon\alpha_{i}\alpha_{j}, \quad j\geq i+2;\qquad \alpha_{23}\alpha_{13}\alpha_{12} = \alpha_{12}\alpha_{13}\alpha_{23}.
	\]
	Note the similarity of the first equation with the condition (C1) in~\cite{lai2025schurification}.
	We expect that as soon as $F\wr\cal{H}_n$ acts on $F^{\otimes n}$, we can construct its Schur version; see also \cref{rmk:Lai-upcoming}.
\end{rmk}

The action in \cref{prop:action} does not give rise to a map $F\wr\cal{H}_n\to \End_{\cal{C}}(F^{\otimes n})$.
On one hand, there is no natural map $A\to \End_{\cal{C}}(A)$ for a general category $\cal{C}$.
On the other hand, as $\cal{C}$ is monoidal, the exchange relation $(1\otimes f)\circ (g\otimes 1) = (g\otimes 1)\circ (1\otimes f)$ would imply $H_1H_3 = H_3H_1$, which does not hold when $\epsilon \neq 1$. 
If $\cal{C}$ is rigid, both problems can be resolved by working instead in the natural enrichment $\underline{\cal{C}}$ of $\cal{C}$ over itself, whose objects are the same as objects in $\cal{C}$, and $\Hom_{\underline{\cal{C}}}(M,N)\coloneqq M^*\otimes N$; see e.g.~\cite{MorPen19}.
Then if $F\in \cal{C}$, we have a natural map of algebras $F \to \End_{\underline{\cal{C}}}(F)=F^*\otimes F$.
The category $\underline{\cal{C}}$ is usually not monoidal, but rather has a braided interchange relation; see e.g.~\cite[§1.2]{BrundanEllis} for a discussion of this failure for super vector spaces.
This feature allows us to capture the relation $H_3H_1 = \epsilon H_1H_3$.

Assuming the relations (\ref{def:quad-al},~\ref{def:cube-al}) hold, we can introduce diagrammatic calculus for $F\wr \cal{H}_n$.
We draw copies of $F$ as solid lines, and remove all the other objects from diagrams by abuse of notation:
\[
\tikz[thick,xscale=.6,yscale=.6,font=\footnotesize,baseline=(current  bounding  box.center)]{
	\ntxt{-1.8}{0.8}{$x\in F_X\rightsquigarrow$}
	\ntxt{-2.2}{0.3}{$1\leq i\leq n$}

	\draw (0,0) -- (0,2);
	\ntxt{0.4}{1.2}{$\ldots$}
	\dmerge{1}{1.2}{2}{1.8}
	\draw [line width=3pt, white] (-0.5,0) .. controls (-0.5,0.6) and (1,0.2) .. (1,0.8);
	\draw [color=OI7] (-0.5,0) .. controls (-0.5,0.6) and (1,0.2) .. (1,0.8);
	{\color{OI7} \ntxt{-0.5}{-0.5}{$X$}}
	\draw (2,0) -- (2,1.2);
	\ntxt{2}{-0.5}{$i$}
	\opbox{0.8}{0.8}{1.2}{1.2}{$x$}
	\draw (1.5,1.8) -- (1.5,2);
	\ntxt{2.6}{1.2}{$\ldots$}
	\draw (3,0) -- (3,2);
	\ntxt{3.6}{0.8}{$\eqqcolon$}

	\begin{scope}[xshift=4.3cm]
		\draw (0,0) -- (0,2);
		\ntxt{0.75}{0.8}{$\ldots$}
		\draw (1.5,0.8) -- (1.5,0);
		\draw (1.5,1.2) -- (1.5,2);
		\ntxt{1.5}{-0.5}{$i$}
		\ntxt{2.25}{0.8}{$\ldots$}
		\opbox{1.3}{0.8}{1.7}{1.2}{$x$}
		\draw (3,0) -- (3,2);
		\ntxt{3.5}{0.8}{;}
	\end{scope}
}\qquad 
\tikz[thick,xscale=.6,yscale=.6,font=\footnotesize,baseline=(current  bounding  box.center)]{
	\begin{scope}[xshift=0cm]
		\ntxt{-1.6}{0.8}{$H_i\rightsquigarrow$}
		\draw (0,0) -- (0,2);
		\ntxt{-0.5}{-0.5}{$H$}
		\draw[line width=3pt, white] (-0.5,0) .. controls (-0.5,0.6) and (1,0.2) .. (1.1,0.7);
		\draw[dotted,line cap=round,rounded corners] (-0.5,0) .. controls (-0.5,0.6) and (1,0.2) .. (1.1,0.7);
		\ntxt{0.5}{0.8}{$\ldots$}
		\draw (1,0) -- (1.4,0.7);
		\draw (1.6,1.3) -- (2,2);
		\draw (2,0) -- (1.6,0.7);
		\draw (1.4,1.3) -- (1,2);
		\ntxt{1}{-0.5}{$i$}
		\ntxt{2}{-0.5}{$i+1$}
		\opbox{1}{0.7}{2}{1.3}{$h_i$}
		\ntxt{2.5}{0.8}{$\ldots$}
		\draw (3,0) -- (3,2);
		\ntxt{3.6}{0.8}{$\eqqcolon$}
	\end{scope}

	\begin{scope}[xshift=4.3cm]
		\draw (0,0) -- (0,2);
		\ntxt{0.5}{0.8}{$\ldots$}
		\draw (1,0) -- (2,2);
		\draw (2,0) -- (1,2);
		\ntxt{1}{-0.5}{$i$}
		\ntxt{2}{-0.5}{$i+1$}
		\ntxt{2.5}{0.8}{$\ldots$}
		\draw (3,0) -- (3,2);
		\ntxt{3.5}{0.8}{.}
	\end{scope}
}
\]
Then the defining relations of $F\wr \cal{H}_n$ can be drawn as follows:
\[
\tikz[thick,xscale=.6,yscale=.6,font=\footnotesize,baseline=(current  bounding  box.center)]{
	\begin{scope}[xshift=0cm]
		\draw (0,0) -- (0,0.3);
		\opbox{-0.2}{0.3}{0.2}{0.7}{$x$}
		\draw (0,0.7) -- (0,2);
		\ntxt{0.75}{0.9}{$\ldots$}
		\draw (1.5,0) -- (1.5,1);
		\draw (2.5,0) -- (2.5,1);
		\crosin{1.5}{1}{2.5}{2}
		\ntxt{3}{1}{$=$}
	\end{scope}

	\begin{scope}[xshift=4cm]
		\draw (0,0) -- (0,1.3);
		\opbox{-0.5}{1.3}{0.5}{1.8}{$t(x)$}
		\draw (0,1.8) -- (0,2);
		\ntxt{0.85}{0.9}{$\ldots$}
		\draw (1.5,1) -- (1.5,2);
		\draw (2.5,1) -- (2.5,2);
		\crosin{1.5}{0}{2.5}{1}
		\ntxt{2.9}{0.9}{,}
	\end{scope}
}\qquad
\tikz[thick,xscale=.6,yscale=.6,font=\footnotesize,baseline=(current  bounding  box.center)]{
	\begin{scope}[xshift=0cm]
		\crosin{0}{0}{1}{1}
		\draw (0,1) -- (0,2);
		\draw (1,1) -- (1,2);
		\ntxt{1.65}{0.9}{$\ldots$}
		\draw (2.5,0) -- (2.5,1.3);
		\opbox{2.3}{1.3}{2.7}{1.7}{$x$}
		\draw (2.5,1.7) -- (2.5,2);
		\ntxt{3}{1}{$=$}
	\end{scope}

	\begin{scope}[xshift=3.5cm]
		\crosin{0}{1}{1}{2}
		\draw (0,0) -- (0,1);
		\draw (1,0) -- (1,1);
		\ntxt{1.65}{0.9}{$\ldots$}
		\draw (2.5,0) -- (2.5,0.3);
		\opbox{2}{0.3}{3}{0.8}{$t(x)$}
		\draw (2.5,0.8) -- (2.5,2);
		\ntxt{3.2}{0.9}{,}
	\end{scope}
}\qquad
\tikz[thick,xscale=.6,yscale=.6,font=\footnotesize,baseline=(current  bounding  box.center)]{
	\begin{scope}[xshift=0cm]
		\draw (0,0) -- (0,0.3);
		\draw (1,0) -- (1,0.3);
		\opbox{-0.2}{0.3}{1.2}{0.7}{$x$}
		\draw (0,0.7) -- (0,1);
		\draw (1,0.7) -- (1,1);
		\crosin{0}{1}{1}{2}
		\ntxt{1.6}{1}{$=$}
	\end{scope}

	\begin{scope}[xshift=2.2cm]
		\crosin{0}{0}{1}{1}
		\draw (0,1) -- (0,1.3);
		\draw (1,1) -- (1,1.3);
		\opbox{-0.2}{1.3}{1.2}{1.8}{$\sigma(x)$}
		\draw (0,1.8) -- (0,2);
		\draw (1,1.8) -- (1,2);	
		\ntxt{1.5}{1}{$+$}
		\draw (2,0) -- (2,0.8);
		\draw (3,0) -- (3,0.8);
		\opbox{1.8}{0.8}{3.2}{1.3}{$\rho(x)$}
		\draw (2,1.3) -- (2,2);
		\draw (3,1.3) -- (3,2);
		\ntxt{3.4}{0.9}{,}
	\end{scope}
}
\]

\[
	\tikz[thick,xscale=.6,yscale=.6,font=\footnotesize,baseline=(current  bounding  box.center)]{
	\crosin{0}{0}{1}{1}
	\crosin{0}{1}{1}{2}
	\ntxt{1.5}{1}{$=$}
	\crosin{2}{0}{3}{1}
	\draw (2,1) -- (2,1.3);
	\draw (3,1) -- (3,1.3);
	\opbox{1.8}{1.3}{3.2}{1.8}{$S$}
	\draw (2,1.8) -- (2,2);
	\draw (3,1.8) -- (3,2);
	\ntxt{3.5}{1}{$+$}	
	\draw (4,0) -- (4,0.8);
	\draw (5,0) -- (5,0.8);
	\opbox{3.8}{0.8}{5.2}{1.3}{$R$}
	\draw (4,1.3) -- (4,2);
	\draw (5,1.3) -- (5,2);
	\ntxt{5.4}{0.9}{,}
}\qquad
\tikz[thick,xscale=.6,yscale=.6,font=\footnotesize,baseline=(current  bounding  box.center)]{
	\begin{scope}[xshift=0cm]
		\crosin{0}{0}{1}{1}
		\draw (0,1) -- (0,2);
		\draw (1,1) -- (1,2);
		\ntxt{1.75}{0.9}{$\ldots$}
		\crosin{2.5}{1}{3.5}{2}
		\draw (2.5,0) -- (2.5,1);
		\draw (3.5,0) -- (3.5,1);
		\ntxt{4}{1}{$=$}
	\end{scope}

	\begin{scope}[xshift=5cm]
		\ntxt{-0.5}{1}{$\epsilon$}
		\crosin{0}{1}{1}{2}
		\draw (0,0) -- (0,1);
		\draw (1,0) -- (1,1);
		\ntxt{1.75}{0.9}{$\ldots$}
		\crosin{2.5}{0}{3.5}{1}
		\draw (2.5,1) -- (2.5,2);
		\draw (3.5,1) -- (3.5,2);
		\ntxt{3.8}{0.9}{,}
	\end{scope}
}\qquad
\tikz[thick,xscale=.6,yscale=.6,font=\footnotesize,baseline=(current  bounding  box.center)]{
	\draw (0,0) -- (2,2);
	\draw (2,0) -- (0,2);
	\draw (1,0) .. controls (2.2,1) .. (1,2);
	\ntxt{2.5}{1}{$=$}
	\draw (3,0) -- (5,2);
	\draw (5,0) -- (3,2);
	\draw (4,0) .. controls (2.8,1) .. (4,2);
	\ntxt{5.3}{1}{.}
}
\]
Note the presence of the map $t = \tau_{H,F}$ in the first two relations; it occurs because of our definition of $\sigma_i$ in \cref{subs:qwp}.

\subsection{Deforming cubic relation}\label{ssec:def-cubic}
It is important for us to consider a slight generalization of QWP algebras, where we allow the degree $3$ braid relation to deform.
Let us pick additional elements $C,\overline{C}\in (F^{\otimes 3})_{H}$, $D,\overline{D}\in (F^{\otimes 3})_{H^2}$, $E\in (F^{\otimes 3})_{H^3}$.
As before, this gives rise to the elements
\[
	C_i,\overline{C}_i\in A_{H}, \quad D_i,\overline{D}_i\in A_{H^2}, \quad E_i\in A_{H^3},\qquad 1\leq i\leq n-2,
\]
where e.g. $E_i = \iota_{i,i+1,i+2}(E)$, and so on.

\begin{defn}
	Let $H\in\cal{C}$ be invertible, $F\in\cal{C}$ an algebra, $\sigma: F^{\otimes 2}\to F^{\otimes 2}$ an algebra endomorphism, $\rho: H\otimes F^{\otimes 2}\to F^{\otimes 2}$ a left $\sigma$-derivation.
	Denote $\qbar\coloneqq (\sigma,\rho,S,R,C,\overline{C},D,\overline{D},E)$.
	The \textit{deformed quantum wreath product} $F\wr\cal{H}_n = F\wr \cal{H}_n(\qbar)$ associated to $(F,\qbar)$ is the quotient of $F^{\otimes n}\otimes T(\Vn)$ by wreath, quadratic, and degree $2$ braid relation, as well as the following degree $3$ braid relation:
	\begin{align}
	H_{i+1} H_i H_{i+1} = H_i H_{i+1} H_i + C_i H_i H_{i+1} + \overline{C}_i H_{i+1} H_i + D_i H_i + \overline{D}_i H_{i+1} + E_i,  \quad 1\leq i\leq n-2.\tag{braid'}
\end{align}
\end{defn}

As in \cref{ssec:loc-alg}, we can also replace $A = F^{\otimes n}$ by its localization, and allow the new deformation parameters to live there.
The notion of having a PBW basis extends to deformed QWPs in an obvious way.
The necessary conditions~(\ref{def:br1}--\ref{def:br5}) will deform accordingly, but we do not need their explicit form.
However, even after taking this deformation into account, the analogue of \cref{thm:QWP-basis} is not true in general.

For the deformed QWPs, one can easily prove the analogue of \cref{prop:action}.
We will only need it in the case $C = \overline{C} = 0$, where the relation~\eqref{def:cube-al} needs to be replaced with
\begin{equation}
	\label{def:cube-al-def}\tag{R2'}
	\begin{aligned}
			(\sigma_{12}(\alpha_1)&\sigma_2(\alpha_1) +\sigma_1\rho_2(\alpha_1)+D)\alpha_1 + \rho_1(\sigma_2(\alpha_1)\alpha_2) + \rho_{12}(\alpha_1) + E\\
			&= (\sigma_{21}(\alpha_2)\sigma_1(\alpha_2)+\sigma_2\rho_1(\alpha_2)-\overline{D})\alpha_2 + \rho_2(\sigma_1(\alpha_2)\alpha_1) + \rho_{21}(\alpha_2).
	\end{aligned}
\end{equation}
Similarly, in the diagrammatics the last relation gets replaced with
\[
\tikz[thick,xscale=.6,yscale=.6,font=\footnotesize,baseline=(current  bounding  box.center)]{
	\draw (0,0) -- (2,2);
	\draw (2,0) -- (0,2);
	\draw (1,0) .. controls (2.2,1) .. (1,2);
	\ntxt{2.5}{1}{$=$}
	\draw (3,0) -- (5,2);
	\draw (5,0) -- (3,2);
	\draw (4,0) .. controls (2.8,1) .. (4,2);
	\ntxt{5.5}{1}{$+$}
	\crosin{6}{0}{7}{1}
	\draw (6,1) -- (6,1.2);
	\draw (7,1) -- (7,1.2);
	\draw (8,0) -- (8,1.2);
	\opbox{5.8}{1.2}{8.2}{1.8}{$D$}
	\draw (6,1.8) -- (6,2);
	\draw (7,1.8) -- (7,2);
	\draw (8,1.8) -- (8,2);
	\ntxt{8.5}{1}{$+$}
	\crosin{10}{0}{11}{1}
	\draw (9,0) -- (9,1.2);
	\draw (10,1) -- (10,1.2);
	\draw (11,1) -- (11,1.2);
	\opbox{8.8}{1.2}{11.2}{1.8}{$\overline{D}$}
	\draw (9,1.8) -- (9,2);
	\draw (10,1.8) -- (10,2);
	\draw (11,1.8) -- (11,2);
	\ntxt{11.5}{1}{$+$}
	\draw (12,0) -- (12,0.7);
	\draw (13,0) -- (13,0.7);
	\draw (14,0) -- (14,0.7);
	\opbox{11.8}{0.7}{14.2}{1.3}{$E$}
	\draw (12,1.3) -- (12,2);
	\draw (13,1.3) -- (13,2);
	\draw (14,1.3) -- (14,2);
	\ntxt{14.5}{1}{.}
}
\]

\medskip
\section{PBW bases versus Yang--Baxter equations}\label{sec:computations}
Let us remind the reader that $H$ is invertible.
We continue to use Yoneda-style vocabulary from \cref{ssec:invertible}.
For the ease of reference, let us separate the conditions \eqref{def:br1} and \eqref{def:br4} into two:
\begin{center}
\begin{minipage}{0.5\linewidth}
	\begin{align}
	&\sigma_1\sigma_2\sigma_1 = \sigma_2 \sigma_1 \sigma_2,                    \label{def:br11}\tag{P5.1}    \\
	&S_i = \sigma_j\sigma_i(S_j), \quad R_i = \sigma_j\sigma_i(R_j),    \label{def:br41}\tag{P8.1}
	\end{align}
\end{minipage}\begin{minipage}{0.5\linewidth}
    \begin{align}
    &\sigma_i \sigma_j \rho_i = \rho_j\sigma_i\sigma_j, \label{def:br12}\tag{P5.2}    \\
    &\rho_j\sigma_i(S_j) = 0 = \rho_j\sigma_i(R_j).         \label{def:br42}\tag{P8.2}
    \end{align}
    \end{minipage}
\end{center}

\subsection{Inner skew-derivations}
Let us restrict our attention to a broad class of automorphisms $\sigma$ and $\sigma$-derivations $\rho$ which is much easier to analyse.

\begin{defn}\label{def:tn-braiding}
	Let $\phi,\psi:F\to F$ be two commuting algebra automorphisms, and let $\tau = \tau_{F,F}$ be the braiding in $\cal{C}$.
	We call $(\phi\otimes \psi)\circ \tau: F^{\otimes 2}\to F^{\otimes 2}$ a \textit{twisted natural} braiding.
\end{defn}

Note that any twisted natural braiding $\sigma\coloneqq (\phi\otimes \psi)\circ \tau$ satisfies braid relations.
Indeed, using the naturality of $\tau$ we have:
\begin{align*}
	\sigma_1 \sigma_2 \sigma_1 & = (\phi\otimes \psi\otimes \id)\circ \tau_1\circ (\id\otimes\phi\otimes \psi)\circ \tau_2\circ (\phi\otimes \psi\otimes \id)\circ \tau_1\\
	& = (\phi^2\otimes \psi\phi\otimes \psi^2) \circ \tau_1\tau_2\tau_1 = (\phi^2\otimes \phi\psi\otimes \psi^2) \circ \tau_2\tau_1\tau_2
	= \sigma_2 \sigma_1 \sigma_2.
\end{align*}
Moreover, when $\cal{C}$ is symmetric monoidal, \cref{ex:symm-braiding} tells us that this is a braiding by algebra automorphisms.

\begin{defn}
	Let $A$ be an algebra, $\sigma: A\to A$ an algebra automorphism.
	We call the map
	\[
	\rho_r: H\otimes A\to A,\qquad \rho_r(a) = ra-\sigma(a)r
	\]
	the ($H$-twisted) \textit{inner $\sigma$-derivation} associated to $r\in A_H$.
	We say that $a\in A$ is \textit{$\sigma$-central} if $\rho_a\equiv 0$.
\end{defn}

In many cases of interest, being is $\sigma$-central is an extremely restrictive condition.

\begin{lem}\label{lem:no-central}
	Let $\cal{C}$ be a symmetric monoidal category, $\sigma$ a twisted natural braiding, and $\mathfrak{Br}_n\to \End(F^{\otimes n})$, $w = s_{i_1}\ldots s_{i_l}\mapsto \sigma_w\coloneqq\sigma_{i_1}\ldots \sigma_{i_l}$ the natural homomorphism.
	Assume that there exists a central ($\bbk$-)element $c\in F$, satisfying
	\[
		\phi(c) = \psi(c) = c, \qquad c_1\pm c_2\in F\otimes F \text{ are not zero divisors}.
	\]
	Then for any $w\in \mathfrak{S}_n\setminus \{\id\}$, there are no non-zero $\sigma_w$-central elements in $F^{\otimes n}$.
\end{lem}
\begin{proof}
	Assume $x\in F^{\otimes n}$ is $\sigma_w$-central.
	Then
	\[
		0 = xc_1 - \sigma_w(c_1)x = (c_1-\epsilon^{\ell(w)}c_{w(1)})x.
	\]
	Since $(c_1\pm c_2)$ is not a zero divisor, the same holds for $(c_1-\epsilon^{\ell(w)}c_{w(1)})$, and so $x$ must vanish.
\end{proof}

In particular, if $\bbk[x]\subset F$ is a central subalgebra such that $F$ is a free $\bbk[x]$-module, the lemma above applies for $c = x$.

\begin{defn}\label{defn:TIC}
	We say that the deformed QWP algebra $F\wr \cH_n$ is a \textit{TIC algebra}, if the following assumptions hold:
	\begin{enumerate}
		\item $\cal{C}$ is symmetric monoidal;
		\item $\sigma = (\phi\otimes \psi)\circ \tau$ is a twisted natural braiding;
		\item $\rho = \rho_r$ is an inner $\sigma$-derivation for $r\in (F^{\otimes 2})_H$, such that
		\begin{equation}\label{eq:r-psi-condition}
		(\phi\otimes \phi)(r) = (\psi\otimes \psi)(r) = r;
		\end{equation}
		\item There exists $c\in F$ satisfying the conditions of \cref{lem:no-central}.
\end{enumerate}
\end{defn}

From now on, \textit{we work exclusively with TIC algebras}, unless otherwise stated.
As explained above, they automatically satisfy the conditions \eqref{def:wr1}, \eqref{def:wr2}, \eqref{def:br11}.
Furthermore, the condition on $r$ implies that $\sigma_{12}(r_1) = r_2$, $\sigma_{21}(r_2) = r_1$:
\begin{align*}
	\sigma_{12}(r_1) &= (\phi^2\otimes \psi\otimes \psi)\circ \tau_{12}(r\otimes 1) = \phi^2(1)\otimes (\psi\otimes \psi)(r) = r_2;\\ 
	\sigma_{21}(r_2) &= (\phi\otimes \phi\otimes \psi^2)\circ \tau_{21}(1\otimes r) = (\phi\otimes \phi)(r)\otimes\psi^2(1) = r_1.
\end{align*}

\subsection{The case $\rho = 0$}
In order to apply \cref{thm:QWP-basis} to TIC algebras, let us first restrict our attention to the simplest case $\rho = 0$.

\begin{prop}\label{prop:rho-zero-nondef}
	For any TIC algebra with $\rho = 0$ which has a PBW basis, we have $C = \overline{C} = D = \overline{D} = E = 0$.
\end{prop}
\begin{proof}
	Consider the necessary condition $H_2(H_1(H_2x)) = H_{212}x$ for $x\in F^{\otimes 3}$:
	\begin{align*}
		H_2(H_1&(H_2x)) = \sigma_{212}(x)H_{212}\\
		 &= \sigma_{121}(x)H_{121} + \sigma_{121}(x)C H_{12} + \sigma_{121}(x)\overline{C} H_{21} + \sigma_{121}(x)DH_1 + \sigma_{121}(x)\overline{D} H_2 + \sigma_{121}(x)E,\\
		H_{212}x &= H_{121}x + C H_{12}x + \overline{C} H_{21}x + DH_1x + \overline{D} H_2x + Ex\\
		 &= \sigma_{121}(x)H_{121} + C \sigma_{12}(x)H_{12} + \overline{C} \sigma_{21}(x)H_{21} + D\sigma_{1}(x)H_1 + \overline{D} \sigma_{2}(x)H_2 + Ex.
	\end{align*} 
	As $\sigma$ is invertible, term-by-term comparison yields that $C$ is $\sigma_2$-central, $\overline{C}$ is $\sigma_1$-central, $D$ is $\sigma_{12}$-central, $\overline{D}$ is $\sigma_{21}$-central, and $E$ is $\sigma_{121}$-central.
	We conclude by \cref{lem:no-central}. 
\end{proof}

\begin{cor}\label{cor:PBW-for-rho-nil}
	A TIC algebra $F\wr \cal{H}_n$ with $\rho = 0$ has PBW basis if and only if: 
	\begin{gather*}
		S = C = \overline{C} = D = \overline{D} = E = 0,\quad R \text{ is $\sigma^2$-central},\\
		R = \sigma(R),\quad R_2 = \sigma_{12}(R_1), \quad R_1 = \sigma_{21}(R_2).
	\end{gather*}
\end{cor}
\begin{proof}
	Thanks to \cref{prop:rho-zero-nondef}, we can apply \cref{thm:QWP-basis}.
	Conditions \eqref{def:wr1}, \eqref{def:wr2}, \eqref{def:br11} are automatically satisfied for any TIC algebra.
	Conditions \eqref{def:br12}, \eqref{def:br2}, \eqref{def:br3}, \eqref{def:br42}, \eqref{def:br5} are vacuous when $\rho = 0$.
	Condition \eqref{def:qu1} is equivalent to the requirement that $S$ is $\sigma$-central and $R$ is $\sigma^2$-central. 
	By \cref{lem:no-central}, this implies that $S=0$.
	The condition \eqref{def:TTT} becomes $R = \sigma(R)$.
	Finally, the condition \eqref{def:br41} becomes $R_2 = \sigma_{12}(R_1)$, $R_1 = \sigma_{21}(R_2)$.
\end{proof}

\subsection{Renormalizing the generators}
Let us now consider a TIC algebra $F\wr \cal{H}_n$ with $\rho = \rho_r$, $r\in (F^{\otimes 2})_H$.
Consider the ($H$-)elements $\widetilde{H}_i \coloneqq H_i-r_i\in F\wr \cal{H}_n$.
Note that the wreath relation can be rewritten as follows:
\begin{align*}
	H_ia = \sigma_i(a)H_i + \rho_i(a) \quad\Leftrightarrow\quad (H_i-r_i)a = \sigma_i(a)(H_i-r_i)
	\quad\Leftrightarrow\quad \widetilde{H}_i a = \sigma_i(a)\widetilde{H}_i.
\end{align*}
Let us introduce the following elements:
\begin{align*}
	Y(r) & \coloneqq r_1r_2 + \sigma_1(r_2r_1) - r_2\sigma_1(r_2)\in (F^{\otimes 3})_{H^2},\\
	\overline{Y}(r) & \coloneqq r_2r_1 + \sigma_2(r_1r_2) - r_1\sigma_2(r_1) \in (F^{\otimes 3})_{H^2},\\
	B(r) &\coloneqq r_1\sigma_2(r_1)r_2 - r_2\sigma_1(r_2)r_1 + \sigma_2(r_1)R_2 - \sigma_1(r_2)R_1\in (F^{\otimes 3})_{H^3}.
\end{align*}
A direct computation yields the following analogues of quadratic and braid relations between $\widetilde{H}_i$'s:
\begin{gather*}
	\widetilde{H}_i^2 = S_i'\widetilde{H}_i + R_i',\qquad \widetilde{H}_i\widetilde{H}_j = \epsilon\widetilde{H}_j\widetilde{H}_i;\\
	\widetilde{H}_{i+1}\widetilde{H}_i\widetilde{H}_{i+1}
	= \widetilde{H}_i \widetilde{H}_{i+1} \widetilde{H}_i + C_i \widetilde{H}_i \widetilde{H}_{i+1} + \overline{C}_i \widetilde{H}_{i+1} \widetilde{H}_i + D'_i \widetilde{H}_i + \overline{D}' \widetilde{H}_{i+1} + E',
\end{gather*}
where the coefficients are given as follows:
\begin{align*}
	S' &\coloneqq S-r-\sigma(r), \qquad R' \coloneqq R+Sr-r^2,\\
	D' &\coloneqq D + Y(r) + \sigma_1(r_2)S'_1 + C\sigma_1(r_2) + \overline{C}r_2,\\
	\overline{D}' &\coloneqq \overline{D} - \overline{Y}(r) - \sigma_2(r_1)S'_2 + Cr_1 + \overline{C}\sigma_2(r_1),\\
	E' &= E - B(r) + Y(r)r_1 - \overline{Y}(r)r_2 + (\sigma_1(r_2)S_1' + \overline{C}r_2)r_1 - (\sigma_2(r_1)S_2' + Cr_1)r_1.
\end{align*}
Note that the first two coefficients in the cubic relation remain the same because of the equations $r_i = \sigma_{ji}(r_j)$, $\{i,j\} = \{1,2\}$, which hold in any TIC algebra.

\begin{lem}\label{prop:TIC-renorm}
	Let $\qbar = (\sigma,\rho_r,S,R,C,\overline{C},D,\overline{D},E)$ be the deformation datum of a TIC algebra, and write $\qbar' = (\sigma,0,S',R',C,\overline{C},D',\overline{D}',E')$ in the notation above.
	The assignment $H_i\mapsto \widetilde{H}_i + r_i$, $a\mapsto a$, $a\in A = F^{\otimes n}$ defines an isomorphism of algebras $F\wr \cal{H}_n(\qbar)\simeq F\wr \cal{H}_n(\qbar')$.
	In particular, $F\wr \cal{H}_n(\qbar)$ has a PBW basis if and only if $F\wr \cal{H}_n(\qbar')$ does.
\end{lem}
\begin{proof}
	By the computations above, the assignment produces a homomorphism of algebras; moreover, it has an obvious inverse $\widetilde{H}_i\mapsto H_i-r_i$.
	For the second claim, assume that $F\wr \cal{H}_n(\qbar) \simeq \bigoplus_{w\in \fkS_n} A\otimes H_w$.
	Under the map $H_i\mapsto \widetilde{H}_i + r_i$, each element $H_w$ gets sent to an element of the form $\widetilde{H}_w + \sum_{w'<w} a_{w'} \widetilde{H}_{w'}$.
	By upper-triangularity, this means if a non-trivial linear combination of elements $\widetilde{H}_w$ vanishes in $F\wr \cal{H}_n(\qbar')$, then some other non-trivial linear combination of elements $H_w$ must vanish in $F\wr \cal{H}_n(\qbar)$, which contradicts our assumption.
	Hence $F\wr \cal{H}_n(\qbar')$ has a PBW basis, and the inverse implication is proved analogously.
\end{proof}

Combining \cref{cor:PBW-for-rho-nil,prop:TIC-renorm}, we obtain the following elegant criterion for TIC algebras to have a PBW basis.

\begin{thm}\label{thm:deformed-TIC-PBW}
	A TIC algebra $F\wr \cal{H}_n(\qbar)$, $\qbar = (\sigma,\rho_r,S,R,C,\overline{C},D,\overline{D},E)$ has a PBW basis if and only if the following conditions hold:
	\begin{gather}
		\label{eq:SR-condition}S = r + \sigma(r), \qquad R = \gamma - \sigma(r)r;\\
		\label{eq:gamma-condition}\gamma\text{ is $\sigma^2$-central},\qquad \gamma = \sigma(\gamma), \qquad \gamma_i = \sigma_{ji}(\gamma_j), \quad\{i,j\} = \{1,2\},
	\end{gather}\vspace{-18pt}
	\begin{equation}
		\pushQED{\qed}
		\label{eq:CDE-condition}\qquad C = \overline{C} = 0,\qquad D = -Y(r),\qquad \overline{D} = \overline{Y}(r),\qquad E = B(r).\qedhere
		\popQED
	\end{equation}
\end{thm}

In other words, for any choice of $r$ (satisfying~\eqref{eq:r-psi-condition}) and $\gamma$ (satisfying~\eqref{eq:gamma-condition}) there exists the unique septuple of parameters $(S,R,C,\overline{C},D,\overline{D},E)$ such that $F\wr \cal{H}_n(\qbar)$ has a PBW basis.

\begin{cor}
	Let $F\wr \cal{H}_n(\qbar)$ be a TIC algebra with non-deformed cubic braid relation.
	Then it has a PBW basis if and only if $S = r + \sigma(r)$, $R = \gamma - \sigma(r)r$, $\gamma$ satisfies~\eqref{eq:gamma-condition}, and $Y(r) = \overline{Y}(r) = B(r) = 0$.
\end{cor}

Note that the conditions $Y(r)=0$, $\overline{Y}(r) = 0$ are the associative Yang--Baxter equations (see e.g.~\cite{Aguiar01}), and $B(r) = 0$ is a deformation of the quantum Yang--Baxter equation.

\begin{rmk}
	Without the condition (4) in the definition of TIC algebra, the statement of \cref{thm:deformed-TIC-PBW} is very much false.
	Indeed, let $F = \bbk$, $r=0$, and consider the finite Hecke algebra $\cal{H}(\fkS_2) = \bbk[H]/(H^2-(q-1)H+q)$.
	Then $S = q-1 \neq 0 = r +\sigma(r)$.
\end{rmk}

\subsection{Compatibility with localization}
In order to treat our examples of interest, we need to pass to the localized setting of \cref{ssec:loc-alg}.
We write $A = F^{\otimes n}$ as before.
Let $x: H\to F^{\otimes 2}$ be a central non-zero divisor, and pick an element $r \in F^{\otimes 2}[x^{-1}]$.
The inner derivation $\rho = \rho_r$ is a priori valued in $F^{\otimes 2}[x^{-1}]$, hence a choice of deformation parameters $S,R,\ldots$ only gives rise to a localized QWP algebra $(F\wr \cal{H}_n)_\loc$ as in \cref{ssec:loc-alg}.
However, let us assume that $\rho(\mathbf{f})\in F^{\otimes 2}$ for all $\mathbf{f}\in F^{\otimes 2}$, and all the deformation parameters belong to $A\subset A_\loc$.
In this case, the deformed QWP algebra $F\wr \cal{H}_n$ is well-defined.
If furthermore the conditions of \cref{defn:TIC} hold, we call $F\wr \cal{H}_n$ a TIC algebra by abuse of notation.

\begin{prop}\label{lem:localize}
	Let $x\in F^{\otimes 2}$ be a central non-zero divisor, and consider a TIC algebra $F\wr \cal{H}_n(\qbar)$ with $r \in F^{\otimes 2}[x^{-1}]$.
	It has a PBW basis if and only if the conditions~(\ref{eq:SR-condition}-\ref{eq:CDE-condition}) of \cref{thm:deformed-TIC-PBW} hold.
\end{prop}

\begin{proof}
	Note that the subalgebra of $(F\wr \cal{H}_n)_\loc$ generated by $A, H_i$ coincides with the image of the map $\varphi_n$ in \cref{lem:phi-surj}.
	Hence, we have the following commutative square:
	\[
		\begin{tikzcd}
			\bigoplus_{w\in \fkS_n} A\otimes H_w \ar[r,hook]\ar[d,two heads,"\phi_n"]& \bigoplus_{w\in \fkS_n} A_\loc\otimes H_w\ar[d,two heads,"(\phi_n)_\loc"]\\
			F\wr \cal{H}_n\ar[r,hook] & (F\wr \cal{H}_n)_\loc
		\end{tikzcd}
	\]
	The vertical arrows are surjective by \cref{lem:phi-surj}, while the horizontal arrows are injective as the localizing element $x$ is not a zero divisor.
	By \cref{thm:deformed-TIC-PBW}, the arrow on the right is an isomorphism if and only if the conditions~(\ref{eq:SR-condition}-\ref{eq:CDE-condition}) hold.

	As the square above is commutative, $(\phi_n)_\loc$ is injective only if $\phi_n$ is injective.
	In the other direction, let us assume that $(\phi_n)_\loc$ is not injective.
	Fix a non-zero $G\in \bigoplus_{w\in \fkS_n} A_\loc\otimes H_w$ such that $(\phi_n)_\loc(G) = 0$.
	Multiplying $G$ by elements of $\mathscr{S}_x$ on the left, and using the assumption that $x$ is not a zero divisor, we can further assume that $G\in \bigoplus_{w\in \fkS_n} A\otimes H_w$.
	Using the commutativity of the square again, we deduce that $\phi_n(G) = 0$.
	Hence $(\phi_n)_\loc$ is an isomorphism if and only if $\phi_n$ is, so we may conclude.
\end{proof}

\begin{cor}\label{cor:PBW-in-loc}
	Let $r\in F^{\otimes 2}[x^{-1}]$, such that $\rho_r$ preserves $F^{\otimes 2}$.
	Assume that $\gamma = R - \sigma(r)r$ satisfies~\eqref{eq:gamma-condition}.
	Then there exists a choice of parameters $S,C,\overline{C},D,\overline{D},E$ such that $F\wr \cal{H}_n$ has a PBW basis if and only if $r+\sigma(r)\in F^{\otimes 2}$, $Y(r),\overline{Y}(r),B(r)\in F^{\otimes 3}$.
	If these conditions hold, the correct choice is given by equations \eqref{eq:SR-condition}, \eqref{eq:CDE-condition}.\qed
\end{cor}

\subsection{Polynomial representation}
Replacing $H$ with $\widetilde{H}-r$ as in \cref{prop:TIC-renorm}, we immediately obtain the following analogue of \cref{prop:action}. 

\begin{prop}
	Let $F\wr \cal{H}_n$ be a TIC algebra, and write $\gamma = R + \sigma(r)r$.
	The maps \eqref{eq:repn} define an action of $F\wr \cal{H}_n$ on $F^{\otimes n}$ if and only if the element $\widetilde{\alpha} \coloneqq \alpha - r$ satisfies the following conditions:
	\begin{align}
		&\tag{R1'} \gamma = \sigma(\widetilde{\alpha})\widetilde{\alpha},\\
		&\label{def:QYBE-poly}\tag{QYB} \sigma_{12}(\widetilde{\alpha}_1)\sigma_2(\widetilde{\alpha}_1)\widetilde{\alpha}_1 = \sigma_{21}(\widetilde{\alpha}_2)\sigma_1(\widetilde{\alpha}_2)\widetilde{\alpha}_1.
	\end{align}
\end{prop}

In particular, \cref{thm:deformed-TIC-PBW} can be slightly rephrased to include the parameter $\alpha$.

\begin{cor}\label{cor:TIC-and-poly}
	Consider a TIC algebra $F\wr \cal{H}_n(\qbar)$.
	The following conditions are equivalent:
	\begin{enumerate}
		\item The algebra $F\wr \cal{H}_n(\qbar)$ has a PBW basis, and $\alpha\in F\otimes F$ defines a polynomial representation;
		\item Conditions (\ref{eq:SR-condition}--\ref{eq:CDE-condition}) hold with $\gamma \coloneqq \sigma(\widetilde{\alpha})\widetilde{\alpha}$, and $\widetilde{\alpha} = (\alpha -r)$ satisfies the quantum Yang--Baxter equation \eqref{def:QYBE-poly}. \qed
	\end{enumerate}
\end{cor}

\subsection{Relation to Lie bialgebras}\label{sec:Lie-bialg}
In this subsection, consider the case $\cal{C} = \Vec$, $\sigma = \tau$ the standard braiding.
Recall that an associative $\bbk$-algebra $F$ together with a $\bbk$-linear map $\delta:F\to F\otimes F$ is called a \textit{Lie bialgebra} if $\delta$ is a Lie cobracket, and
\begin{equation*}
	\delta([a,b]) = (\ad_a\otimes 1 + 1\otimes \ad_a)\delta(b) - (\ad_b\otimes 1 + 1\otimes \ad_b)\delta(a).
\end{equation*}

Given $r\in F\otimes F$, consider the map 
\[
	\delta_r\coloneqq \rho_r\circ (\eta_F \otimes \id_F): F\to F\otimes F; \qquad \delta_r(x) = r(1\otimes x)-(x\otimes 1)r.
\]

\begin{prop}\label{cor:Lie-bialg}
	Consider a TIC algebra $F\wr \cal{H}_n$ in $\Vec$ with $\sigma = \tau$, $S=0$ and $D = -\overline{D}$.
	If it has a PBW basis, then $(F,\mu-\mu\circ\tau, \delta_r)$ is a Lie bialgebra.
\end{prop}
\begin{proof}
	Thanks to \cref{thm:deformed-TIC-PBW}, we know that $S = r + \sigma(r)$, $D = -Y(r)$, $\overline{D} = \overline{Y}(r)$.
	Because of our assumptions, this implies
	\[
		r_{12} + r_{21} = 0,\qquad 
		Y(r) - \overline{Y}(r) = [r_{12},r_{13}] + [r_{12},r_{23}] + [r_{13},r_{23}] = 0.
	\]
	By~\cite[Prop.~2.1.2]{ChariPressley}, we conclude that $(F,\mu-\mu\circ\tau, \delta_r)$ is a Lie bialgebra.
\end{proof}

\begin{rmk}\label{rmk:Lai-upcoming}
	One can interpret \cref{cor:Lie-bialg} as saying that the (cocommutative) universal enveloping algebra $U(\mathfrak{gl}_1(F))$ admits a non-cocommutative deformation, supplied by $r$.
	A natural question then is whether a similar statement can be made about $\mathfrak{gl}_n(F)$, whether we have $S=D+\overline{D}=0$ or not.
	This is the subject of ongoing work with C.-J. Lai~\cite{LM26}, where we construct a certain deformation of $U(\mathfrak{gl}_n(F))$ as a limit of wreath Schur algebras associated to polynomial quantum wreath products $F[x]\wr \cal{H}_n$; see \cref{ssec:PQWP} for definitions.
\end{rmk}

\subsection{Useful lemmata}
Recall that trigonometric solutions of the graded version of associative Yang--Baxter equation $Y(r) = 0$ were studied in~\cite{Schedler03}.
One of the curious observations of \textit{loc.~cit.} is that such solutions automatically satisfy the \textit{quantum} Yang--Baxter equation as well, which, provided that $R$ is a scalar, is precisely $B(r) = 0$.
A similar observation holds in our situation.

\begin{lem}\label{lem:simplify-YYB}
	Let $F\wr \cal{H}_n$ be a TIC algebra with PBW basis.
	Assume that $S\in \bbk$.
	If $Y(r) = 0$, then we automatically have $\overline{Y}(r) = 0$.
	If moreover $\sigma(r)r$ is central and $\sigma_1(r_2) = \sigma_2(r_1)$, we have $B(r) = \sigma_1(r_2)(R_2 - R_1)$.
\end{lem}
\begin{proof}
	Applying $\sigma_2$ to $Y(r)$, we get:
	\begin{align*}
		0 &= \sigma_2(Y(r)) = \sigma_2(r_1)\sigma_2(r_2) + r_1(S-\sigma_2(r_1)) - (S-r_2)r_1\\
		&= \sigma_2(r_1)\sigma_2(r_2) - r_1\sigma_2(r_1) + r_2r_1 = \overline{Y}(r).
	\end{align*}
	For the second claim, it suffices to show that $r_1\sigma_2(r_1)r_2 = r_2\sigma_1(r_2)r_1$:
	\begin{align*}
		r_1\sigma_2(r_1)r_2 & \stackrel{\overline{Y}(r) = 0}{=\joinrel=\joinrel=} r_2r_1r_2 + S\sigma_2(r_1)r_2 - \sigma_2(r_1)r_2^2 \\
		&\stackrel{Y(r) = 0}{=\joinrel=\joinrel=} r_2^2\sigma_1(r_2) - Sr_2\sigma_1(r_2) + r_2\sigma_1(r_2)r_1 + S\sigma_2(r_1)r_2 - \sigma_2(r_1)r_2^2\\
		& = r_2\sigma_1(r_2)r_1 + (\sigma_2(r_2)r_2)\sigma_1(r_2) - \sigma_2(r_1)(\sigma_2(r_2)r_2) = r_2\sigma_1(r_2)r_1,
	\end{align*}
	and so we are done.
\end{proof}

In particular, \cref{lem:simplify-YYB} provides us with a rich source of TIC algebras with PBW basis.
Namely, let $r$ be a fractional skew-symmetric solution of the associative Yang--Baxter equation $Y(r) = 0$.
Set $R \in \bbk$, and assume that $\sigma = \tau$ the natural braiding.
Then \cref{cor:PBW-in-loc} applies, as long as $\sigma(r)r$ is central.

\medskip
\section{Examples of TIC algebras}\label{sec:expls}
While the class of TIC algebras is by definition smaller than the class of (deformed) QWP algebras, essentially all examples of QWP algebras of ``affine Hecke flavour'' that occur in nature are actually TIC algebras.
In this section, we list all such examples the author is aware of.
In each case, we specify $\sigma$, $r$ and $R$, and then deduce the rest of the relations from \cref{thm:deformed-TIC-PBW,cor:PBW-in-loc} by imposing that our algebra has a PBW basis.
This recovers various basis theorems in the literature.

From now on, we assume that the category $\cal{C}$ is either $\Vec$ or $\sVec$, and set $\sigma = \tau$ everywhere except for \cref{ssec:ns-Frob,ssec:Sergeev}.
The following claim immediately follows from assumption~\eqref{eq:r-psi-condition}:
\begin{lem}
	Assume that $\sigma(r)r$ is $\sigma$-invariant and $\sigma^2$-central.
	Then $\gamma$ satisfies~\eqref{eq:gamma-condition} if and only if $R$ does.\qed
\end{lem}
In all our examples, the element $\sigma(r)r$ is always $\sigma$-invariant and $\sigma^2$-central, so we will never mention the element $\gamma$ again.

\subsection{Type $A$ affine Hecke algebras}\label{ssec:affine-Hecke}
Let $\cal{C} = \Vec$, $F = \bbk[x]$, so that $A = \bbk[x_1,\ldots, x_n]$, and $\sigma_i$ acts by permutation.
Set $r = \frac{h + qx_1}{x_1-x_2}$, $h,q\in \bbk$, so that $\rho = (h + qx_1)\partial$, where $\partial$ is the Demazure operator 
\[
\partial(f) = \frac{f(x_1,x_2) - f(x_2,x_1)}{x_1-x_2}.
\]
We have $S = r + \sigma(r) = q$, and $D = 0$:
\begin{align*}
	Y(r) = \frac{(h+qx_1)(h+qx_2)((x_1-x_2)- (x_2-x_3) - (x_1-x_3))}{(x_1-x_2)(x_2-x_3)(x_1-x_3)} = 0.
\end{align*}
By \cref{lem:simplify-YYB}, we therefore have $\overline{D} = 0$, and
\begin{align*}
	E = r_{13}(R_2-R_1) = -\partial_{13}(R_1).
\end{align*}
In particular, $E = 0$ if and only if $R\in \bbk$.
In this case, we recover affine Hecke algebra for $h=0, q = \mathbf{q}-1, R=\mathbf{q}$ (after inverting $x$), and its degenerate version for $h = 1, q = 0, R = 1$.

\subsection{Formal Hecke algebras}
A certain generalization of affine Hecke algebras was proposed in~\cite{hoffnung2014formal}.
Recall that a (one-dimensional, commutative) formal group law over a commutative ring $\mathrm{R}$ is a power series $F(u,v)\in (u+v) +uv \mathrm{R}\llbracket u,v\rrbracket$, which satisfies the associativity condition $F(u,F(v,w)) = F(F(u,v),w)$.
For each formal group law, the authors of \textit{loc.~cit.} construct an associated \textit{formal Hecke algebra}, which specializes to affine/degenerate affine Hecke algebra for multiplicative/additive group law.
Instead of recalling the precise definition, let us show how to recover it from \cref{thm:deformed-TIC-PBW}.

Very briefly, one sets $\bbk = \mathrm{R}\llbracket h \rrbracket$, and $F = \bbk\llbracket x \rrbracket$.
For each $v\in \bbZ^n$, the formal group law $F$ gives rise to an element $x_v\in \mathrm{R}\llbracket x_1,\ldots,x_n \rrbracket\subset \bbk\llbracket x_1,\ldots,x_n \rrbracket$, as well as certain elements $\mu(x_\gamma), \mu(x_{-\gamma})\in \bbk$, see~\cite[Sec.~2-3]{hoffnung2014formal} for details.
Let us write $\beta = (1,-1)$, and distinguish the two cases, namely where $\kappa \coloneqq x_\beta - x_{-\beta}$ vanishes or not.
Set $\Theta = \mu(x_\gamma) - \mu(x_{-\gamma})$, $R = \mu(x_\gamma)\mu(x_{-\gamma})$, and
\[
	r\in \bbk\llbracket x_1,x_2\rrbracket, \qquad 
	r = \begin{cases}
		\Theta\frac{x_{-\beta}}{x_\beta + x_{-\beta}}, & \kappa\neq 0,\\
		\frac{2h}{x_{\beta}}, & \kappa= 0;
	\end{cases}
\]
As $\sigma(x_\beta) = x_{-\beta}$, we have $S=\Theta$ if $\kappa \neq 0$ and $S = 0$ otherwise.
Since the ring $F^{\otimes n}$ is commutative and $R \in \bbk$, we have $E = B(r) = r_{13}(R_1-R_2) = 0$.
Furthermore,
\[
	D = r_{13}(r_{23}-r_{21}) - r_{12}r_{23} = - \overline{D}. 
\]
This coefficient is precisely the linear term $\sigma_{12}$ in the cubic relation~\cite[Th.~8.14]{hoffnung2014formal}.
Hence the coefficients imposed by \cref{thm:deformed-TIC-PBW} recover the defining relations of formal affine Hecke algebra of type $A$ as listed in~\cite[Th.~8.14]{hoffnung2014formal}.

\subsection{Odd nilHecke}\label{ssec:odd-Hecke}
Let $\cal{C} = \sVec$, and set $F = \bbk[x]$, where $x$ is an \textit{odd} variable.
Then $A$ is the ring of odd polynomials
\[
	A = \bbk\langle x_1,\ldots, x_n\rangle / (x_ix_j + x_jx_i,\enspace i\neq j).
\]
Let $H\in \sVec$ be the odd invertible object, so that $\epsilon = -1$.
The action of $\sigma_i$ on $A$ is given by signed transposition, that is $\sigma_i(x_j) = -x_j$ if $j\not \in \{i,i+1\}$, and $\sigma_i(x_i) = x_{i+1}$.
Note that any polynomial in $x_i^2$, $1\leq i\leq n$ belongs to the center of $A$, and set
\[
	r = \frac{x_1-x_2}{x_1^2-x_2^2}.
\]
It is easy to check that the resulting $\rho_r$ is precisely the odd divided difference operator as defined in~\cite{ellis2014odd}.
We have $S = 0$, and $D = 0$ by a direct check.
By \cref{lem:simplify-YYB}, this implies $\overline{D} = 0$, and $E = r_{13}(R_2-R_1)$.
Similarly to \cref{ssec:affine-Hecke}, we have $E = 0$ if and only if $R\in \bbk$; for $R = 0$ we recover the odd nilHecke algebra, as considered in~\cite{ellis2014odd}.

\subsection{RAFH and PQWP}\label{ssec:PQWP}
Let $\cal{C} = \Vec$, $F'$ a finite-dimensional (unital) algebra, $F = F'[x]$, and pick a $\sigma$-invariant, $\sigma$-central element $\Delta\in F'\otimes F'$.
Note that when $\Delta$ is non-degenerate, its dual $\Delta^\vee\in (F'\otimes F')^\vee$ defines a symmetric Frobenius algebra structure on $F'$. 
Set
\begin{equation}\label{eq:r-mat-RAFH}
	r = \frac{\Delta(h+qx_1)}{x_1-x_2}, \qquad h,q\in \bbk.	
\end{equation}
It is easy to check that for $a\in F'\otimes F'$, $f\in \bbk[x_1,x_2]$ we have $\rho(a\otimes f) = \Delta a (h+qx_1)\partial(f)$.
We have $S = q\Delta$.
Note that $\Delta_{12}\Delta_{23} = \Delta_{23}\Delta_{13} = \Delta_{13}\Delta_{12}$ and $\Delta_{23}\Delta_{12} = \Delta_{12}\Delta_{13} = \Delta_{13}\Delta_{23}$ by the properties of $\Delta$.
Factoring out $\Delta_{12}\Delta_{23}$ and $\Delta_{23}\Delta_{12}$, the exact same computation as in \cref{ssec:affine-Hecke} shows that $D = \overline{D} = 0$.
Analogously, we have $E = -r_{13}(R_1-R_2)$.
If $R\in F'\otimes F'$, then $\Delta_{13}R_{23} = R_{32}\Delta_{13}$, and so we have 
\[
	r_{13}(R_1-R_2) = r_{13}R_1 - R_{21}r_{13} = r_{13}(R_1 - R_{21}) = 0.
\]
This recovers RAFH algebras as considered in~\cite{MS}, which for $q=0$, resp. $h=0$ specialize to affine wreath product algebras in~\cite{Sav_AWPA2018}, resp.~\cite{rosso2020quantum}.

One can slightly generalize this definition by keeping $R\in F'\otimes F'$, and setting 
\begin{equation}\label{eq:r-PQWP}
	r = \frac{\Delta^{00}+ \Delta^{10}x_1 + \Delta^{01}x_2 + \Delta^{11}x_1x_2}{x_1-x_2}, 
\end{equation}
where $\Delta^{ij}\in F'\otimes F'$ are $\sigma$-invariant, $\sigma$-central elements, possibly degenerate.
Here, $S = \Delta^{10}-\Delta^{01}$.
A direct (if lengthy) computation shows that $E = B(r) = 0$, whereas $\overline{D} = -D = \Delta^{00}_1\Delta^{11}_2 - \Delta^{01}_1\Delta^{10}_2$.
When the latter expression vanishes, we recover the basis theorem for PQWP algebras in~\cite[Prop.~3.10]{lai2025schurification}.

\subsection{Quiver Hecke algebras}\label{ssec:KLR}
Let $I$ be a finite set, $F' = \bigoplus_{i\in I} \bbk e_i$ a direct sum of one-dimensional algebras, where $e_i$, $i\in I$ are idempotents, and $F = F'[x]$.
Set
\[
	r = \sum_{i\in I} \frac{e_i\otimes e_i}{x_1-x_2},
\]
and pick a $\sigma$-invariant element $R\in F\otimes F$.
More explicitly, it means that $R = \sum_{i,j\in I} Q^{ij}(x_1,x_2) e_i\otimes e_j$, where $Q^{ij}\in \bbk[x_1,x_2]$, and $Q^{ij}(x_1,x_2) = Q^{ji}(x_2,x_1)$ for all $i,j\in I$.
We clearly have $S = Y(r) = \overline{Y}(r) = 0$.
Furthermore, as $R\in F\otimes F$ is symmetric, we have
\begin{equation}\label{eq:cst-term-KLR}
	B(r) = r_{13}(R_2-R_1) = \sum_{i,j\in I} \partial_{13} (Q^{ij})(e_i\otimes e_j\otimes e_i),
\end{equation}
where $\partial_{13}(f) = \frac{f(x_1,x_2) - f(x_3,x_2)}{x_1-x_3}$ is the Demazure operator.
In case where the elements $Q^{ij}$ satisfy some additional conditions (see e.g.~\cite[Sec.~2]{foldingKLR}), which guarantee the existence of a convenient grading, the defining relations of the TIC algebra $F\wr \cal{H}_n$ hence coincide with the defining relations of quiver Hecke algebra associated to $Q$; the set of vertices of the corresponding quiver iss $I$, and the arrows from $i$ to $j$ are read off the polynomial $Q^{ij}$.
In particular, \cref{thm:deformed-TIC-PBW} recovers the basis theorem for quiver Hecke algebras~\cite[Th.~2.5]{KL_DACQ2009}.

Building up on \cref{ssec:affine-Hecke,ssec:PQWP}, the definition above admits some natural generalizations.

\subsubsection{$K$-theoretic quiver Hecke algebras}
Let $F'$ be as in \cref{ssec:KLR}, $F = F'[x^{\pm 1}]$, $R\in (F\otimes F)^{\fkS_2}$, and
\[
	r = \sum_{i\in I} \frac{(e_i\otimes e_i)x_1}{x_1-x_2}.
\]
Similarly to before, we have $S=1$, $D = \overline{D} = 0$, and
\[
	E = B(r) = r_{13}(R_2-R_1) = -\sum_{i,j\in I} x_1\partial_{13} (Q^{ij})(e_i\otimes e_j\otimes e_i),
\]
For example, let us set $I = [1,n]$, $Q^{ij} = 0$ if $|i-j|\neq 1$, and $Q^{i,i+1}(x_1,x_2) = Q^{i+1,i}(x_2,x_1) = x_1x_2^{-1}-1$ for $1\leq i\leq n-1$.
In this case, all coefficients in $E = \sum_{i,j} E_{iji}(e_i\otimes e_j\otimes e_i)$ vanish unless $|i-j|= 1$, and 
\begin{align*}
	E_{i,i+1,i} &= x_1\partial_{13}(1-x_1x_2^{-1}) = -x_1x_2^{-1}\partial_{13}(x_1) = -x_1x_2^{-1};\\
	E_{i+1,i,i+1} &= x_1\partial_{13}(1-x_2x_1^{-1}) = -x_1x_2\partial_{13}(x_1^{-1}) = x_2x_3^{-1}.
\end{align*}
Hence, we recover precisely the relations of the $K$-theoretic quiver Hecke algebra for the quiver $A_n$, as computed in \cite[Sec.~6.2.2]{Zhou23master}.
In particular, we deduce from \cref{thm:deformed-TIC-PBW} that these algebras have a diagrammatic basis akin to the usual quiver Hecke algebras.
Note that for the $A_n$-quiver, the existence of basis also follows from the geometric realization of $K$-theoretic quiver Hecke algebra as the $K$-theory of quiver Steinberg; in the non-symmetric case, this argument will not apply.

\subsubsection{Frobenius quiver Hecke algebras}\label{ssec:Frob-KLR}
Let $\{F'_i\}_{i\in I}$ be a collection of finite-dimensional symmetric Frobenius algebras, and $\Delta^i\in F'_i\otimes F'_i$ the Frobenius elements.
Note that $\Delta \coloneqq \sum_{i\in I}\Delta^i$ is a Frobenius element for $F' \coloneqq \bigoplus_{i\in I} F'_i$.
Define $F = F'[x]$, and set
\[
	r = \frac{\Delta}{x_1-x_2} = \sum_{i\in I} \frac{\Delta'_i}{x_1-x_2} \in F\otimes F.
\]
Pick $R\in F\otimes F$ central and symmetric.
Note that being central is now a non-trivial condition; unraveling it, we see that for each $i,j\in I$ we have
\[
	Q^{ij} = \sum_{k,l} f^{kl}_{ij} x_1^kx_2^l,\qquad f^{kl}_{ij}\in Z(F'_i\otimes F'_j).
\]
Analogously to \cref{ssec:PQWP}, we have $S = D = \overline{D} = 0$, and $E = r_{13}(R_2-R_1)$.
Using the special form of $R$, we see that $E = \sum_{i,j} E_{iji}(e_i\otimes e_j\otimes e_i)$, where
\[
	E_{iji} = \frac{\Delta^i_{13}}{x_3-x_1}\left(\sum_{k,l} (f^{kl}_{ij}\otimes e_i) x_1^kx_2^l - (e_i\otimes f^{kl}_{ji})x_3^kx_2^l\right)
	= -\sum_{k,l} \Delta^i_{13}(f^{kl}_{ij}\otimes e_i) \partial_{13}(x_1^k)x_2^l \in F^{\otimes 3}.
\]
Hence \cref{lem:localize} continues to apply, and we get a basis theorem for the resulting TIC algebra.
Note that it can be equipped with a grading analogous to the usual quiver Hecke algebra, where $F'$ sits in degree $0$.
It would be very interesting to see whether these algebras categorify some variant of quantum groups.

\subsection{Super-versions of RAFH}\label{ssec:PQWP-super}
Let $\cal{C} = \sVec$, $F'$ an associative superalgebra, and $F = F'[x]$.
The invertible object $H$ is either even or odd.
Regardless, the element $R\in (F^{\otimes 2})_{H^2}$ is always even.
We can extend the setup of the previous section in two ways, depending on whether $x$ is an even or an odd variable.

\subsubsection{Even case}
Let $x$ be even, and let $\Delta^0, \Delta^1\in F'\otimes F'$ be two $\sigma$-invariant, $\sigma$-central elements of same parity as $H$.
Set
\[
	r = \frac{\Delta^0+\Delta^1x_1}{x_1-x_2}.
\]
Then $S = \Delta^1$, and the same argument as in \cref{ssec:PQWP} shows that $D = \overline{D} = E = 0$, provided that $R\in F'\otimes F'$ is $\sigma$-invariant and central.
We leave the analysis for the more general choice~\eqref{eq:r-PQWP} to the interested reader.

\subsubsection{Odd case}
Let $x$ be odd, and let $\Delta\in F'\otimes F'$ be a $\sigma$-invariant, $\sigma$-central elements of parity opposite to $H$.
Set 
\[
	r = \frac{\Delta(x_1-x_2)}{x_1^2-x_2^2}.
\]
Then $S = 0$, and once again $D = \overline{D} = E = 0$.
This gives a precise sense to the ``odd wreath product algebras'' in~\cite[Rem.~3.3]{Sav_AWPA2018} for symmetric Frobenius algebras.

\begin{rmk}\label{rmk:super-KLR}
Quiver Hecke superalgebras~\cite{superKLR} do not immediately fit our theory, because our crossings $H_i$ have fixed parity, while the parity of crossings in quiver Hecke superalgebras depends on which strands are being crossed.
This can be remedied this by dropping the invertibility condition on $H$, and setting $H = \bbk^{1|0}\oplus \bbk^{0|1}$.
Roughly speaking, we can define $\sigma$ by asking that for each pair of colours only the summand of $H$ of appropriate parity acts by swapping the tensor factors in $F'\otimes F'$, and the other summand acts by zero.
We then put polynomials of appropriate parity at each vertex. 
The resulting QWP is too big; however, it will contain a big annihilator, and after factoring it out one obtains the desired quiver Hecke superalgebra.
As QWP algebras in question are not TIC algebras in the sense of \cref{defn:TIC}, we leave the details to an interested reader.
\end{rmk}

\subsection{Twisted polynomial rings}\label{ssec:ns-Frob}
Let us consider one important case where $\sigma$ is different from the natural braiding $\tau$.
Set $\cal{C} = \Vec$ or $\sVec$, $F'\in \cal{C}$ an algebra, and fix an automorphism $\psi$ of $F'$ \textit{of finite order $\theta$}.
Define $F = \bbk[x]\ltimes F'$, for $x$ an even variable.
The semidirect product notation here means that $F = \bbk[x]\otimes F'$ as an element of $\vec{\cal{C}}$, and 
\[
	(x^k\otimes f_1)(x^l\otimes f_2) = x^{k+l}\otimes ((\psi')^{-l}(\tau(f_1))f_2); \qquad k,l\in \bbN, \quad f_1,f_2\in F'.
\]
Note that $\psi$ extends to an automorphism of $F$ by setting $\psi(x) = x$.
Since $\psi^\theta = \id$, the subalgebra $\bbk[x^\theta]\subset F$ belongs to the center.
Define $\tau^\psi\coloneqq (\id \otimes \psi)\circ \tau = \tau \circ (\psi \otimes \id)$, and pick a $\tau^\psi$-invariant, $\tau^\psi$-central element $\Delta\in F'\otimes F'$.
Unpacking these conditions, we have 
\[
	\psi_1(\Delta) = \tau(\Delta); \qquad \Delta (f_1\otimes f_2) = \tau(\psi(f_1)\otimes f_2)\Delta, \quad f_1,f_2\in F'.
\]
For example, if $F'$ is a non-symmetric Frobenius algebra and $\psi$ its Nakayama automorphism, we can set $\Delta$ to be the dual of the trace form.

\begin{rmk}
	One may wonder whether we could start with an \textit{odd} automorphism $\psi:\bbk^{0|1}\otimes F' \to F'$.
	In fact, we immediately run into trouble, as the two sides of the relation $fx = x\psi^{-1}(f)$ have different parity, regardless of the parity of $x$; compare this with~\cite[Rem.~3.3]{Sav_AWPA2018}.
	This also obliges us to restrict our attention to Frobenius algebras with even trace form. 
\end{rmk}

Let us define the following elements of $F\otimes F$, observing that $\Delta^{(1)} = \Delta$:
\begin{equation}\label{eq:Delta-k}
	\Delta^{(k)}\coloneqq \sum_{i=0}^{k-1} x_2^{k-1-i}\Delta x_1^i \in F\otimes F.
\end{equation}
\begin{lem}\label{lem:Delta-k-props}
	The elements $\Delta^{(k)}$ satisfy the following properties:
	\begin{enumerate}
		\item For any $f_1,f_2\in F'$, we have $\Delta^{(k)}(f_1\otimes f_2) = \tau(\psi^k(f_1)\otimes f_2)\Delta^{(k)}$;
		\item We have $\Delta^{(k)}x_1 - x_2\Delta^{(k)} = \Delta x^k_1 - x^k_2\Delta$ and $\Delta^{(k)}x_2 - x_1\Delta^{(k)} = x^k_2 \tau(\Delta) - \tau(\Delta) x^k_1$;
		\item For any $k\geq 1$, we have $(\psi\otimes \psi)(\Delta^{(k)}) = \Delta^{(k)}$.
	\end{enumerate}
	In particular, for $k=\theta$ we have 
	\begin{equation}\label{eq:Delta-theta}
	\begin{gathered}
		\Delta^{(\theta)}f = \tau(f)\Delta^{(\theta)}, \quad f\in F'\otimes F';\\
		\Delta^{(\theta)}x_1 - x_2\Delta^{(\theta)} = \Delta (x_1^\theta - x_2^\theta),\qquad \Delta^{(\theta)}x_2 - x_1\Delta^{(\theta)} = \tau(\Delta) (x_2^\theta - x_1^\theta).
	\end{gathered}
	\end{equation}
\end{lem}
\begin{proof}
	The proof is by computation. For the first claim, for any $0\leq i\leq k-1$ we have
	\begin{align*}
		x_2^{k-1-i}\Delta x_1^i(f_1\otimes f_2)
		&= x_2^{k-1-i}\Delta (\psi^i(f_1)\otimes f_2)x_1^i
		= x_2^{k-1-i} \tau(\psi^{i+1}(f_1)\otimes f_2)\Delta x_1^i\\
		&= \tau(\psi^{k}(f_1)\otimes f_2)x_2^{k-1-i} \Delta x_1^i,
	\end{align*}
	and so the desired formula follows after summation.
	For the second claim, we have
	\begin{align*}
		\Delta^{(k)}x_1 - x_2\Delta^{(k)} &= (\Delta^{(k+1)}-x^k_2\Delta) - (\Delta^{(k+1)}-\Delta x^k_1) 
		= \Delta x^k_1 - x^k_2\Delta,\\
		\Delta^{(k)}x_2 - x_1\Delta^{(k)} &= \psi_2^{-1}(\Delta^{(k+1)}-\Delta x^k_1) - \psi_1(\Delta^{(k+1)}-x^k_2 \Delta)
		= x^k_2 \tau(\Delta) - \tau(\Delta) x^k_1,
	\end{align*}
	where in the second chain of equalities we used $\psi_1(\Delta) = \tau(\Delta) = \psi_2^{-1}(\Delta)$.
	Finally, for the last claim we note that
	\[
		\psi_2\psi_1(\Delta) = \psi_2\tau(\Delta) = \tau\psi_1(\Delta) = \tau^2(\Delta) = \Delta,
	\]
	so that each term of $\Delta^{(k)}$ is $\psi\otimes \psi$-invariant.
\end{proof}

Unlike the case of RAFH algebras, we cannot set $h$ and $q$ in~\eqref{eq:r-mat-RAFH} to be simultaneously non-zero, and so need to consider two separate cases.

\subsubsection{Additive case}\label{subs:add-case}
Let us set $\sigma = \tau = \tau_{F,F}$, and $\rho = \rho_r$, where
\begin{equation}\label{eq:r-mat-add-ns}
	r = \frac{\Delta^{(\theta)}}{x_1^\theta- x_2^\theta}.
\end{equation}
It follows from \cref{lem:Delta-k-props} that $\rho(f) = 0$ for $f\in F'\otimes F'$, and $\rho(x_1) = \Delta$, $\rho(x_2) = -\tau(\Delta)$.
As a consequence, $\rho_r$ restricts to a skew derivation of $F\otimes F$; in fact, this is the exact same skew derivation as in~\cite[Sec.~4.1]{Sav_AWPA2018}.
We are hence in the situation of \cref{lem:localize}.
Note that 
\begin{align*}
	\tau(\Delta^{(\theta)}) &= \sum_{i=0}^{\theta-1} x_1^{\theta-1-i}\tau(\Delta) x_2^i
	= \sum_{i=0}^{\theta-1} x_2^i \psi_2^{-i}\psi_1^{\theta-1-i}\tau(\Delta) x_1^{\theta-1-i}
	= \sum_{i=0}^{\theta-1} x_2^i (\psi_1\psi_2)^{-i}(\Delta) x_1^{\theta-1-i}
	\\&= \Delta^{(\theta)}.
\end{align*}
In particular, $S = r+\tau(r) = \frac{\Delta^{(\theta)} - \tau(\Delta^{(\theta)})}{x_1^\theta- x_2^\theta} = 0$.
Furthermore, it follows from the formulas~\eqref{eq:Delta-theta} that $\tau(r)r$ is $\tau$-invariant and ($\tau^2$-)central, so that the same is true of $R$.
We further assume that $R\in F'\otimes F'$.
The proof of the following is by direct computation, see \cref{app:AYB-nonsymm}.
\begin{prop}\label{prop:YB-ns}
	For the element $r$ given by~\eqref{eq:r-mat-add-ns} we have $Y(r) = \overline{Y}(r) = B(r) = 0$.\qed
\end{prop}
As a consequence, we recover the basis theorem in~\cite[Th.~4.6]{Sav_AWPA2018} for $\cal{C} = \Vec$, and extends it to super-vector spaces.

\subsubsection{Multiplicative case}\label{subs:mult-case}
Now, let $\sigma$ be the twisted natural braiding $\tau^\psi = (\id \otimes \psi)\circ \tau$, and
\begin{equation}\label{eq:ns-r-mult}
	r = \frac{\Delta^{(\theta)}x_1}{x_1^\theta- x_2^\theta}.
\end{equation}
Since $\psi$ acts trivially on $x_i$'s, we have $\rho_r(x_1) = \Delta x_1$, $\rho_r(x_2) = -\tau(\Delta)x_1$ similarly to the additive case.
Furthermore, for any $f\in F'\otimes F'$ we have
\begin{equation}\label{eq:rho-restr-van}
	\rho_r(f) = \frac{\Delta^{(\theta)}x_1 f - \psi_2\tau(f)\Delta^{(\theta)}x_1}{x_1^\theta- x_2^\theta}
	= \frac{\tau\psi_1(f)\Delta^{(\theta)}x_1  - \psi_2\tau(f)\Delta^{(\theta)}x_1}{x_1^\theta- x_2^\theta} = 0.
\end{equation}
Hence $\rho_r$ restricts to a skew derivation of $F\otimes F$, and we can once again apply \cref{lem:localize}.
We have
\[
	S = r+\sigma(r) = \frac{\Delta^{(\theta)}x_1 - \psi_2\tau(\Delta^{(\theta)})x_2}{x_1^\theta- x_2^\theta}
	= \frac{\Delta^{(\theta)}x_1 - x_2\Delta^{(\theta)}}{x_1^\theta- x_2^\theta} = \Delta.
\]
Note that $\Delta^{(\theta)}x_1\Delta = \sigma(\Delta) \Delta^{(\theta)}x_1 = \Delta \Delta^{(\theta)}x_1$, so that $r\Delta = \Delta r$.
This implies that $\sigma(r)r$ is $\sigma$-invariant:
\[
	\sigma(\sigma(r)r) = \sigma((\Delta - r)r) = (\Delta - \sigma(r))\sigma(r) = r(\Delta - r) = (\Delta - r)r = \sigma(r)r. 
\]
It is also easy to check that $\sigma(r)r$ is $\sigma^2$-central.
Indeed, for $x_1, x_2$ this immediately follows from the $\tau^2$-centrality of $\sigma(rx_1^{-1})rx_1^{-1}$ in the additive case, while for $f\in F'\otimes F'$ we have
\[
	\sigma(r)r f = \frac{x_2(\Delta^{(\theta)})^2x_1f}{x_1^\theta- x_2^\theta}
	= \frac{\psi_2\tau^2\psi_1(f)x_2(\Delta^{(\theta)})^2x_1}{x_1^\theta- x_2^\theta}
	= \sigma^2(f) \sigma(r)r.
\]
As always, the element $R$ is hence $\sigma$-invariant and $\sigma^2$-central.
Assume that $R\in F'\otimes F'$.
Let us write $r = r'x_1$, where $r'$ is given by~\eqref{eq:r-mat-add-ns}.
Using \cref{prop:YB-ns}, we have
\begin{align*}
	Y(r) 
	&= r_1r_2 + \sigma_1(r_2r_1) - r_2\sigma_1(r_2)
	= r'_1\psi_1(r'_2)x_1x_2 + \sigma_1(r'_2)\psi_1\sigma_1(r'_1)x_1x_2 - r'_2\psi_2\sigma_1(r'_2)x_1x_2 \\
	&= (r'_1 r'_2 + \tau_1(r'_2) \tau_1\psi_1\psi_2(r'_1) - r'_2\tau_1(r'_2))x_1x_2 
	= (r'_1 r'_2 + \tau_1(r'_2r'_1) - r'_2\tau_1(r'_2))x_1x_2 = 0.
\end{align*}
The equality $\overline{Y}(r) = 0$ is proved analogously.
Using the special form of $R$, we also have
\begin{align*}
	r_2\sigma_1(r_2)r_1 &- r_1\sigma_2(r_1)r_2 
	= (r'_2\tau_1\psi_1^2(r'_2)\psi_1\psi_2(r'_1) - r'_1\tau_2\psi_1\psi_2(r'_1)\psi_1^2(r_2))x_1^2x_2\\
	&= (r'_2\tau_1(r'_2)r'_1 - r'_1\tau_2(r'_1)r'_2)x_1^2x_2 = 0,\\
	\sigma_2(r_1)R_2 
	&= \sigma_2(r'_1)x_1R_2  
	= R_2\sigma^{-1}_2(r'_1)x_1
	= \tau_2(\tau_2(R_2)r'_1)x_1
	= \tau_2(r'_1)\tau_{212}(R_2)x_1\\
	&= \tau_1(r'_2)\tau_1(R_1)x_1
	= \sigma_1(r'_2)x_1\sigma_1^{-1}(R_1)
	= \sigma_1(r_2)R_1,
\end{align*}
so that $B(r) = 0$ as well.

\begin{rmk}
	In~\cite[Rem.~2.8]{rosso2020quantum}, the authors state that the case of non-symmetric Frobenius algebra can be treated analogously to the symmetric one, i.e. with $\sigma = \tau$. 
	However, it is clear that the associative Yang--Baxter equation $Y(r) = 0$ is only satisfied for $\sigma = \tau^\psi$.
\end{rmk}

\subsection{Skew PQWP algebras}\label{ssec:skew-PQWP}
Let us again fix an algebra $F'\in \cal{C}$, an automorphism $\psi$ of $F'$ of finite order $\theta$, and $F\coloneqq \bbk[x]\ltimes F'$.
Set $\sigma = \tau$, and pick a $\tau$-invariant, $\tau$-central element $\Delta\in F'\otimes F'$.
As opposed to the setup of \cref{ssec:ns-Frob}, $\Delta$ and $\psi$ do not interact yet; let us \textit{impose} that $(\psi\otimes \psi)(\Delta) = \Delta$.
We again define $\Delta^{(k)}$ by the formula~\eqref{eq:Delta-k}, and $r$ by the formula~\eqref{eq:ns-r-mult}.
The same argument as in \cref{lem:Delta-k-props} shows that
\begin{gather*}
	\Delta^{(\theta)}x_1f = f\Delta^{(\theta)}x_1,\qquad \Delta^{(\theta)}x_1-x_2\Delta^{(\theta)} = \Delta(x_1^\theta - x_2^\theta),\\
	\Delta^{(\theta)}x_2-x_1\Delta^{(\theta)} = \psi_2^{-1}(\Delta^{(\theta+1)}-\Delta x^\theta_1) - \psi_1(\Delta^{(\theta+1)}-x^\theta_2 \Delta)
		= \psi_1(\Delta) (x^\theta_2 - x^\theta_1).
\end{gather*}
In particular, we have $\rho(f) = 0$ for $f\in F'$, $\rho(x_1) = \Delta x_1$, and $\rho(x_2) = -\psi_1(\Delta) x_1 = -x_1\Delta$. 
Hence, we are in the situation of \cref{lem:localize}.
It is clear that $S = r + \tau(r) = \Delta$.
As always, let $R\in F'\otimes F'$ be $\tau$-invariant and central.
A computation completely analogous to the one in \cref{prop:YB-ns} shows that $D = \overline{D} = E = 0$.

\begin{defn}\label{def:skew-PQWP}
	We call the resulting TIC algebra $F\wr \mathcal{H}_n$ a \textit{skew PQWP algebra}.
\end{defn}
In particular, let $F' = \bbk[t]/(t^m-1)$, $\psi(t) = \xi^2 t$, and $F = \bbk[x^{\pm 1}]\ltimes F'$, where $\xi$ is a primitive $\theta$-th root of unity, $m$ an even number, and $\theta$ divides $m$.
Set either $R = q$ or $R = q t^{m/2}\otimes t^{m/2}$, where $q\in \bbk$.
Then \cref{def:skew-PQWP} specializes to the skew PQWP algebras introduced in~\cite{buciumas2026quantum}, and so \cref{thm:deformed-TIC-PBW} specializes to~\cite[Th.~3.1.5]{buciumas2026quantum}.

\begin{rmk}
	In fact, the skew PQWP algebras considered in~\cite{buciumas2026quantum} can be realized as ordinary PQWP algebras.
	Indeed, in the setup of \cref{ssec:PQWP}, let $F' = \Mat_{m\times m}$ be the matrix algebra, $\Delta = \sum_{i,j = 1}^{m} E_{ij}$ the Frobenius element corresponding to the trace form, $F = F'[y^{\pm 1}]$.
	Let us assume $m = \theta$ for simplicity.
	One can easily verify that the map
	\begin{gather*}
		\Mat_{m\times m}[y^{\pm 1}] \to \bbk[x^{\pm 1}]\ltimes \bbk[t]/(t^m-1),\\
		y\mapsto x^m, \qquad E_{ii}\mapsto \sum_{j=1}^{m} (\xi^{2i}t)^j,\qquad 
		(E_{12}+E_{23}+\cdots+E_{m-1,m}+yE_{m1})\mapsto x
	\end{gather*}
	is an isomorphism of algebras.
	Moreover, the element $r'=\frac{\Delta y_1}{y_1-y_2}$ is sent precisely to the element $r$ defined by~\eqref{eq:ns-r-mult}.
	This establishes an isomorphism between PQWP and skew PQWP
	\[
		\Mat_{m\times m}[y^{\pm 1}] \wr \cal{H}_n \simto (\bbk[x^{\pm 1}]\ltimes \bbk[t]/(t^m-1)) \wr \cal{H}_n.
	\]
	When $m = k\theta$ is a multiple of $\theta$, one obtains a similar isomorphism by setting $F' = (\Mat_{\theta\times \theta})^{\oplus k}$.
\end{rmk}

\subsection{Affine Sergeev algebra}\label{ssec:Sergeev}
Finally, let us consider a slightly more involved example.
Let $\cal{C} = \sVec$, and $\mathsf{Cl} = \bbk[c]/(c^2+1)$ the Clifford algebra, where $c$ is odd.
Define $F = \mathsf{Cl}\ltimes \bbk[x^{\pm 1}]$, where $x$ is an even variable, and the semidirect product is taken with respect to the action $cx = x^{-1}c$.
Let $\psi:F\to F$ be the automorphism defined by $\psi(x) = x$, $\psi(c) = -c$.
We take $\sigma = \tau^\psi$, which permutes $x_i$'s and $c_i$'s without signs.
Note that the last condition of \cref{defn:TIC} is satisfied for the element $x+x^{-1}$.

Consider the element $m = \prod_{i,j = \pm 1}(x_1^i-x_2^j) = \prod_{i,j = \pm 1}(1-x_1^ix_2^j)$.
It is clearly not a zero divisor, and is central in $F^{\otimes 2}$.
Therefore in the localization $F^{\otimes 2}[m^{-1}]$ the (right) inverses of all factors in the products above will make sense.
Let $\varepsilon = q-q^{-1}$, where $q\in \bbk^*$ is a scalar.
Set $R = 1$ and
\begin{equation}\label{eq:rmat-Sergeev}
	r = \varepsilon(x_2(x_1-x_2)^{-1} + c_1c_2x_1x_2(1-x_1x_2)^{-1}).
\end{equation}
A quick computation shows that $\rho(x_1) = \varepsilon(c_1c_2x_1-x_2)$, $\rho(x_2) = \varepsilon(1+c_1c_2)x_2$, $\rho(c_1) = 0$, $\rho(c_2) = \varepsilon(c_1-c_2)$, so that \cref{lem:localize} applies.
Furthermore, we have $S = r + \sigma(r) = \varepsilon$, and a direct computation (see \cref{app:Sergeev}) shows that $Y(r) = \overline{Y}(r) = B(r) = 0$, so that the cubic braid relation is not deformed.
Comparing with formulas~\cite[(1.2, 2.2, 3.1-2)]{jones1999affine}, we see that the resulting TIC algebra is precisely the affine Sergeev algebra (also known as affine Hecke-Clifford superalgebra), and \cref{thm:deformed-TIC-PBW} recovers the basis theorem~\cite[Th.~2.2]{brundan2001hecke}.

\begin{rmk}
	The \textit{degenerate} affine Sergeev algebra already fits in the framework of \cref{subs:add-case}, see~\cite[Ex.~3.7]{Sav_AWPA2018}.
\end{rmk}

\section{Existence of polynomial representations}\label{sec:poly-exist}
In this section, we briefly address the question whether the polynomial representation $F^{\otimes n}$ of \cref{subs:polyrep} is defined in our cases of interest. 
For simplicity, we assume throughout that $\sigma = \tau$.

\subsection{Non-deformed case}
Let us first consider the case $D = \overline{D} = E = 0$.
This holds for the algebras considered in \cref{ssec:affine-Hecke,ssec:odd-Hecke,ssec:PQWP,ssec:PQWP-super,subs:add-case,ssec:skew-PQWP}.
In all these cases, we have a smaller subalgebra $F'\subset F$, such that $R\in F'\otimes F'$ and all skew derivations $\rho_i$ vanish on $(F')^{\otimes n}$.
Note that whenever $R = 0$, the equations (\ref{def:quad-al},~\ref{def:cube-al}) are satisfied by $\alpha = 0$.
Hence in this case, the polynomial representation can always be defined by \cref{prop:action}.

Let us next assume that $S = 0$.
Then the equation~\eqref{def:quad-al} becomes $R = \alpha_{21}\alpha_{12}$.
Assume that there exists a $\sigma$-invariant, $\sigma$-central element $\sqrt{R}\in F' \otimes F'$, such that $(\sqrt{R})^2 = R$; for example, when $R = 1$ we can take $\sqrt{R} = \pm 1$.
Then $\alpha = \sqrt{R}$ immediately satisfies~\eqref{def:quad-al}.
Moreover, as $\alpha$ is central and $\rho_i(\alpha_j) = 0$ for all $i,j$, the equation~\eqref{def:cube-al} also holds.

Now consider the case $S \neq 0$.
In all our examples, the element $S$ is then a $\sigma$-invariant, $\sigma$-central element of $F'\otimes F'$.
Let us assume for simplicity that $2$ is invertible in $\bbk$, and
\[
	\Delta \coloneqq S/2, \qquad R = 1,\qquad \alpha \in \bbk[\Delta]\subset F'\otimes F'.
\]
Then we can write down a formal solution of~\eqref{def:quad-al}:
\[
	\alpha = \Delta \pm \sqrt{1+\Delta^2} = \Delta \pm \left(1 + \frac{\Delta^2}{2} - \frac{\Delta^4}{8} + \cdots\right)
\]
This solution might or might not converge to an element of $\bbk[\Delta]$.
Let us assume it does, which is the case e.g. when $\Delta$ is nilpotent.
As $\alpha \in F'\otimes F'$, the equation~\eqref{def:cube-al} assumes the simpler form $\alpha_{23}\alpha_{13}\alpha_{12} = \alpha_{12}\alpha_{13}\alpha_{23}$.
Recall that we have $\Delta_{12}\Delta_{23} = \Delta_{23}\Delta_{13} = \Delta_{13}\Delta_{12}$ and $\Delta_{23}\Delta_{12} = \Delta_{12}\Delta_{13} = \Delta_{13}\Delta_{23}$.
Let us denote $C = \alpha - \Delta$; this element is central by definition.
Expanding the products, we obtain
\[
	\alpha_{23}\alpha_{13}\alpha_{12}-\alpha_{12}\alpha_{13}\alpha_{23} = (C_{12}-C_{13}+C_{23})(\Delta_{13}\Delta_{12}-\Delta_{12}\Delta_{13}).
\]
If $F'$ is commutative, then $\Delta_{13}\Delta_{12} = \Delta_{12}\Delta_{13}$ and so this expression trivially vanishes; in general, it does not. 

\begin{rmk}
	The difficulty arises specifically from our assumption $R = 1$.
	For example, if we have $R \coloneqq a(a-1)S^2$ for some $a\in \bbk$, then the element $\alpha\coloneqq aS$ clearly satisfies the conditions (\ref{def:quad-al},~\ref{def:cube-al}).
\end{rmk}

\subsection{Deformed case}\label{ssec:def-case-repn}
Let us now drop the assumption that the elements $D$, $\overline{D}$, $E$ vanish.
As explained in \cref{ssec:def-cubic}, we need to replace the condition~\eqref{def:cube-al} of \cref{prop:action} with~\eqref{def:cube-al-def}.
Observe that when $D = \overline{D} = 0$ and $E \neq 0$, neither $\alpha = 1$ nor $\alpha = 0$ satisfy the equation~\eqref{def:cube-al-def}.

Instead of trying to treat the general case, let us illustrate what can happen by looking at \cref{ssec:KLR}.
Consider quiver Hecke algebra for the $A_n$-quiver, that is when
\[
	R = \sum_{i,j = 1}^n Q^{ij}(x_1,x_2) e_i\otimes e_j, \qquad Q^{ij} = 0 \text{ for }|i-j|\neq 1,\qquad Q_{i,i+1} = x_1-x_2.
\]
We have $S = D = \overline{D} = 0$, and it follows from the equation~\eqref{eq:cst-term-KLR} that
\[
E = \sum_{i = 1}^{n-1} (e_i\otimes e_{i+1}\otimes e_i - e_{i+1}\otimes e_i\otimes e_{i+1}).
\]
Take $\alpha$ of the form $\alpha = \sum_{|i-j| = 1}\alpha^{ij}(e_i\otimes e_j)$,  $\alpha^{ij}\in \bbk[x_1,x_2]$.
The equation~\eqref{def:quad-al} rewrites as
\[
	\alpha^{i+1,i}(x_2,x_1)\alpha^{i,i+1}(x_1,x_2) = x_1-x_2, \qquad 1\leq i\leq n-1.
\]
Since we work in a commutative ring and the polynomials $\alpha^{ij}$ are at most linear, the equation~\eqref{def:cube-al-def} becomes
\[
	\sigma_1\rho_2(\alpha_1)\alpha_1 + E = \sigma_2\rho_1(\alpha_2)\alpha_2.
\]
A direct computation shows that this equation holds if we set $(\alpha^{i,i+1},\alpha^{i+1,i}) \in \{(x_1-x_2,1),(1,x_2-x_1)\}$ for each $i$.
This choice can be made independently for each $i$, so that we obtain $2^{n-1}$ versions of the polynomial representation.
This recovers the well-known fact that the polynomial representation of quiver Hecke algebra depends on the choice of the orientation of the quiver.

\begin{rmk}
	Recall the Frobenius quiver Hecke algebras in \cref{ssec:Frob-KLR}, and set $Q^{ij} = 0$ for $|i-j|\neq 1$, $Q^{i,i+1} = f_i(x_1-x_2)$, $f_i\in Z(F'_i\otimes F'_{i+1})$.
	Then a computation analogous to the above shows that, setting $(\alpha^{i,i+1},\alpha^{i+1,i}) = (f_i(x_1-x_2),1)$ or $(1,f_i(x_2-x_1))$, we obtain a polynomial representation.
\end{rmk}

\appendix
\section{Vanishing of deformation terms}
\subsection{Proof of \cref{prop:YB-ns}}\label{app:AYB-nonsymm}
Recall that $\sigma = \tau$ and $\tau(r) = -r$.
By \cref{lem:simplify-YYB}, $Y(r) = 0$ will imply $\overline{Y}(r) = 0$.
Bringing all terms to the common denominator, we need to check that the following expression vanishes:
\[
	\Delta_{12}^{(\theta)}\Delta_{23}^{(\theta)}(x_1^{\theta}-x_3^{\theta}) - \Delta_{13}^{(\theta)}\Delta_{12}^{(\theta)}(x_2^{\theta}-x_3^{\theta}) - \Delta_{23}^{(\theta)}\Delta_{13}^{(\theta)}(x_1^{\theta}-x_2^{\theta}).
\]
Moving all $x_i$'s to the right, and remembering that $(\psi\otimes\psi)(\Delta^{(\theta)}) = \Delta^{(\theta)}$, the terms rewrite as follows:
\begin{align*}
	\Delta_{12}^{(\theta)}\Delta_{23}^{(\theta)}(x_1^{\theta}-x_3^{\theta}) & = \sum_{i,j=0}^{\theta - 1}\psi_2^{-i-1}(\Delta_{12})\psi_2^{j-i}(\Delta_{23})x_1^ix_2^{\theta - 1 +j-i}x_3^{\theta - 1 -j}(x_1^{\theta}-x_3^{\theta}),\\
	\Delta_{13}^{(\theta)}\Delta_{12}^{(\theta)}(x_2^{\theta}-x_3^{\theta}) &=\sum_{i,j=0}^{\theta - 1}\psi_3^{-i-1}(\Delta_{13})\psi_2^{-j-i-1}(\Delta_{12})x_1^{i+j}x_2^{\theta - j-1}x_3^{\theta - 1 -i}(x_2^{\theta}-x_3^{\theta}),\\
	\Delta_{23}^{(\theta)}\Delta_{13}^{(\theta)}(x_1^{\theta}-x_2^{\theta}) &=\sum_{i,j=0}^{\theta - 1}\psi_3^{-i-1}(\Delta_{23})\psi_3^{-j-i-2}(\Delta_{13})x_1^{j}x_2^{i}x_3^{2\theta-i-j-2}(x_1^{\theta}-x_2^{\theta}),
\end{align*}
Note that every term has polynomial degree $3\theta-2$, and in each term at least one variable occurs with the degree at least $\theta$.
Let us check that the coefficient at the term $x_1^{\theta + i}x_2^jx_3^{2\theta - 2 -i-j}$ vanishes for all $i,j$.
Picking out the terms of $x_1$-degree $\theta+i$ for a fixed $i$, repeatedly using $(\psi\otimes \psi)(\Delta) = \Delta$, and setting $x_1=x_3=1$ to unburden the notation, we get:
\begin{align*}
	\sum_{j=\theta-1-i}^{2\theta-2-i}& \psi_1^{i+1}(\Delta_{12})\psi_3^{-j-1}(\Delta_{23}) x_2^j
	- \sum_{j=\theta}^{2\theta-2-i}  \psi_1^{i+1}\psi_3^{-j-1}(\Delta_{13})\psi_1^{i+1}(\Delta_{12})  x_2^j\\
	& + \sum_{j=0}^{\theta-2-i} \psi_3^{-i-j-2}(\Delta_{13})\psi_2^{-i-1}(\Delta_{12})  x_2^j
	- \sum_{j=0}^{\theta-1} \psi_3^{-j-1}(\Delta_{23})\psi_3^{-i-j-2}(\Delta_{13})  x_2^j\\
	& = \sum_{j=0}^{\theta-2-i} \psi_2^{-i-1}\psi_3^{-i-j-2}(\Delta_{13}\Delta_{12} - \Delta_{23}\Delta_{13})x_2^j +  \sum_{j=\theta-1-i}^{\theta-1}\psi_1^{i+1}\psi_3^{-j-1}(\Delta_{12}\Delta_{23} - \Delta_{23}\Delta_{13})x_2^j \\
	&+ \sum_{j = \theta}^{2\theta-2-i}\psi_1^{i+1}\psi_3^{-j-1}(\Delta_{12}\Delta_{23} - \Delta_{13}\Delta_{12})x_2^j.
\end{align*}
Since $\Delta$ is $\tau^\psi$-central and $\tau^\psi$-invariant, we have 
\[
\Delta_{12}\Delta_{23} = \psi_2(\Delta_{13})\Delta_{12} 
= \Delta_{13}\Delta_{12} = \psi_3(\Delta_{32})\Delta_{13}
= \Delta_{23}\Delta_{13}.
\]
It follows that the expression above vanishes.
The coefficients at the terms $x_1^ix_2^{\theta + j}x_3^{2\theta - 2 -i-j}$, $x_1^ix_2^jx_3^{\theta+(2\theta -2-i- j)}$ vanish by an analogous computation.
We conclude that $Y(r) = \overline{Y}(r) = 0$.

Finally, let us check that $B(r) = 0$.
We have 
\begin{align*}
	\Delta_{13}^{(\theta)}R_1 = R_{32}\Delta_{13}^{(\theta)} = \Delta_{13}^{(\theta)}R_{32} = \Delta_{13}^{(\theta)}R_2,
\end{align*}
so that $r_{13}R_1 = r_{13}R_2$, and we may conclude by \cref{lem:simplify-YYB}.\qed

\subsection{Computations for affine Sergeev algebra}\label{app:Sergeev}
Let us check that the $r$-matrix defined by~\eqref{eq:rmat-Sergeev} satisfies $Y(r) = \overline{Y}(r) = B(r) = 0$.
Without loss of generality, we can set $\varepsilon = 1$.
Observe the following simple identities in the localization $F^{\otimes 2}[m^{-1}]$, which follow from the relations in $F = \mathsf{Cl}\ltimes \bbk[x^{\pm 1}]$:
\[
	(x_1-x_2)^{-1}c_i = (-1)^ic_ix_i(x_1x_2-1)^{-1}, \quad (x_1x_2-1)^{-1}c_i = (-1)^ic_ix_i(x_1-x_2)^{-1}, \qquad i\in \{1,2\}.
\]
Using these equalities, let us compute $Y(r)$.
First, let us denote:
\begin{gather*}
	P = (x_1-x_2)^{-1}(x_1-x_3)^{-1}(x_2-x_3)^{-1},\\
	P' = (1-x_1x_2)^{-1}(1-x_1x_3)^{-1}(1-x_2x_3)^{-1},
	\quad P^{ij} = P'(1-x_ix_j)(x_i-x_j)^{-1}.
\end{gather*}
We write $c_{ij}\coloneqq c_ic_j$, $x_{ij}\coloneqq x_ix_j$  for brevity.
In these notations, we obtain
\begin{align*}
	r_1r_2 &= (x_2(x_1-x_2)^{-1} + c_1c_2x_1x_2(1-x_1x_2)^{-1})(x_3(x_2-x_3)^{-1} + c_2c_3x_2x_3(1-x_2x_3)^{-1})\\
	& = x_2x_3(x_1-x_3)P + c_{12}x_{123}(1-x_{13})P^{12} + c_{13}x_{13}x_2(1-x_{13})P^{13} - c_{23}x_{23}(x_1-x_3)P^{23};\\
	\sigma_1(r_2r_1)&=x_1x_3(x_3-x_2)P - c_{12}x_{123}x_1(x_2-x_3)P^{12} - c_{13}x_{13}x_1(1-x_{23})P^{13} + c_{23}x_{23}x_1(1-x_{23})P^{23};\\
	r_2\sigma_1(r_2)&= x_3^2(x_1-x_2)P + c_{12}x_{123}(1-x_{12})P^{12} - c_{13}x_{13}(x_1-x_2)P^{13} + c_{23}x_{23}x_3(1-x_{12})P^{23}.
\end{align*}
Comparing term by term the coefficients at $P$, $P^{ij}$, we conclude that $Y(r) = 0$.
By \cref{lem:simplify-YYB}, we also have $\overline{Y}(r) = 0$.
Note that $\sigma_1(r_2) = \sigma_2(r_1)$ and $\sigma(r)r$ is central; applying \cref{lem:simplify-YYB} once again, we obtain $B(r) = \sigma_1(r_2)(R_2-R_1) = 0$.\qed

\begingroup
\setstretch{0.9}
\bibliography{zot-bib}{}
\bibliographystyle{alphaabbr}
\endgroup

\end{document}